\documentclass[a4paper,reqno,11pt, oneside]{amsart}

\usepackage{amssymb}
\usepackage{amstext}
\usepackage{amsmath}
\usepackage{amscd}
\usepackage{amsthm}
\usepackage{amsfonts}
\usepackage{enumerate}
\usepackage{graphicx}
\usepackage{latexsym}
\usepackage{mathrsfs}
\usepackage{mathtools}
\usepackage{here}
\usepackage{ytableau}

\usepackage{setspace}
\renewcommand{\arraystretch}{1.12}

\usepackage{tikz}
\usepackage{tikz-cd}
\usetikzlibrary{arrows}
\usetikzlibrary{decorations.markings}
\usetikzlibrary{patterns,decorations.pathreplacing}

\makeatletter
\pgfdeclarepatternformonly[\LineSpace]{my north west lines}{\pgfqpoint{-1pt}{-1pt}}{\pgfqpoint{\LineSpace}{\LineSpace}}{\pgfqpoint{\LineSpace}{\LineSpace}}
{
    \pgfsetcolor{\tikz@pattern@color} 
    \pgfsetlinewidth{0.6pt}
    \pgfpathmoveto{\pgfqpoint{0pt}{\LineSpace}}
    \pgfpathlineto{\pgfqpoint{\LineSpace + 0.1pt}{- 0.1pt}}
    \pgfusepath{stroke}
}
\makeatother 
\newdimen\LineSpace
\tikzset{
    line space/.code={\LineSpace=#1},
    line space=8.5pt
}

\usepackage{hyperref}

\newtheorem{theorem}{Theorem}[section]

\newtheorem{lemma}[theorem]{Lemma}
\newtheorem{proposition}[theorem]{Proposition}
\newtheorem{definition-proposition}[theorem]{Definition-Proposition}

\theoremstyle{definition}
\newtheorem{definition}[theorem]{Definition}
\newtheorem{example}[theorem]{Example}
\newtheorem{observation}[theorem]{Observation}

\theoremstyle{remark}
\newtheorem{remark}[theorem]{Remark}

\makeatletter
  
  \@addtoreset{equation}{section}
\makeatother

\newcommand{\AAA}{\mathbb{A}}

\newcommand{\ZZ}{\mathbb{Z}}
\newcommand{\QQ}{\mathbb{Q}}
\newcommand{\RR}{\mathbb{R}}
\newcommand{\CC}{\mathbb{C}}
\newcommand{\PP}{\mathbb{P}}
\newcommand{\XX}{\mathbb{X}}

\newcommand{\bfe}{\mathbf{e}}
\newcommand{\bfs}{\mathbf{s}}

\newcommand{\bfw}{\mathbf{w}}

\newcommand{\calA}{\mathcal{A}}
\newcommand{\calC}{\mathcal{C}}

\newcommand{\calO}{\mathcal{O}}
\newcommand{\calP}{\mathcal{P}}
\newcommand{\calQ}{\mathcal{Q}}
\newcommand{\calT}{\mathcal{T}}
\newcommand{\calV}{\mathcal{V}}
\newcommand{\calX}{\mathcal{X}}

\newcommand{\rmH}{\mathrm{H}}
\newcommand{\fkS}{\mathfrak{S}}
\newcommand{\sfF}{\mathsf{F}}
\newcommand{\sfS}{\mathsf{S}}

\newcommand{\GL}{\operatorname{GL}}
\newcommand{\Hom}{\operatorname{Hom}}
\newcommand{\Ker}{\operatorname{Ker}}
\newcommand{\wt}{\operatorname{wt}}
\newcommand{\Proj}{\operatorname{Proj}}
\newcommand{\mut}{\mathsf{mut}}
\newcommand{\qmut}{\mathsf{Qmut}}
\newcommand{\width}{\mathsf{wd}}
\newcommand{\Newt}{\mathsf{Newt}}

\newcommand{\Gr}{\operatorname{Gr}}
\newcommand{\Trop}{\mathsf{Trop}}
\newcommand{\val}{\operatorname{val}}
\newcommand{\rec}{\operatorname{rec}}
\newcommand{\minlex}{\operatorname{min_{lex}}}
\newcommand{\slLie}{\mathfrak{sl}}
\newcommand{\Conv}{\operatorname{Conv}}
\newcommand{\Cone}{\operatorname{Cone}}

\definecolor{Green}{cmyk}{1,0,1,0.4}

\begin{document}

\title[Combinatorial mutations of Newton--Okounkov polytopes]
{Combinatorial mutations of Newton--Okounkov polytopes arising from plabic graphs}
\author[A Higashitani \and Y. Nakajima]{Akihiro Higashitani \and Yusuke Nakajima} 

\address[A Higashitani]{ Department of Pure and Applied Mathematics, Graduate School of Information Science and Technology, Osaka University, Suita, Osaka 565-0871, Japan}
\email{higashitani@ist.osaka-u.ac.jp}

\address[Y. Nakajima]{Department of Mathematics, Kyoto Sangyo University, Motoyama, Kamigamo, Kita-Ku, Kyoto, Japan, 603-8555}
\email{ynakaji@cc.kyoto-su.ac.jp}



\maketitle

\begin{abstract} 
It is known that the homogeneous coordinate ring of a Grassmannian has a cluster structure, which is induced from the combinatorial structure of a plabic graph. 
A plabic graph is a certain bipartite graph described on the disk, and there is a family of plabic graphs giving a cluster structure of the same Grassmannian. 
Such plabic graphs are related by the operation called square move which can be considered as the mutation in cluster theory. 
By using a plabic graph, we also obtain the Newton--Okounkov polytope which gives a toric degeneration of the Grassmannian. 
The purposes of this article is to survey these phenomena 
and observe the behavior of Newton--Okounkov polytopes under the operation called the combinatorial mutation of polytopes. 
In particular, we reinterpret some operations defined for Newton--Okounkov polytopes using the combinatorial mutation. 
\end{abstract}

\setcounter{tocdepth}{1}
\tableofcontents

\section{Introduction} 
\label{sec_intro}

\subsection{Background} 
A \emph{toric degeneration} of a complex projective variety $\XX$ is a flat family over the affine line $\AAA^1$ 
whose fiber $\XX_0$ over $0$ is a toric variety and the other fibers are isomorphic to $\XX$. 
Since it is a flat family, $\XX$ and $\XX_0$ share many important properties. 
Since the toric variety $\XX_0$ can be studied using the associated combinatorial data (e.g., a polytope or polyhedral fan), 
a toric degeneration gives effective viewpoints to understand the properties of $\XX$. 
A toric degeneration of a projective variety is given by using tropical geometry and Gr\"{o}bner theory. 
More precisely, suppose that the homogeneous coordinate ring of $\XX$ is of the form $A=\CC[x_1,\dots, x_n]/I$. 
Then the tropicalization $\calT(I)$ consisting of weight vectors whose corresponding initial ideals do not contain any monomial 
is a subfan of the Gr\"{o}bner fan of $I$, and weight vectors contained in the same cone $C$ of $\calT(I)$ determine the same initial ideal ${\rm in}_C(I)$. 
If the initial ideal ${\rm in}_C(I)$ associated to a maximal cone $C$ is a prime and binomial ideal, 
in which case $C$ is called a \emph{maximal prime cone}, 
then we have the toric degeneration whose central fiber is the toric variety $\Proj(\CC[x_1,\dots, x_n]/{\rm in}_C(I))$, see e.g., \cite[Theorem~15.17]{Eis}.
On the other hand, Kaveh and Manon showed in \cite{KM} that this toric degeneration can be obtained from a \emph{Newton--Okounkov body}. 
Namely, by using vectors in a maximal prime cone $C$, one can define a valuation $v$ on $A{\setminus}\{0\}$ such that 
the associated graded algebra ${\rm gr}_v(A)$ is isomorphic to $\CC[x_1,\dots, x_n]/{\rm in}_C(I)$. 
Also, we have the Newton--Okounkov body $\Delta(A,v)$ associated to $v$, which turns out to be a rational polytope 
and the toric variety associated to $\Delta(A,v)$ is $\Proj(\CC[x_1,\dots, x_n]/{\rm in}_C(I))$, see e.g., \cite{And,Kav}. 
Furthermore, it was shown in \cite[Theorem~1.11]{KMM} that any toric degeneration comes from a certain valuation in such a way.

\medskip

For the case where $\XX$ is the Grassmannian $\Gr(k,n)$ of $k$-planes in $\CC^n$, the toric degenerations of $\XX$ have been studied actively, and 
there are many studies related to the above constructions, see e.g., \cite{Bos,Bos2,BFFHL,CM,EH,MoSh,RW,SS} and the references therein. 
In this article, we will focus on Newton--Okounkov bodies obtained from \emph{plabic graphs}, 
which are rational polytopes, and hence we will call those \emph{Newton--Okounkov polytopes}. 
A plabic graph is a certain bicolored graph described on the disk (see Definition~\ref{def_plabic}). 
A special class of plabic graphs called ``type $\pi_{k,n}$" gives rise to a labeled seed 
realizing the homogeneous coordinate ring $\CC[\Gr(k,n)]$ of $\Gr(k,n)$ as the cluster algebra \cite{Sco} (see Subsection~\ref{subsec_prelim_plabic}). 
Furthermore, using the combinatorial structure on a plabic graph $G$ of type $\pi_{k,n}$, one defines a valuation on $\CC[\Gr(k,n)]{\setminus}\{0\}$ 
and this determines the Newton--Okounkov polytope $\Delta(G)$ associated to $G$ (see Subsection~\ref{subsec_prelim_NO}). 
This polytope was well studied in \cite{RW}, especially it gives a toric degeneration of the Grassmannian $\Gr(k,n)$ (after rescaling the polytopes), see Theorem~\ref{thm_toric_degeneration}. 
On the other hand, when one defines the cluster algebra, the operation called the \emph{mutation} is important. 
The \emph{labeled seed} associated to $G$ consists of a certain set of variables and the quiver obtained as the dual of $G$. 
In this situation, we can apply the mutation to the quiver associated to $G$, 
and such a mutation can be understood in terms of the moving operations on plabic graphs given in Definition~\ref{def_moves_plabic}. 
Applying the moving operations corresponding to the mutation, we can obtain another plabic graph $G^\prime$ which is also type $\pi_{k,n}$, 
and hence the Newton--Okounkov polytope $\Delta(G^\prime)$ also gives a toric degeneration of $\Gr(k,n)$. 
Thus, it is interesting to understand the relationship between these Newton--Okounkov polytopes. 
In \cite{RW}, Rietsch and Williams showed that the Newton--Okounkov polytopes $\Delta(G)$ and $\Delta(G^\prime)$ 
are related by the \emph{tropicalized cluster mutation} (see Theorem~\ref{thm_trop_map_NO} for more details). 
In addition, there are other results relating polytopes giving toric degenerations of $\Gr(k,n)$, e.g. \cite{Bos,BMN,Cas,CHM,EH,RW}. 
For example, in \cite{EH}, Escobar and Harada investigated a \emph{wall-crossing map}, which is a piecewise-linear map relating 
two Newton--Okounkov polytopes associated to adjacent maximal prime cones in the tropicalization $\calT(I)$. 
Ilten explained such a wall-crossing from a viewpoint of combinatorial mutation of a polytope in \cite[Appendix]{EH}. 
Also, in \cite{CHM}, the authors showed that ``block diagonal matching polytopes", which give toric degenerations of $\Gr(k,n)$, 
are related by the combinatorial mutations and unimodular transformations. 
The \emph{combinatorial mutation} of a polytope is an operation transforming a given polytope into another one 
while keeping essential properties (see Section~\ref{sec_prelim_mutation} for more details). 
This operation was introduced in \cite{ACGK} in the context of the mirror symmetry of Fano manifolds. 
Two polytopes related by a combinatorial mutation share some properties, e.g., their Ehrhart series are the same, and hence so are the volumes, 
the number of lattice points etc.
Thus, it would be worthy of observing a relationship among polytopes giving toric degenerations of $\Gr(k, n)$ from the viewpoint of combinatorial mutations. 

\subsection{Summary and contents} 
\label{subsec_intro_summary}

The purpose of this article is to reinterpret several operations defined for Newton--Okounkov polytopes arising from plabic graphs 
from the viewpoint of combinatorial mutations of polytopes. 

For this purpose, in Section~\ref{sec_prelim_plabic}, we survey Newton--Okounkov polytopes arising from plabic graphs referring to \cite{BFFHL,MaSc,Pos,RW,Sco} etc. 
Precisely, we recall Grassmannians in Subsection~\ref{subsec_Grassmannian}. 
In Subsection~\ref{subsec_prelim_plabic}, we introduce plabic graphs and the move operations on plabic graphs (see Definition~\ref{def_moves_plabic}). 
Then we introduce plabic graphs of type $\pi_{k,n}$ which are related by the move operations (see Proposition~\ref{prop_move_equiv}), 
and each face of a plabic graph of type $\pi_{k,n}$ is labeled by a $k$-subset or a Young diagram. 
In particular, the \emph{rectangle plabic graph} $G_{k,n}^{\rm rec}$ of type $\pi_{k,n}$, which is a plabic graph whose faces are labeled by Young diagrams of rectangle shape, is an important object throughout this article. 
We then introduce the quiver defined as the dual of a plabic graph, 
and give a brief explanation about the construction of the  homogeneous coordinate ring of the Grassmannian $\Gr(k,n)$ as the cluster algebra 
with the labeled seed arising from a plabic graph $G$ of type $\pi_{k,n}$ \cite{Sco}. 
Then, in Subsection~\ref{subsec_prelim_NO}, we explain how to construct the Newton--Okounkov polytope from a plabic graph $G$ of type $\pi_{k,n}$. 
We first define a \emph{perfect orientation} of $G$ and 
consider a collection of directed paths on $G$ called \emph{$J$-flow} where $J$ is a $k$-subset in $[n]=\{1,\dots,n\}$. 
Using a special choice of perfect orientations, we can define the valuation $\val_G$ on $\CC(\Gr(k,n)){\setminus}\{0\}$. 
Then, using the valuation $\val_G$, we define the Newton--Okounkov polytope $\Delta_G$ as in Definition~\ref{def_NObody} and 
this gives a toric degeneration of $\Gr(k,n)$ (see Theorem~\ref{thm_toric_degeneration}). 

In Section~\ref{sec_prelim_mutation}, we survey the combinatorial mutations of polytopes, referring to \cite{ACGK,Hig}. 
We will review two kinds of combinatorial mutations. 
One of those is defined for a lattice polytope $P$ that contains the origin ${\bf 0}$, and another one is defined for the polar dual $P^*$ of $P$. 
Although we mainly use the mutation on the dual side, they are closely associated, thus we introduce both of them. 
We note that the dual one is defined using a piecewise-linear map as in Definition~\ref{def_mutation_polygonM}, 
and this mutation preserves some properties of a polytope, e.g., the volume, the number of lattice points. 

In Section~\ref{sec_tropical_mutation}, we observe the \emph{tropicalized cluster mutation} defined for the Newton--Okounkov polytopes
arising from plabic graphs from the viewpoint of combinatorial mutations. 
Let $\Delta_G$ be the Newton--Okounkov polytope obtained from a plabic graph $G$ of type $\pi_{k,n}$. 
As we mentioned, a certain moving operation on $G$ corresponds to the mutation of the quiver associated to $G$. 
Such an operation makes $G$ another plabic graph $G^\prime$ of type $\pi_{k,n}$, and hence 
we have another Newton--Okounkov polytope $\Delta_{G^\prime}$. 
Then the relationship between $\Delta_G$ and $\Delta_{G^\prime}$ can be understood by using the tropicalized cluster mutation 
as shown in \cite[Corollaries~11.16 and 11.17, Theorem~16.18]{RW} (see also Theorem~\ref{thm_trop_map_NO}). 
In particular, any Newton--Okounkov polytope arising from a plabic graph of type $\pi_{k,n}$ can be related by the tropicalized cluster mutations. 
This tropicalized cluster mutation is a certain piecewise-linear map, and it preserves the number of lattice points (see \cite[Corollary~16.19]{RW}). 
In this point, the tropicalized cluster mutation is similar to the combinatorial mutation on the dual side. 
Thus, we compare these piecewise-linear maps, 
and show that the tropicalized cluster mutation can be written as the composition of a combinatorial mutation and a unimodular transformation, 
see Proposition~\ref{prop_tropical=combmutation}. 
We note that this statement can be considered as the dual of \cite[Proposition~3]{Cas}. 

In Section~\ref{sec_mutation_youngposet}, we focus on two kinds of special polytopes called 
\emph{order polytopes} and \emph{chain polytopes} introduced by Stanley \cite{S86}, 
which are constructed from a partially ordered set (poset), see Section~\ref{sec_mutation_youngposet} for the precise definition. 
For a poset $\Pi$, we denote by $\calO(\Pi)$ (resp., $\calC(\Pi)$) the order polytope (resp., chain polytope) associated to $\Pi$. 
If we consider the poset $\Pi_{k,n}$ whose Hasse diagram takes the form as in Figure~\ref{hasse_poset_kn}, 
then the polytopes $\calO(\Pi_{k,n})$ and $\calC(\Pi_{k,n})$ are meaningful in representation theory. 
First, it is known that lattice points in $\calO(\Pi_{k,n})$ parametrize a basis of finite dimensional 
irreducible representation of the Lie algebra $\slLie_{n+1}$ \cite{GT}. 
On the other hand, lattice points in the chain polytope $\calC(\Pi_{k,n})$ also parametrize a different basis of 
finite dimensional irreducible representation of $\slLie_{n+1}$ \cite{FFL} (this phenomenon was originally conjectured in \cite{Vin}). 
Thus, $\calO(\Pi_{k,n})$ and $\calC(\Pi_{k,n})$ contain the same number of lattice points. 
After these works, $\calO(\Pi_{k,n})$ is called the \emph{Gelfand--Tsetlin $($GT$)$ polytope}, 
and $\calC(\Pi_{k,n})$ is called the \emph{Feigin--Fourier--Littelmann--Vinberg $($FFLV$)$ polytope}. 
The point is that these polytopes are obtained as the Newton--Okounkov polytopes
associated to special types of plabic graphs. 
Precisely, as shown in \cite[Lemma~16.2]{RW} (see also Proposition~\ref{prop_unimodular_GT}), 
the GT polytope is unimodularly equivalent to the Newton--Okounkov polytope associated to the rectangle plabic graph $G_{k,n}^{\rm rec}$ of type $\pi_{k,n}$. 
On the other hand, the FFLV polytope is unimodularly equivalent to the Newton--Okounkov polytope 
associated to the ``dual plabic graph" $(G_{n-k,n}^{\rm rec})^\vee$ of $G_{n-k,n}^{\rm rec}$. 
This was first proved in \cite{FF}, but we will provide another proof in Theorem~\ref{thm_unimodular_FFLV2}. 
A feature of our method is that we give the explicit unimodular transformation realizing this equivalence. 
Whereas, it is known that the GT polytope $\calO(\Pi_{k,n})$ and the FFLV polytope $\calC(\Pi_{k,n})$ are unimodularly equivalent 
via the \emph{transfer map}, which was also introduced by Stanley \cite{S86}. 
The transfer map can also be written as a sequence of combinatorial mutations \cite[Theorem~4.1]{Hig}. 
In our situation, this sequence is constructed by using the partial order of Young diagrams labeling the rectangle plabic graph $G_{k,n}^{\rm rec}$, 
see Theorem~\ref{thm_transfer_mutation}. 
Collecting these results, we have the following diagram for the rectangle plabic graphs $G=G_{k,n}^{\rm rec}$ and $G^\prime=G_{n-k,n}^{\rm rec}$: 
\begin{equation}
\label{diagram_our_results}
\begin{tikzcd}
\Delta_G \arrow[r,"{\rm(b)}"]\arrow[d, "{\rm(a)}"'] & \calO(\Pi_{k,n}) \arrow[d,"{\rm(d)}"]\\
\Delta_{(G^\prime)^\vee} \arrow[r,"{\rm(c)}"']& \calC(\Pi_{k,n})
\end{tikzcd}
\end{equation}

\begin{itemize}
\setlength{\parskip}{0pt} 
\setlength{\itemsep}{3pt}
\item[(a)] A sequence of the tropicalized cluster mutations discussed in Section~\ref{sec_tropical_mutation}, which is interpreted as the composition of combinatorial mutations and unimodular transformations (see Proposition~\ref{prop_tropical=combmutation}). 
\item[(b)] The unimodular transformation given in \cite[Lemma~16.2]{RW} (cf. Proposition~\ref{prop_unimodular_GT}). 
\item[(c)] The unimodular transformation given in Theorem~\ref{thm_unimodular_FFLV2}, see also \cite{FF}. 
\item[(d)] The transfer map which is interpreted as a sequence of combinatorial mutations 
along the poset structure of Young diagrams as in Theorem~\ref{thm_transfer_mutation}. 
\end{itemize}

We note that the recent paper \cite{FK} explains a connection between the GT polytope and the FFLV polytope in more general contexts. 

The tropicalized cluster mutation mentioned above is defined as the tropicalization of the exchange relation which is used for defining the cluster algebra. 
This tropicalization is given by the ``minimum convention". 
On the other hand, the tropicalization with respect to the ``maximum convention" would also appear in the context of wall-crossings in the tropical Grassmannian (see \cite[Remark~5.11]{EH}). 
In Section~\ref{sec_max_tropical}, we first give a brief review of the tropical Grassmannian and its wall-crossing phenomenon. 
We then observe that the tropicalization with respect to the maximum convention appear when 
we consider the wall-crossing maps introduced in \cite{EH} for the valuations associated to plabic graphs. 
We also see that two kinds of tropicalizations are related by the piece-wise linear map used in the combinatorial mutation. 

In Appendix~\ref{appendix_Gr(3,6)}, we apply the statements discussed in this article to the case of the Grassmannian $\Gr(3,6)$. 
In this case, there are 34 plabic graphs of type $\pi_{3,6}$. 
The Newton--Okounkov polytopes associated to them are not necessarily integral, but they are transformed into one another 
by the tropicalized cluster mutations (and hence the composition of combinatorial mutations and unimodular transformations). 
We will give the mutation graph of plabic graphs of type $\pi_{3,6}$, which coincides with the one of the associated Newton--Okounkov polytopes. 
Thus, from this graph, we can find a sequence of mutations connecting $G_{3,6}^{\rm rec}$ and $(G_{3,6}^{\rm rec})^\vee$ 
(or $\Delta_{G_{3,6}^{\rm rec}}$ and $\Delta_{(G_{3,6}^{\rm rec})^\vee}$), which corresponds to the part (a) in \eqref{diagram_our_results}. 
We also give the descriptions of linear transformations $f$ and $g$ satisfying 
$f(\Delta_{G_{3,6}^{\rm rec}})=\calO(\Pi_{3,6})$ and $g(\Delta_{(G_{3,6}^{\rm rec})^\vee})=\calC(\Pi_{3,6})$, 
which respectively correspond to the parts (b) and (c) in \eqref{diagram_our_results}.

\subsection*{Notation and convention}
When we write $\ZZ^S, \QQ^S, \RR^S$ etc., for a finite set $S$, 
these mean $\ZZ^{|S|}, \QQ^{|S|}, \RR^{|S|}$ where $|S|$ is the cardinality of $S$. 

We employ the usual notation used in toric geometry. 
Fix a positive integer $d$. 
Let $N$ be a lattice of rank $d$, and let $M\coloneqq\Hom_\ZZ(N,\ZZ)$. 
We also define $N_\RR\coloneqq N\otimes_\ZZ\RR$ and $M_\RR\coloneqq M\otimes_\ZZ\RR$. 
In the sequel, we often identify $N$ and $M$ with $\ZZ^d$, and $N_\RR$ and $M_\RR$ with $\RR^d$. 
We denote the natural inner product by $\langle\;,\;\rangle: M_\RR\times N_\RR\rightarrow\RR$. 
We recall some fundamental materials on convex geometry. 
A \emph{hyperplane} (resp. \emph{affine half-space}) in $N_\RR$ is a subset defined by 
$\{ u \in N_\RR \mid \langle w,u \rangle = h \}$ (resp. $\{ u \in N_\RR \mid \langle w,u \rangle \geq h \}$) for some $w \in M_\RR$ and $h \in \RR$. 
We call $P \subset N_\RR$ a \emph{polyhedron} if $P$ is defined as the intersection of finitely many affine half-spaces. 
A polyhedron is said to be a \emph{polytope} if it is bounded. 
It is well known that a polytope $P \subset N_\RR$ is described as $\Conv(V)$ for some finite set $V$ in $N_\RR$, 
where $\Conv(V)$ denotes the \emph{convex hull} of $V$, which is the smallest convex set containing $V$. 
Let $P \subset N_\RR$ be a polytope and let $\calV(P)$ denote the set of vertices of $P$. 
We say that $P$ is a \emph{lattice polytope} (or \emph{integral polytope}) if $\calV(P) \subset N$, and 
we say that $P$ is a \emph{rational polytope} if there is a positive integer $n$ such that $nP$ is a lattice polytope. 
Those objects in $M_\RR$ are defined in the similar way. For more details on those notions, consult, e.g., \cite{Sch}. 
We say that two polytopes $P,Q\subset \RR^d$ are \emph{unimodularly equivalent}, denoted by $P\cong Q$, 
if they are transformed into each other by unimoduar transformations. 
Here, we say that a linear transformation is \emph{unimodular} if its representation matrix is in $\GL(d,\ZZ)$. 

\section{Preliminaries on Newton--Okounkov polytopes arising from plabic graphs} 
\label{sec_prelim_plabic}

In this section, we will introduce plabic graphs that give the cluster structures on Grassmannians, and related notions. 
This section is mostly for a review and fixing notation, see e.g. \cite{BFFHL,MaSc,Pos,RW,Sco} for more details. 

\subsection{Grassmannians}\label{subsec_Grassmannian}
Let $n, k$ be integers with $n\ge 1$ and $1\le k\le n-1$. 
We denote by ${ [n] \choose k }$ the set of all $k$-element subsets of $[n]\coloneqq \{1, \dots, n\}$. 
Let $\XX_{k,n}\coloneqq\Gr(k,n)$ be the Grassmannian of $k$-planes in $\CC^n$. 
More precisely, we fix a basis $\{v_1,\dots,v_n\}$ of $\CC^n$. 
We consider a $k$-dimensional subspace $W$ of $\CC^n$ whose basis is $\{w_1,\dots,w_k\}$. 
Thus, this basis can be described as a linear combination of $\{v_1,\dots,v_n\}$, that is, $w_j=\sum_{i=1}^n b_{ji}v_i$ for $j=1,\dots,k$. 
We consider the $k\times n$ matrix $B\coloneqq(b_{ji})$. 
For $J\in { [n] \choose k }$, we denote by $B_J$ the maximal minor of $B$ defined by choosing the columns indexed by $J$. 
The description of the matrix $B$ depends on a choice of a basis of $W$, but it is determined up to the multiplication of an invertible matrix. 
Thus, for $J \in { [n] \choose k }$, $B_J$ is determined up to the determinant of an invertible matrix. 
Then we can define the map sending $W$ to $(B_J)\in\PP^{{ [n] \choose k }-1}$. 
It is known that this induces the embedding $\XX_{k,n}\hookrightarrow \PP^{{ [n] \choose k }-1}$ called the \emph{Pl\"{u}cker embedding}. 
In particular, $\XX_{k,n}$ is a projective variety. 
By this description of $\XX_{k,n}$, the homogeneous coordinate ring of $\XX_{k,n}$ is given as follows. 
Let $L$ be a $k\times n$ matrix of indeterminates. 
For $J\in { [n] \choose k }$, we denote by $L_J$ the maximal minor of $L$ obtained by choosing the columns indexed by $J$. 
Then the homogeneous coordinate ring $A_{k,n}\coloneqq\CC[\XX_{k,n}]$ is generated by $\{L_J \mid J\in { [n] \choose k } \}$. 
Considering the morphism $\varphi_{k,n}$ from a polynomial ring $\CC[p_J \mid J\in { [n] \choose k } ]$ to $A_{k,n}$ 
such that $\varphi_{k,n}(p_J)=L_J$, we see that $A_{k,n}\cong \CC[p_J \mid J\in { [n] \choose k } ]\big/I_{k,n}$, 
where $I_{k,n}\coloneqq\Ker (\varphi_{k,n})$. 
Here, $I_{k,n}$ is called the \emph{Pl\"{u}cker ideal} and the relations among $p_J$ determined by 
generators of $I_{k,n}$ are called \emph{Pl\"{u}cker relations}. 
We denote the coset of a variable $p_J$ in $A_{k,n}$ by $\overline{p}_J$. 
We call $\overline{p}_J$ (or $L_J$) a \emph{Pl\"{u}cker coordinate}. 
It is known that $A_{k,n}$ has a cluster structure \cite{Sco} as we will see later. 
Such a structure can be obtained from a plabic graph which was developed by Postnikov \cite{Pos} for studying the totally positive Grassmannian.

\subsection{Plabic graphs}
\label{subsec_prelim_plabic}

We then introduce a \emph{plabic graph}. 
There are several combinatorial objects equivalent to this graph, e.g., a \emph{Postnikov diagram}, a \emph{dimer model on the disk}, 
and an \emph{alternating strand diagram} (see \cite{BKM,Pos,Sco}). 

\begin{definition}
\label{def_plabic}
A \emph{plabic graph} is a planar bicolored graph $G$ embedded in a disk up to homotopy equivalence, 
which has $n$ \emph{boundary vertices} on the boundary of the disk that are not colored and numbered $1, \dots, n$ in a clockwise order. 
In addition, $G$ has \emph{internal vertices} located in strictly inside the disk, and colored black or white. 
We suppose that each boundary vertex is adjacent to a single internal vertex, 
and assume that the subgraph obtained by removing boundary vertices is bipartite. 
The \emph{valency} (or \emph{degree}) of a vertex of $G$ is the number of edges that are incident to that vertex. 
We call a vertex whose valency is $d$ a $d$-valent vertex. 
\end{definition}

In the following, we assume that plabic graphs are connected, and that any leaf (i.e., a $1$-valent vertex) of a plabic graph is a boundary vertex. 
For example, Figure~\ref{ex_plabic1} is a plabic graph. 

\begin{figure}[H]
\begin{center}
\scalebox{0.8}{
\begin{tikzpicture}
\newcommand{\boundaryrad}{2cm} 
\newcommand{\noderad}{0.13cm} 
\newcommand{\nodewidth}{0.035cm} 
\newcommand{\edgewidth}{0.035cm} 
\newcommand{\boundarylabel}{8pt} 

\draw (0,0) circle(\boundaryrad) [gray, line width=\edgewidth];

\foreach \n/\a in {1/60, 2/0, 3/300, 4/240, 5/180, 6/120} {
\coordinate (B\n) at (\a:\boundaryrad); 
\coordinate (B\n+) at (\a:\boundaryrad+\boundarylabel); };
\foreach \n in {1,...,6} { \draw (B\n+) node {\small $\n$}; }
\foreach \n/\a/\r in {7/60/0.65, 8/0/0.65, 9/310/0.65, 10/270/0.5, 11/240/0.7, 12/200/0.4, 13/170/0.65, 14/120/0.65} {
\coordinate (B\n) at (\a:\r*\boundaryrad); }; 
\foreach \s/\t in {7/8, 8/9, 9/10, 10/11, 11/12, 12/13, 13/14, 14/7, 7/10, 7/12, 
1/7, 2/8, 3/9, 4/11, 5/13, 6/14} {
\draw[line width=\edgewidth] (B\s)--(B\t); 
};

\foreach \x in {7,9,11,13} {
\filldraw [fill=white, line width=\edgewidth] (B\x) circle [radius=\noderad] ;}; 
\foreach \x in {8,10,12,14} {
\filldraw [fill=black, line width=\edgewidth] (B\x) circle [radius=\noderad] ; };

\end{tikzpicture}
}
\end{center}
\caption{}
\label{ex_plabic1}
\end{figure}

A face of a plabic graph is called \emph{internal}, if it does not intersect with the boundary of the disk. 
Other faces are called \emph{boundary} faces. 

\medskip

\begin{definition}
\label{def_moves_plabic}
We  introduce the fundamental operations on a plabic graph as follows (see also Figure~\ref{fig_move_operation}):  
\begin{itemize}
\setlength{\parskip}{0pt} 
\setlength{\itemsep}{5pt}
\setlength{\leftskip}{0.3cm}
\item[(M1)] 
If a plabic graph contains a square consisting of four internal vertices with alternating colors, 
each of which is $3$-valent, then we switch the colors of these vertices 
(i.e., black vertices become white and vice versa). Then we insert some $2$-valent vertices to keep the bipartiteness. 
This operation is called the \emph{square move} (or \emph{urban renewal}). 
\item[(M2)]
If a plabic graph contains an internal $2$-valent vertex that is not adjacent to the boundary, 
then we can remove it and merge two adjacent vertices. Note that this operation can be reversed. 
\item[(M3)]
We may remove or add vertices of valency $2$ so that the resulting graph to be bipartite. 
\item[(R)] 
If a plabic graph contains two $3$-valent internal vertices of different colors connected by a pair of parallel edges, 
then we remove those vertices and edges and glue the remaining pair of edges together. 
\end{itemize}

These move operations make a given plabic graph another one. 
We say that two plabic graphs are called \emph{move-equivalent} if they are transformed into each other by applying the moves (M1)--(M3). 
We say that a plabic graph is \emph{reduced} if there is no graph in its move-equivalence class to which (R) can be applied. 

\begin{figure}[H]
\begin{center}
\scalebox{0.9}{
\begin{tikzpicture}
\newcommand{\noderad}{0.1cm} 
\newcommand{\nodewidth}{0.03cm} 
\newcommand{\edgewidth}{0.03cm} 

\node at (-1.2,0) {(M1) : }; 

\node at (0.8,0) {
\begin{tikzpicture}
\coordinate (A) at (0,0); \coordinate (B) at (1,0); 
\coordinate (C) at (1,1); \coordinate (D) at (0,1); 

\path (A) ++(225:0.9) coordinate (A+); \path (B) ++(315:0.9) coordinate (B+); 
\path (C) ++(45:0.9) coordinate (C+); \path (D) ++(135:0.9) coordinate (D+); 

\foreach \s/\t in {A/B,B/C,C/D,D/A,A/A+,B/B+,C/C+,D/D+} {
\draw[line width=\edgewidth] (\s)--(\t); }

\foreach \n in {A,C} {
\filldraw [fill=black, line width=\nodewidth] (\n) circle [radius=\noderad] ;}
\foreach \n in {B,D} {
\filldraw [fill=white, line width=\nodewidth] (\n) circle [radius=\noderad] ;} 
\end{tikzpicture}
}; 

\node at (2.8,0) {\Large$\longleftrightarrow$}; 

\node at (4.8,0) {
\begin{tikzpicture}
\path (A) ++(225:0.5) coordinate (A1); \path (B) ++(315:0.5) coordinate (B1); 
\path (C) ++(45:0.5) coordinate (C1); \path (D) ++(135:0.5) coordinate (D1); 

\foreach \s/\t in {A/B,B/C,C/D,D/A,A/A+,B/B+,C/C+,D/D+} {
\draw[line width=\edgewidth] (\s)--(\t); }

\foreach \n in {A,C,B1,D1} {
\filldraw [fill=white, line width=\nodewidth] (\n) circle [radius=\noderad] ;} 
\foreach \n in {B,D,A1,C1} {
\filldraw [fill=black, line width=\nodewidth] (\n) circle [radius=\noderad] ;}
\end{tikzpicture}
};

\node at (-5,-2.2) {(M2) : }; 
\node at (-2.5,-2.2) {
\begin{tikzpicture}
\coordinate (A) at (0,0); \coordinate (B) at (2,0); 
\coordinate (C) at (1,0); 

\path (A) ++(135:0.7) coordinate (A1); \path (A) ++(180:0.7) coordinate (A2); \path (A) ++(225:0.7) coordinate (A3); 
\path (B) ++(315:0.7) coordinate (B1); \path (B) ++(0:0.7) coordinate (B2); \path (B) ++(45:0.7) coordinate (B3); 

\draw[line width=\edgewidth] (A)--(B); 
\draw[line width=\edgewidth] (A)--(A1); \draw[line width=\edgewidth] (A)--(A2); \draw[line width=\edgewidth] (A)--(A3); 
\draw[line width=\edgewidth] (B)--(B1); \draw[line width=\edgewidth] (B)--(B2); \draw[line width=\edgewidth] (B)--(B3); 

\filldraw [fill=black, line width=\nodewidth] (A) circle [radius=\noderad] ;
\filldraw [fill=black, line width=\nodewidth] (B) circle [radius=\noderad] ;
\filldraw [fill=white, line width=\nodewidth] (C) circle [radius=\noderad] ;
\end{tikzpicture}
}; 

\node at (0,-2.2) {\Large$\longleftrightarrow$}; 

\node at (1.5,-2.2) {
\begin{tikzpicture}
\coordinate (A) at (0,0); \coordinate (B) at (0,0); 

\path (A) ++(135:0.7) coordinate (A1); \path (A) ++(180:0.7) coordinate (A2); \path (A) ++(225:0.7) coordinate (A3); 
\path (B) ++(315:0.7) coordinate (B1); \path (B) ++(0:0.7) coordinate (B2); \path (B) ++(45:0.7) coordinate (B3); 

\draw[line width=\edgewidth] (A)--(B); 
\draw[line width=\edgewidth] (A)--(A1); \draw[line width=\edgewidth] (A)--(A2); \draw[line width=\edgewidth] (A)--(A3); 
\draw[line width=\edgewidth] (B)--(B1); \draw[line width=\edgewidth] (B)--(B2); \draw[line width=\edgewidth] (B)--(B3); 

\filldraw [fill=black, line width=\nodewidth] (A) circle [radius=\noderad] ;
\filldraw [fill=black, line width=\nodewidth] (B) circle [radius=\noderad] ;
\end{tikzpicture}
}; 

\node at (4.2,-2.2) {(M3) : }; 

\node at (6,-2.2) {
\begin{tikzpicture}
\coordinate (O) at (1,0); 
\coordinate (A) at (0,0); \coordinate (B) at (2,0); 
\draw[line width=\edgewidth] (A)--(B); 
\filldraw [fill=white, line width=\nodewidth] (O) circle [radius=\noderad] ;
\end{tikzpicture}
}; 

\node at (7.8,-2.2) {\Large$\longleftrightarrow$}; 

\node at (9.6,-2.2) {
\begin{tikzpicture}
\draw[line width=\edgewidth] (A)--(B); 
\end{tikzpicture}
}; 

\node at (-1.4,-3.8) {(R) : }; 
\node at (0.6,-3.8) {
\begin{tikzpicture}
\coordinate (A) at (0,0); \coordinate (B) at (1,0); 

\path (A) ++(180:0.7) coordinate (A2); \path (B) ++(0:0.7) coordinate (B2); 

\draw[line width=\edgewidth] (A) to [bend left] (B); \draw[line width=\edgewidth] (A) to [bend right] (B); 
\draw[line width=\edgewidth] (A)--(A2); \draw[line width=\edgewidth] (B)--(B2); 

\filldraw [fill=black, line width=\nodewidth] (A) circle [radius=\noderad] ;
\filldraw [fill=white, line width=\nodewidth] (B) circle [radius=\noderad] ;
\end{tikzpicture}
}; 

\node at (2.8,-3.8) {\Large$\longleftrightarrow$}; 

\node at (5,-3.8) {
\begin{tikzpicture}
\draw[line width=\edgewidth] (A2)--(B2); 
\end{tikzpicture}
}; 
\end{tikzpicture}}
\end{center}
\caption{The list of operations on plabic graphs}
\label{fig_move_operation}
\end{figure}
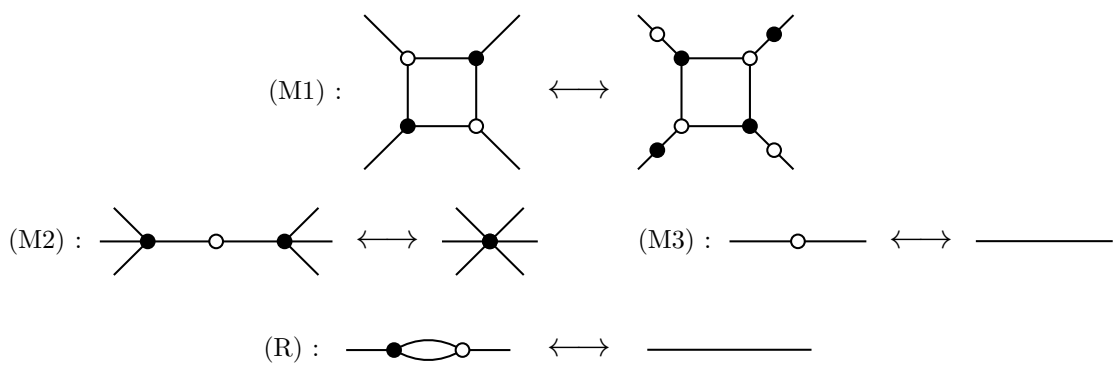
\end{definition}

Note that we often use the operation (M2) for making a square face of a plabic graph into a square face whose four vertices are all $3$-valent, 
see Figure~\ref{fig_square_move}. 
Thus, the square move (M1) can be applied to any square face after this modification. 
If the resulting bipartite graph contain $2$-valent vertices, then we can contract those using (M2) or (M3).
Note that (M3) can be used for a $2$-valent vertex that is adjacent to the boundary. 
In the following, we simply call this series of operations a square move even if we use (M2) and (M3) supplementally. 

\begin{figure}[H]
\begin{center}
\scalebox{0.85}{
\begin{tikzpicture}
\newcommand{\noderad}{0.1cm} 
\newcommand{\nodewidth}{0.03cm} 
\newcommand{\edgewidth}{0.03cm} 

\node at (0.3,0) {
\begin{tikzpicture}
\coordinate (A) at (0,0); \coordinate (B) at (1,0); 
\coordinate (C) at (1,1); \coordinate (D) at (0,1); 

\path (A) ++(180:0.7) coordinate (A1); \path (A) ++(-90:0.7) coordinate (A2); 
\path (B) ++(0:0.7) coordinate (B1); \path (B) ++(-90:0.7) coordinate (B2); 
\path (C) ++(0:0.7) coordinate (C1); \path (C) ++(90:0.7) coordinate (C2); 
\path (D) ++(180:0.7) coordinate (D1); \path (D) ++(90:0.7) coordinate (D2); 

\foreach \s/\t in {A/B,B/C,C/D,D/A,A/A1,A/A2,B/B1,B/B2,C/C1,C/C2,D/D1,D/D2} {
\draw[line width=\edgewidth] (\s)--(\t); }

\foreach \n in {A,C} {
\filldraw [fill=black, line width=\nodewidth] (\n) circle [radius=\noderad] ;}
\foreach \n in {B,D} {
\filldraw [fill=white, line width=\nodewidth] (\n) circle [radius=\noderad] ;} 
\end{tikzpicture}
}; 

\draw[line width=\edgewidth, <->] (1.8,0)--(3.2,0) node[midway,xshift=0cm,yshift=0.35cm] {\small (M2)} ;

\node at (5,0) {
\begin{tikzpicture}
\path (A) ++(225:0.5) coordinate (Aa); \path (A) ++(225:1) coordinate (Ab); 
\path (B) ++(315:0.5) coordinate (Ba); \path (B) ++(315:1) coordinate (Bb); 
\path (C) ++(45:0.5) coordinate (Ca); \path (C) ++(45:1) coordinate (Cb); 
\path (D) ++(135:0.5) coordinate (Da); \path (D) ++(135:1) coordinate (Db); 

\path (Ab) ++(180:0.4) coordinate (A1); \path (Ab) ++(-90:0.4) coordinate (A2); 
\path (Bb) ++(0:0.4) coordinate (B1); \path (Bb) ++(-90:0.4) coordinate (B2); 
\path (Cb) ++(0:0.4) coordinate (C1); \path (Cb) ++(90:0.4) coordinate (C2); 
\path (Db) ++(180:0.4) coordinate (D1); \path (Db) ++(90:0.4) coordinate (D2); 

\foreach \s/\t in {A/B,B/C,C/D,D/A,A/Aa,Aa/Ab,Ab/A1,Ab/A2,B/Ba,Ba/Bb,Bb/B1,Bb/B2, 
C/Ca,Ca/Cb,Cb/C1,Cb/C2,D/Da,Da/Db,Db/D1,Db/D2}{
\draw[line width=\edgewidth] (\s)--(\t); }

\foreach \n in {A,C,Ab,Ba,Cb,Da} {
\filldraw [fill=black, line width=\nodewidth] (\n) circle [radius=\noderad] ;}
\foreach \n in {B,D,Aa,Bb,Ca,Db} {
\filldraw [fill=white, line width=\nodewidth] (\n) circle [radius=\noderad] ;} 
\end{tikzpicture}
}; 

\draw[line width=\edgewidth, <->] (6.8,0)--(8.2,0) node[midway,xshift=0cm,yshift=0.35cm] {\small (M1)};  

\node at (10,0) {
\begin{tikzpicture}

\path (A) ++(225:0.66) coordinate (Aa); \path (A) ++(225:0.33) coordinate (Ac); 
\path (B) ++(315:0.66) coordinate (Ba); \path (B) ++(315:0.33) coordinate (Bc); 
\path (C) ++(45:0.66) coordinate (Ca); \path (C) ++(45:0.33) coordinate (Cc); 
\path (D) ++(135:0.66) coordinate (Da); \path (D) ++(135:0.33) coordinate (Dc); 

\foreach \s/\t in {A/B,B/C,C/D,D/A,A/Aa,Aa/Ab,Ab/A1,Ab/A2,B/Ba,Ba/Bb,Bb/B1,Bb/B2, 
C/Ca,Ca/Cb,Cb/C1,Cb/C2,D/Da,Da/Db,Db/D1,Db/D2}{
\draw[line width=\edgewidth] (\s)--(\t); }

\foreach \n in {B,D,Ab,Ac,Ba,Cb,Cc,Da} {
\filldraw [fill=black, line width=\nodewidth] (\n) circle [radius=\noderad] ;}
\foreach \n in {A,C,Aa,Bb,Bc,Ca,Db,Dc} {
\filldraw [fill=white, line width=\nodewidth] (\n) circle [radius=\noderad] ;} 
\end{tikzpicture}
}; 

\draw[line width=\edgewidth, <->] (11.8,0)--(13.2,0) node[midway,xshift=0cm,yshift=0.35cm] {\small (M2)}
node[midway,xshift=0.1cm,yshift=-0.35cm] {\small or (M3)};  

\node at (15,0) {
\begin{tikzpicture}

\foreach \s/\t in {A/B,B/C,C/D,D/A,A/Aa,Aa/Ab,Ab/A1,Ab/A2,B/Ba,Ba/Bb,Bb/B1,Bb/B2, 
C/Ca,Ca/Cb,Cb/C1,Cb/C2,D/Da,Da/Db,Db/D1,Db/D2}{
\draw[line width=\edgewidth] (\s)--(\t); }

\foreach \n in {B,D,Ab,Cb} {
\filldraw [fill=black, line width=\nodewidth] (\n) circle [radius=\noderad] ;}
\foreach \n in {A,C,Bb,Db} {
\filldraw [fill=white, line width=\nodewidth] (\n) circle [radius=\noderad] ;} 
\end{tikzpicture}
}; 

\end{tikzpicture}}
\end{center}
\caption{}
\label{fig_square_move}
\end{figure}

\begin{example}
\label{ex_plabic_mutation}
Let us consider the plabic graph given in Figure~\ref{ex_plabic1}, which coincides with the left of Figure~\ref{fig_plabic_mutation} below. 
Applying the square move to the face $\mu$, we have another plabic graph as in the right of Figure~\ref{fig_plabic_mutation}. 
Whereas, if we apply the square move to the face $\mu^\prime$, then we can recover the original graph. 
\begin{figure}[H]
\begin{center}
\scalebox{0.8}{
\begin{tikzpicture}
\newcommand{\boundaryrad}{2cm} 
\newcommand{\noderad}{0.13cm} 
\newcommand{\nodewidth}{0.035cm} 
\newcommand{\edgewidth}{0.035cm} 
\newcommand{\boundarylabel}{8pt} 

\node at (0,0){
\begin{tikzpicture}
\draw (0,0) circle(\boundaryrad) [gray, line width=\edgewidth];
\foreach \n/\a in {1/60, 2/0, 3/300, 4/240, 5/180, 6/120} {
\coordinate (B\n) at (\a:\boundaryrad); 
\coordinate (B\n+) at (\a:\boundaryrad+\boundarylabel); };
\foreach \n in {1,...,6} { \draw (B\n+) node {\small $\n$}; }
\foreach \n/\a/\r in {7/60/0.65, 8/0/0.65, 9/310/0.65, 10/270/0.5, 11/240/0.7, 12/200/0.4, 13/170/0.65, 14/120/0.65} {
\coordinate (B\n) at (\a:\r*\boundaryrad); }; 
\foreach \s/\t in {7/8, 8/9, 9/10, 10/11, 11/12, 12/13, 13/14, 14/7, 7/10, 7/12, 
1/7, 2/8, 3/9, 4/11, 5/13, 6/14} {
\draw[line width=\edgewidth] (B\s)--(B\t); };
\foreach \x in {7,9,11,13} {
\filldraw [fill=white, line width=\edgewidth] (B\x) circle [radius=\noderad] ;}; 
\foreach \x in {8,10,12,14} {
\filldraw [fill=black, line width=\edgewidth] (B\x) circle [radius=\noderad] ; };
\node[blue] (V36) at (240:0.2*\boundaryrad) {\small $\mu$} ; 
\end{tikzpicture}};

\draw[<->, line width=0.03cm] (3,0)--(4.5,0) ; 

\node at (7.5,0){
\begin{tikzpicture}
\draw (0,0) circle(\boundaryrad) [gray, line width=\edgewidth];
\foreach \n/\a in {1/60, 2/0, 3/300, 4/240, 5/180, 6/120} {
\coordinate (B\n) at (\a:\boundaryrad); 
\coordinate (B\n+) at (\a:\boundaryrad+\boundarylabel); };
\foreach \n in {1,...,6} { \draw (B\n+) node {\small $\n$}; }
\foreach \n/\a/\r in {7/60/0.65, 8/0/0.65, 9/310/0.65, 10/240/0.65, 11/170/0.65, 12/120/0.65, 13/0/0} {
\coordinate (B\n) at (\a:\r*\boundaryrad); }; 
\foreach \s/\t in {7/8, 8/9, 9/10, 10/11, 11/12, 12/7, 7/13, 9/13, 11/13, 
1/7, 2/8, 3/9, 4/10, 5/11, 6/12} {
\draw[line width=\edgewidth] (B\s)--(B\t); };
\foreach \x in {7,9,11} {
\filldraw [fill=white, line width=\edgewidth] (B\x) circle [radius=\noderad] ;}; 
\foreach \x in {8,10,12,13} {
\filldraw [fill=black, line width=\edgewidth] (B\x) circle [radius=\noderad] ; };
\node[blue] (V36) at (240:0.3*\boundaryrad) {\small $\mu^\prime$} ; 
\end{tikzpicture}};

\end{tikzpicture}
}
\end{center}
\caption{Applying the square move to the face marked by $\mu$ and $\mu^\prime$}
\label{fig_plabic_mutation}
\end{figure}
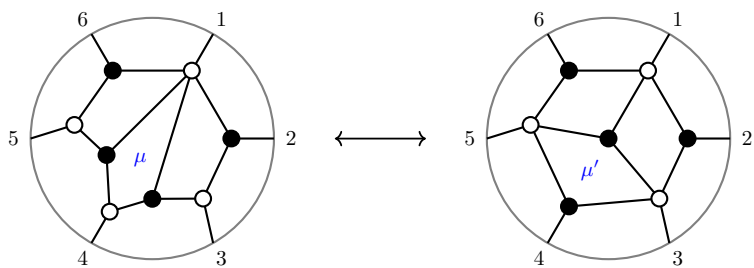
\end{example}

\medskip

In what follows, we assume that a plabic graph is reduced unless otherwise noted. 

\begin{definition}
Let $G$ be a plabic graph with boundary vertices $1, \dots, n$ labeled in a clockwise order. 
We define the \emph{trip permutation} $\pi_G$ as follows. 
First, we start at a boundary vertex $i$ and take a path along the edges of $G$ by turning maximally right 
at an internal black vertex and maximally left at an internal white vertex. 
We end this path if we arrive at a boundary vertex which we denote by $\pi(i)$. 
We call this path the \emph{trip} (or \emph{strand}) from $i$ and denote by $T_i$. 
We then define 
\[
\pi_G\coloneqq
{\small\begin{pmatrix}
1&2&\cdots&n \\
\pi(1)&\pi(2)&\cdots&\pi(n)
\end{pmatrix}}\in \fkS_n, 
\]
which is a permutation of $[n]$. 
\end{definition}

In the following, we consider the permutation 
\[
\pi_{k,n}\coloneqq
{\small\begin{pmatrix}
1&\cdots&n-k&n-k+1&\cdots&n \\
k+1&\cdots&n&1&\cdots&k
\end{pmatrix}}\in \fkS_n,
\]
\noindent
which is called the \emph{Grassmannian permutation}. 
We say that a plabic graph $G$ is type $\pi_{k,n}$ if $\pi_G=\pi_{k,n}$. 
The following is fundamental. 

\begin{proposition}[{\cite[Theorem~13.4]{Pos}}]
\label{prop_move_equiv}
Any two plabic graphs of type $\pi_{k,n}$ are move-equivalent. 
\end{proposition}

In the following, we assume that $G$ is a plabic graphs of type $\pi_{k,n}$. 
We here summarize basic properties of $G$, see \cite{Pos,RW,Sco} for more details. 
First, each trip $T_i$ on $G$ partitions the disk containing $G$ into two parts, 
that is, the part on the left of $T_i$, and the part on the right. 
We put a number $i$ on each face of $G$ which is to the left of $T_i$. 
After considering trips $T_1,\dots, T_n$, each face will contain a $k$ element subset of $[n]$. 
Thus, we can label each face of a plabic graph with a $k$-subset in ${ [n] \choose k } $. 
It is known that the number of faces of $G$ is equal to $s+1$ where $s\coloneqq k(n-k)$, 
and boundary faces are labeled by $k$-subsets of the form $\beta_i\coloneqq [i, i+k-1]$ for any $G$, 
where $i=1,\dots,n$ and each number reduced modulo $n$. 
We then index $k$-subsets of $[n]$ using Young diagrams as follows.

\begin{definition}
\label{def_Young}
Let $\calP_{k,n}$ be the set  of Young diagrams fitting inside a $k\times (n-k)$ rectangle.
We give a natural bijection between $\calP_{k,n}$ and ${ [n] \choose k }$ in the following way. 
For $\lambda\in\calP_{k,n}$, we arrange $\lambda$ so that its top-left corner coincides with the top-left corner of the $k\times (n-k)$ rectangle. 
Then we cut out the southeast border of $\lambda$ along a path 
from the northeast corner to the southwest corner of the rectangle, which consists of $k$ south steps and $n-k$ west steps. 
We label these $n$ steps by $\{1,\dots, n\}$ in order. 
Then we assign the numbers labeling the south steps to $\lambda$. 
This procedure gives a bijection between $\calP_{k,n}$ and ${ [n] \choose k }$, thus we will identify these two sets. 
Note that the $k$-subset $\{n-k+1,\dots, n\}$ corresponds to the empty Young diagram $\varnothing$. 
\end{definition}

Thus, we can label each face in a plabic graph with a Young diagram. 
For example, the faces of the plabic graph given in Figure~\ref{ex_plabic1} are labeled as Figure~\ref{ex_plabic3}. 
We denote by $\widetilde{\calP}_G$ the set of all Young diagrams labeling the faces of a plabic graph $G$, 
thus $|\widetilde{\calP}_G|=s+1$. 
Let $\calP_G\coloneqq\widetilde{\calP}_G{\setminus}\{\varnothing\}$. 

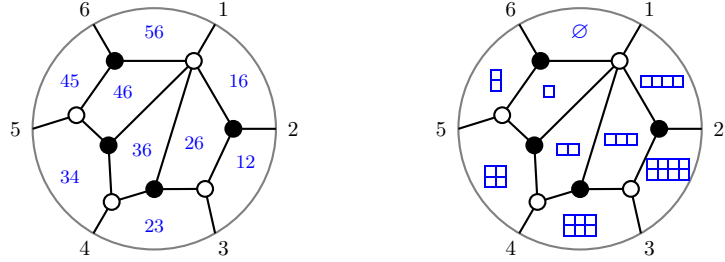
\begin{figure}[H]
\begin{center}
\scalebox{0.8}{
\begin{tikzpicture}[myarrow/.style={black, -latex}]
\newcommand{\boundaryrad}{2cm} 
\newcommand{\noderad}{0.13cm} 
\newcommand{\nodewidth}{0.035cm} 
\newcommand{\edgewidth}{0.035cm} 
\newcommand{\boundarylabel}{8pt} 
\newcommand{\arrowwidth}{0.05cm} 

\ytableausetup{boxsize=4.5pt}

\foreach \n/\a in {1/60, 2/0, 3/300, 4/240, 5/180, 6/120} {
\coordinate (B\n) at (\a:\boundaryrad); 
\coordinate (B\n+) at (\a:\boundaryrad+\boundarylabel); };

\node at (0,0) {
\begin{tikzpicture}
\draw (0,0) circle(\boundaryrad) [gray, line width=\edgewidth];
\foreach \n in {1,...,6} { \draw (B\n+) node {\small $\n$}; }
\foreach \n/\a/\r in {7/60/0.65, 8/0/0.65, 9/310/0.65, 10/270/0.5, 11/240/0.7, 12/200/0.4, 13/170/0.65, 14/120/0.65} {
\coordinate (B\n) at (\a:\r*\boundaryrad); }; 
\foreach \s/\t in {7/8, 8/9, 9/10, 10/11, 11/12, 12/13, 13/14, 14/7, 7/10, 7/12, 
1/7, 2/8, 3/9, 4/11, 5/13, 6/14} {
\draw[line width=\edgewidth] (B\s)--(B\t); 
};
\foreach \x in {7,9,11,13} {
\filldraw [fill=white, line width=\edgewidth] (B\x) circle [radius=\noderad] ;}; 
\foreach \x in {8,10,12,14} {
\filldraw [fill=black, line width=\edgewidth] (B\x) circle [radius=\noderad] ; };
\foreach \n/\a/\r in {12/340/0.8, 23/270/0.8, 34/210/0.8, 45/150/0.8, 56/90/0.8, 16/30/0.8} {
\node[blue] (V\n) at (\a:\r*\boundaryrad) {\footnotesize $\n$} ; }; 
\foreach \n/\a/\r in {26/-20/0.35, 36/240/0.2, 46/130/0.4} {
\node[blue] (V\n) at (\a:\r*\boundaryrad) {\footnotesize $\n$} ; }; 
\end{tikzpicture}};

\node at (7,0) {
\begin{tikzpicture}
\draw (0,0) circle(\boundaryrad) [gray, line width=\edgewidth];
\foreach \n in {1,...,6} { \draw (B\n+) node {\small $\n$}; }
\foreach \n/\a/\r in {7/60/0.65, 8/0/0.65, 9/310/0.65, 10/270/0.5, 11/240/0.7, 12/200/0.4, 13/170/0.65, 14/120/0.65} {
\coordinate (B\n) at (\a:\r*\boundaryrad); }; 
\foreach \s/\t in {7/8, 8/9, 9/10, 10/11, 11/12, 12/13, 13/14, 14/7, 7/10, 7/12, 
1/7, 2/8, 3/9, 4/11, 5/13, 6/14} {
\draw[line width=\edgewidth] (B\s)--(B\t); 
};
\foreach \x in {7,9,11,13} {
\filldraw [fill=white, line width=\edgewidth] (B\x) circle [radius=\noderad] ;}; 
\foreach \x in {8,10,12,14} {
\filldraw [fill=black, line width=\edgewidth] (B\x) circle [radius=\noderad] ; };
\foreach \n/\a/\r in {12/335/0.8, 23/270/0.8, 34/210/0.8, 45/150/0.8, 56/90/0.8, 16/30/0.78} {
\coordinate (V\n) at (\a:\r*\boundaryrad); }; 
\node[blue] at (V12) {\ydiagram{4,4}}; 
\node[blue] at (V23) {\ydiagram{3,3}}; 
\node[blue] at (V34) {\ydiagram{2,2}}; 
\node[blue] at (V45) {\ydiagram{1,1}}; 
\node[blue] at (V56) {$\varnothing$};
\node[blue] at (V16) {\ydiagram{4}};  
\foreach \n/\a/\r in {26/-15/0.35, 36/240/0.2, 46/130/0.4} {
\coordinate (V\n) at (\a:\r*\boundaryrad); }; 
\node[blue] at (V26) {\ydiagram{3}}; 
\node[blue] at (V36) {\ydiagram{2}}; 
\node[blue] at (V46) {\ydiagram{1}}; 
\end{tikzpicture}};

\end{tikzpicture}}
\end{center}
\caption{The plabic graph labeled by $2$-subsets (left) and Young diagrams (right)}
\label{ex_plabic3}
\end{figure}

\begin{definition}
\label{def_rectangle}
Let $G$ be a plabic graph of type $\pi_{k,n}$. 
We say that $G$ is \emph{rectangle} if each face of $G$ is labeled by a Young diagram whose shape is a rectangular grid pattern. 
We denote the rectangle plabic graph of type $\pi_{k,n}$ by $G_{k,n}^{\rec}$. 
\end{definition}

A rectangle plabic graph is also called \emph{regular}, and the method to construct it for fixed $k$ and $n$ is given in \cite[Section~8]{MaSc}. 
For example, the right of Figure~\ref{ex_plabic3} shows that Figure~\ref{ex_plabic1} is the rectangle plabic graph $G_{2,6}^{\rec}$. 
We remark that some papers, e.g., \cite{FF,RW}, use the ``opposite" convention, that is, 
the rectangle plabic graph of type $\pi_{k,n}$ is denoted by $G_{n-k,n}^{\rec}$ in those papers instead of $G_{k,n}^{\rec}$. 

\medskip

We then introduce the quiver, which is a finite directed graph, obtained as the dual of $G$. 

\begin{definition}
\label{def_quivar_G}
Let $G$ be a plabic graph, and assume that $G$ contains no internal $2$-valent vertex. 
(If $G$ contains such a vertex, then we remove it by the operation (M3).)  
We define the quiver $Q(G)$ associated to $G$ in the following way. 
We assign a vertex of $Q(G)$ dual to face in $G$, thus we identify the set of vertices of $Q(G)$ with $\widetilde{\calP}_G$. 
A vertex of $Q(G)$ is said to be \emph{frozen} if the corresponding face in $G$ is a boundary face, and is said to be \emph{mutable} otherwise. 
We draw an arrow dual to any edge $e$ in $G$ that separates two faces, and at least one of which corresponds to a mutable vertex. 
We determine the orientation of arrows so that the white vertex is on the right of the arrow when it crosses over $e$. 
\end{definition}

For example, the following is the quiver associated to the plabic graph given in Figure~\ref{ex_plabic1}. 

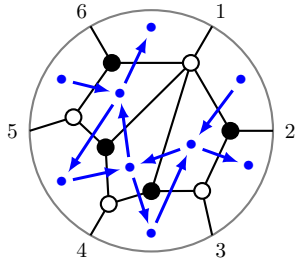
\begin{figure}[H]
\begin{center}
\scalebox{0.8}{
\begin{tikzpicture}[myarrow/.style={black, -latex}]
\newcommand{\boundaryrad}{2cm} 
\newcommand{\noderad}{0.13cm} 
\newcommand{\nodewidth}{0.035cm} 
\newcommand{\edgewidth}{0.035cm} 
\newcommand{\boundarylabel}{8pt} 
\newcommand{\arrowwidth}{0.05cm} 

\foreach \n/\a in {1/60, 2/0, 3/300, 4/240, 5/180, 6/120} {
\coordinate (B\n) at (\a:\boundaryrad); 
\coordinate (B\n+) at (\a:\boundaryrad+\boundarylabel); };
\draw (0,0) circle(\boundaryrad) [gray, line width=\edgewidth];
\foreach \n in {1,...,6} { \draw (B\n+) node {\small $\n$}; }
\foreach \n/\a/\r in {7/60/0.65, 8/0/0.65, 9/310/0.65, 10/270/0.5, 11/240/0.7, 12/200/0.4, 13/170/0.65, 14/120/0.65} {
\coordinate (B\n) at (\a:\r*\boundaryrad); }; 
\foreach \s/\t in {7/8, 8/9, 9/10, 10/11, 11/12, 12/13, 13/14, 14/7, 7/10, 7/12, 
1/7, 2/8, 3/9, 4/11, 5/13, 6/14} {
\draw[line width=\edgewidth] (B\s)--(B\t); 
};
\foreach \x in {7,9,11,13} {
\filldraw [fill=white, line width=\edgewidth] (B\x) circle [radius=\noderad] ;}; 
\foreach \x in {8,10,12,14} {
\filldraw [fill=black, line width=\edgewidth] (B\x) circle [radius=\noderad] ; };
\foreach \n/\a/\r in {12/340/0.85, 23/270/0.85, 34/210/0.85, 45/150/0.85, 56/90/0.85, 16/30/0.85} {
\node[blue] (V\n) at (\a:\r*\boundaryrad) {$\bullet$} ; }; 
\foreach \n/\a/\r in {26/-20/0.35, 36/240/0.35, 46/130/0.4} {
\node[blue] (V\n) at (\a:\r*\boundaryrad) {$\bullet$} ; }; 
\foreach \s/\t in {16/26, 26/12, 26/36, 23/26, 36/23, 36/46, 34/36, 46/34, 45/46, 46/56} {
\draw[shorten >=-0.07cm, shorten <=-0.05cm, myarrow, blue, line width=\arrowwidth] (V\s)--(V\t); };

\end{tikzpicture}}
\end{center}
\caption{The quiver associated to the plabic graph  $G_{2,6}^{\rec}$ given in Figure~\ref{ex_plabic1}}
\label{ex_plabic2}
\end{figure}

\begin{definition}[{Quiver mutation}]
\label{def_quiver_mutation}
Let $\lambda$ be a mutable vertex of a quiver $Q=Q(G)$. 
The \emph{quiver mutation} $\qmut_\lambda$ of $Q$ at $\lambda$ (or in direction $\lambda$) is the operation transforming $Q$ into 
another quiver which will be denoted by $\qmut_\lambda(Q)$ by performing the following steps:
\begin{itemize}
\item[(1)] For each oriented path $i \rightarrow \lambda\rightarrow j$ passing through $\lambda$, we add a new arrow $i\rightarrow  j$ 
unless $i$ and $j$ are both frozen. 
\item[(2)] Reverse the orientation of all arrows incident to $\lambda$. 
\item[(3)] Remove oriented $2$-cycles until we are unable to do this. 
\end{itemize}
Two quivers are called \emph{mutation-equivalent} if they are transformed into each other by the repetition of quiver mutations. 
\end{definition}

The quiver mutation at a certain vertex of $Q(G)$ can be understood in terms of operations in Definition~\ref{def_moves_plabic}. 
In particular, the next lemma follows from the definition (cf. \cite{Pos,Sco}). 

\begin{lemma}
\label{lem_mutation_corresp}
Let $G, G^\prime$ be plabic graphs. 
If $G, G^\prime$ are transformed into each other by the square move {\rm(M1)} 
at a square face $\lambda$, then the associated quivers $Q(G)$ and $Q(G^\prime)$ are transformed into each other by the mutation at 
the vertex corresponding to $\lambda$. 

On the other hand, if $\lambda$ is a square face of $G$, then 
the mutation of $Q(G)$ at a vertex corresponding to $\lambda$ is 
the quiver $Q(G^\prime)$ associated to $G^\prime$ that is obtained by applying the square move to $\lambda$. 
\end{lemma}

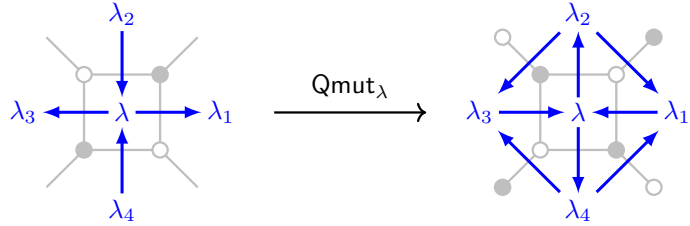
\begin{figure}[H]
\begin{center}
\scalebox{1}{
\begin{tikzpicture}[myarrow/.style={black, -latex}]
\newcommand{\noderad}{0.1cm} 
\newcommand{\nodewidth}{0.03cm} 
\newcommand{\edgewidth}{0.03cm} 
\newcommand{\arrowwidth}{0.04cm} 


\node at (0,0) {
\begin{tikzpicture}
\foreach \n/\a/\b in {1/0/0,2/1/0,3/1/1,4/0/1} {
\coordinate (V\n) at (\a,\b); }
\foreach \n/\a/\r/\m in {1/225/5,2/315/6,3/45/7,4/135/8}{
\path (V\n) ++(\a:0.7) coordinate (V\m); }
\foreach \s/\t in {1/2,2/3,3/4,4/1,1/5,2/6,3/7,4/8}{
\draw[line width=\edgewidth,lightgray] (V\s)--(V\t); }
\foreach \n in{1,3}{
\filldraw [lightgray, fill=lightgray, line width=\nodewidth] (V\n) circle [radius=\noderad] ;
}
\foreach \n in{2,4}{
\filldraw [lightgray, fill=white, line width=\nodewidth] (V\n) circle [radius=\noderad] ;
}
\node[blue] (Q0) at (0.5,0.5) {\small $\lambda$} ;
\node[blue] (Q1) at (1.8,0.5) {\small $\lambda_1$} ;\node[blue] (Q2) at (0.5,1.8) {\small $\lambda_2$} ;
\node[blue] (Q3) at (-0.8,0.5) {\small $\lambda_3$} ;\node[blue] (Q4) at (0.5,-0.8) {\small $\lambda_4$} ;
\foreach \s/\t in {0/1,0/3,2/0,4/0} {
\draw[shorten >=-0.06cm, shorten <=-0.06cm, myarrow, line width=\arrowwidth, blue] (Q\s)--(Q\t); 
};
\end{tikzpicture}};

\node at (6,0) {
\begin{tikzpicture}
\foreach \n/\a/\b in {1/0/0,2/1/0,3/1/1,4/0/1} {
\coordinate (V\n) at (\a,\b); }
\foreach \n/\a/\r/\m in {1/225/5,2/315/6,3/45/7,4/135/8}{
\path (V\n) ++(\a:0.7) coordinate (V\m); }
\foreach \s/\t in {1/2,2/3,3/4,4/1,1/5,2/6,3/7,4/8}{
\draw[line width=\edgewidth,lightgray] (V\s)--(V\t); }
\foreach \n in{2,4,5,7}{
\filldraw [lightgray, fill=lightgray, line width=\nodewidth] (V\n) circle [radius=\noderad] ;
}
\foreach \n in{1,3,6,8}{
\filldraw [lightgray, fill=white, line width=\nodewidth] (V\n) circle [radius=\noderad] ;
}
\node[blue] (Q0) at (0.5,0.5) {\small $\lambda$} ;
\node[blue] (Q1) at (1.8,0.5) {\small $\lambda_1$} ;\node[blue] (Q2) at (0.5,1.8) {\small $\lambda_2$} ;
\node[blue] (Q3) at (-0.8,0.5) {\small $\lambda_3$} ;\node[blue] (Q4) at (0.5,-0.8) {\small $\lambda_4$} ;

\foreach \s/\t in {1/0,3/0,0/2,0/4,2/1,2/3,4/1,4/3} {
\draw[shorten >=-0.06cm, shorten <=-0.06cm, myarrow, line width=\arrowwidth,blue] (Q\s)--(Q\t); }; 
\end{tikzpicture}};


\draw[line width=\edgewidth, ->] (2,0)--(4,0) node[midway,xshift=0cm,yshift=0.35cm] {\small $\qmut_\lambda$} ;

\end{tikzpicture}}
\end{center}
\caption{The quiver mutation at $\lambda$ corresponding to a square face}
\label{fig_quiver_mut}
\end{figure}

\begin{remark}
\label{rem_mutation_not_dual}
By Proposition~\ref{prop_move_equiv} and Lemma~\ref{lem_mutation_corresp}, 
all quivers associated to plabic graphs of type $\pi_{k,n}$ are mutation-equivalent. 
We note that even if a mutable vertex of $Q(G)$ does not correspond to a square face in $G$, 
in which case the number of incoming arrows  (equivalently outgoing arrows) at the vertex is
greater than two, we can apply the quiver mutation. 
However, the resulting quiver is no longer the dual of a plabic graph in this situation. 
\end{remark}

\medskip

We then mention the \emph{cluster structure} on the Grassmannian $\XX_{k,n}=\Gr(k,n)$. 
A \emph{cluster algebra} was introduced by Fomin and Zelevinsky  \cite{FZ}. 
To define a cluster algebra, we consider a certain free generating set of an ambient field, called \emph{$($extended$)$ cluster}, 
and certain combinatorial data (e.g.,  a quiver or a skew symmetric matrix) associated to a given cluster. 
A pair of a cluster and such a combinatorial data is called a \emph{labeled seed}. 
By recursively applying the operation called the \emph{seed mutation}, which makes a given labeled seed another one, 
we have a family of labeled seeds, and hence a family of clusters. 
A cluster algebra is a commutative algebra generated by such a family of clusters, 
and it gives algebraic and combinatorial frameworks to understand the dual canonical basis and total positivity in Lie theory. 
After this algebra was introduced, the relationships between cluster algebras and many fields of mathematics have been discovered. 
The homogeneous coordinate ring $A_{k,n}$ of $\XX_{k,n}$ is one of important examples of cluster algebras as shown in \cite{FZ,Sco}. 
Since the cluster structure endowed with the Grassmannian plays an important role in Section~\ref{sec_tropical_mutation}, 
we now give a brief explanation that $A_{k,n}$ can be obtained from a plabic graph $G$ of type $\pi_{k,n}$ as a cluster algebra. 
It follows from \cite[Theorem~3]{Sco} that a pair of the quiver $Q(G)$ associated to $G$ and 
indeterminate variables indexed by faces of $G$, called \emph{$\calA$-cluster variables}, forms a labeled seed $\Sigma_G$, 
that is, the set of such indeterminate variables is a cluster. 
Since faces of $G$ are divided into two parts (i.e., internal faces and boundary faces), 
a cluster is also divided into \emph{mutable} and \emph{frozen} (or \emph{coefficient}) variables 
that respectively correspond to internal and boundary faces of $G$. 
In this situation, the seed mutation can be defined by combining the quiver mutation $\qmut_\mu(Q(G))$ 
at a mutable vertex $\mu$ and the \emph{exchange relation} that replace a mutable variable corresponding to $\mu$ by another variable. 
Repeating the seed mutations to $\Sigma_G$, we obtain a lot of (usually infinitely many) clusters, 
but we note that frozen variables are preserved. 
Then, the cluster algebra (of geometric type) $\calA_{\Sigma_G}$ associated to $\Sigma_G$ is defined as 
the $R$-algebra generated by variables derived from mutable variables in $\Sigma_G$ by sequences of seed mutations, 
where $R$ is the $\CC$-algebra generated by frozen variables in $\Sigma_G$. 
By \cite[Theorem~3]{Sco}, there is an isomorphism $\calA_{\Sigma_G}\cong A_{k,n}$, 
thus we can consider the homogeneous coordinate ring $A_{k,n}$ as a cluster algebra. 
We note that, in some literatures, the frozen variables are inverted when we define the cluster algebra, 
in which case the labeled seed $\Sigma_G$ gives a cluster structure on a certain open positroid variety (see e.g., \cite{GL,MuSp}). 

\begin{remark}
Throughout this article, we consider the cluster structure arising from a plabic graph of type $\pi_{k,n}$. 
However, when $k=1$ or $n-1$, in which case $\XX_{k,n}\cong\PP^{n-1}$, a plabic graph has only boundary faces, 
and hence the associated quiver has only frozen vertices. 
Since we can not apply the mutations to such a plabic graph, we will exclude these cases. 
\end{remark}

\begin{remark}
A plabic graph has been studied from the viewpoint of representation theory of algebras (see e.g., \cite{BKM,JKS1,JKS2,Pre}). 
More precisely, we can obtain a quiver with a potential which consists of the quiver associated to a plabic graph and a certain linear combination of cycles of the quiver. 
(Note that in this case we also consider an arrow between adjacent frozen vertices.) 
The (\emph{frozen}) \emph{Jacobian algebra} defined from the quiver with a potential is important in this context. 
Also, the mutation introduced in Definition~\ref{def_quiver_mutation} (and hence the square move) can be generalized to the quiver with a potential and the Jacobian algebra as in \cite{Pre}. 
In addition, maximal Cohen-Macaulay modules over a certain algebra related to a plabic graph provides an additive categorification of the cluster structure of $\Gr(k,n)$ \cite{JKS1,JKS2}. 
In this context, the valuation associated to a plabic graph, which will be introduced in Subsection~\ref{subsec_prelim_NO}, is also categorified 
(see \cite[Remark~5.6]{JKS2} and \cite[Corollary~16.19]{RW}). 
\end{remark}

\subsection{Newton--Okounkov bodies associated to plabic graphs}
\label{subsec_prelim_NO}

In this subsection, we define Newton--Okounkov bodies using plabic graphs. 
We mainly refer to \cite{RW}. 

\begin{definition}
\label{def_PO}
We suppose that $\calO$ is a choice of orientation of each edge of a plabic graph $G$. 
An orientation $\calO$ is called a \emph{perfect orientation} 
if every internal white vertex has exactly one incoming arrow and every internal black vertex has exactly one outgoing arrow. 
We denote by $G_\calO$ a plabic graph oriented by $\calO$. 
Also, a perfect orientation $\calO$ is called \emph{acyclic} if $G_\calO$ is acyclic as a directed graph. 
We denote by $I_\calO$ the set of boundary vertices that are sources of $G_\calO$, 
and this set is called the \emph{source set}. 
\end{definition}

Then, we assign a variable $x_\lambda$ to each face $\lambda\in\widetilde{\calP}_G$. 
We call the set of such variables $\{x_\lambda \mid \lambda\in\widetilde{\calP}_G\}$ the \emph{$\calX$-cluster variables}. 

\begin{definition}
Let $J$ be a set of boundary vertices with $|J|=|I_\calO|$. 
A \emph{$J$-flow} $F$ from $I_\calO$ is a collection of 
self-avoiding, vertex disjoint directed paths in $G_\calO$
such that the sources of these paths are $I_\calO-(I_\calO\cap J)$ and the sinks are $J-(I_\calO\cap J)$. 
Each path $w$ in a $J$-flow divides the faces of $G$ into those which are on the left and those which are on the right of the path. 
We define the \emph{weight} $\wt(w)$ of each such path as the product 
of variable $x_\lambda$, where $\lambda$ ranges over all face labels to the left of the path. 
Furthermore,  we define the \emph{weight} $\wt(F)$ of a $J$-flow $F$ as the product of the weights of all paths in the $J$-flow. 
\end{definition}

It is known that any plabic graph of type $\pi_{k,n}$ has an acyclic perfect orientation $\calO$ and the order of $I_\calO$ is $k$ 
(see \cite[the discussion after Remark~11.6]{Pos}). 
Furthermore, there exists a unique acyclic perfect orientation $\calO$ such that $I_\calO=[k]=\{1,\dots, k\}$ \cite[Lemma~6.3 and Remark~6.4]{RW}. 
Up to rotations, such a unique acyclic perfect orientation takes the form as in Figure~\ref{fig_acyclic_PO} around any square face in $G$. 

\begin{figure}[H]
\begin{center}
\scalebox{1}{
\begin{tikzpicture}
\newcommand{\noderad}{0.1cm} 
\newcommand{\nodewidth}{0.03cm} 
\newcommand{\edgewidth}{0.03cm} 

\usetikzlibrary{decorations.markings}
\tikzset{sarrow/.style={decoration={markings,
                               mark=at position .65 with {\arrow[scale=1.3, red]{latex}}},
                               postaction={decorate}}}
                  
 \tikzset{ssarrow/.style={decoration={markings,
                               mark=at position .8 with {\arrow[scale=1.3, red]{latex}}},
                               postaction={decorate}}}

\foreach \n/\a/\b in {A/0/0, B/1/0, C/1/1, D/0/1} {
\coordinate (\n) at (\a,\b); } 
\foreach \n/\a in {A/225, B/315, C/45, D/135} {
\path (\n) ++(\a:0.7) coordinate (\n+); }

\draw[sarrow, line width=\edgewidth] (B)--(A); \draw[sarrow, line width=\edgewidth] (C)--(B); 
\draw[sarrow, line width=\edgewidth] (D)--(C); \draw[sarrow, line width=\edgewidth] (D)--(A); 

\draw[ssarrow, line width=\edgewidth] (A)--(A+); \draw[ssarrow, line width=\edgewidth] (B)--(B+); 
\draw[sarrow, line width=\edgewidth] (C+)--(C); \draw[sarrow, line width=\edgewidth] (D+)--(D); 

\filldraw [fill=black, line width=\nodewidth] (A) circle [radius=\noderad] ;
\filldraw [fill=white, line width=\nodewidth] (B) circle [radius=\noderad] ; 
\filldraw [fill=black, line width=\nodewidth] (C) circle [radius=\noderad] ;
\filldraw [fill=white, line width=\nodewidth] (D) circle [radius=\noderad] ; 
\end{tikzpicture}}
\end{center}
\caption{The acyclic perfect orientation around a square face}
\label{fig_acyclic_PO}
\end{figure}
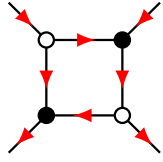

From now on, we will always choose the perfect orientation $\calO$ whose source set is $I_\calO=[k]$ 
for a plabic graph $G$ of type $\pi_{k,n}$. 
For a given $k$-subset $J\in{ [n] \choose k }$, we define the \emph{flow polynomial} 
\begin{equation}
\label{eq_flow_poly}
P^G_J\coloneqq \sum_F\wt(F), \quad \text{where $F$ ranges over all $J$-flows from $I_\calO$}.
\end{equation}
In this case, $x_\varnothing$ does not appear in $P^G_J$ for any $J$, 
because the face labeled by $\varnothing$ locates at the right of any path from $I_\calO$ to $J$. 

\begin{example}
\label{ex_Jflow}
Let us consider the rectangle plabic graph $G=G_{2,6}^{\rec}$ (see Figure~\ref{ex_plabic3}). 
The left of Figure~\ref{fig_Jflow} is the acyclic perfect orientation $\calO$ with the source set $I_\calO=\{1,2\}$. 
The middle and right of Figure~\ref{fig_Jflow} are the $J$-flows $J=\{3,5\}$. 
Thus, we have 
\[
\ytableausetup{boxsize=3pt}
P^G_{\{3,5\}}=x_{\ydiagram{4,4}}^2 x_{\ydiagram{3,3}}x_{\ydiagram{2,2}}x_{\ydiagram{4}}x_{\ydiagram{3}}
+x_{\ydiagram{4,4}}^2 x_{\ydiagram{3,3}}x_{\ydiagram{2,2}}x_{\ydiagram{4}}x_{\ydiagram{3}}x_{\ydiagram{2}}.
\]
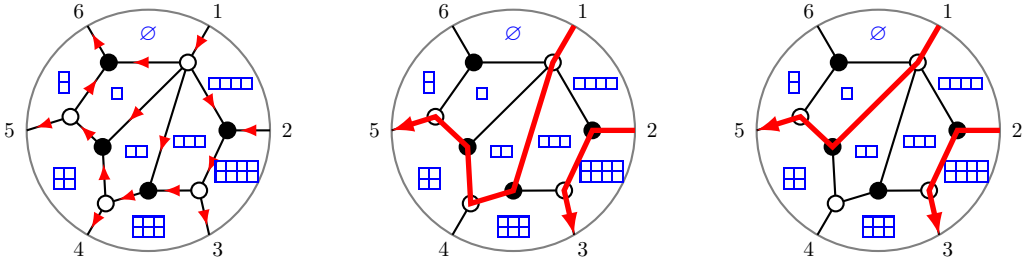
\begin{figure}[H]
\begin{center}
\scalebox{0.8}{
\begin{tikzpicture}[myarrow/.style={black, -latex}]
\newcommand{\boundaryrad}{2cm} 
\newcommand{\noderad}{0.13cm} 
\newcommand{\nodewidth}{0.035cm} 
\newcommand{\edgewidth}{0.035cm} 
\newcommand{\boundarylabel}{8pt} 
\newcommand{\arrowwidth}{0.05cm} 

\ytableausetup{boxsize=4.5pt}
\usetikzlibrary{decorations.markings}
\tikzset{sarrow/.style={decoration={markings, 
                               mark=at position .7 with {\arrow[scale=1.3, red]{latex}}},
                               postaction={decorate}}}                               
\tikzset{ssarrow/.style={decoration={markings, 
                               mark=at position .8 with {\arrow[scale=1.3, red]{latex}}},
                               postaction={decorate}}}
                               
\foreach \n/\a in {1/60, 2/0, 3/300, 4/240, 5/180, 6/120} {
\coordinate (B\n) at (\a:\boundaryrad); 
\coordinate (B\n+) at (\a:\boundaryrad+\boundarylabel); };

\node at (0,0) {
\begin{tikzpicture}
\draw (0,0) circle(\boundaryrad) [gray, line width=\edgewidth];
\foreach \n in {1,...,6} { \draw (B\n+) node {\small $\n$}; }
\foreach \n/\a/\r in {7/60/0.65, 8/0/0.65, 9/310/0.65, 10/270/0.5, 11/240/0.7, 12/200/0.4, 13/170/0.65, 14/120/0.65} {
\coordinate (B\n) at (\a:\r*\boundaryrad); }; 

\foreach \s/\t in {1/7, 2/8, 7/8, 8/9, 9/10, 10/11, 11/12, 12/13, 13/14, 7/10, 7/12, 7/14} {
\draw[sarrow, line width=\edgewidth] (B\s)--(B\t); 
};
\foreach \s/\t in {9/3, 11/4, 13/5, 14/6} {
\draw[ssarrow, line width=\edgewidth] (B\s)--(B\t); 
};

\foreach \x in {7,9,11,13} {
\filldraw [fill=white, line width=\edgewidth] (B\x) circle [radius=\noderad] ;}; 
\foreach \x in {8,10,12,14} {
\filldraw [fill=black, line width=\edgewidth] (B\x) circle [radius=\noderad] ; };
\foreach \n/\a/\r in {12/335/0.8, 23/270/0.8, 34/210/0.8, 45/150/0.8, 56/90/0.8, 16/30/0.78} {
\coordinate (V\n) at (\a:\r*\boundaryrad); }; 
\node[blue] at (V12) {\ydiagram{4,4}}; 
\node[blue] at (V23) {\ydiagram{3,3}}; 
\node[blue] at (V34) {\ydiagram{2,2}}; 
\node[blue] at (V45) {\ydiagram{1,1}}; 
\node[blue] at (V56) {$\varnothing$};
\node[blue] at (V16) {\ydiagram{4}};  
\foreach \n/\a/\r in {26/-15/0.35, 36/240/0.2, 46/130/0.4} {
\coordinate (V\n) at (\a:\r*\boundaryrad); }; 
\node[blue] at (V26) {\ydiagram{3}}; 
\node[blue] at (V36) {\ydiagram{2}}; 
\node[blue] at (V46) {\ydiagram{1}}; 
\end{tikzpicture}};

\node at (6,0) {
\begin{tikzpicture}[myarrow/.style={-latex}]
\draw (0,0) circle(\boundaryrad) [gray, line width=\edgewidth];
\foreach \n in {1,...,6} { \draw (B\n+) node {\small $\n$}; }
\foreach \n/\a/\r in {7/60/0.65, 8/0/0.65, 9/310/0.65, 10/270/0.5, 11/240/0.7, 12/200/0.4, 13/170/0.65, 14/120/0.65} {
\coordinate (B\n) at (\a:\r*\boundaryrad); }; 
\foreach \s/\t in {7/8, 8/9, 9/10, 10/11, 11/12, 12/13, 13/14, 14/7, 7/10, 7/12, 
1/7, 2/8, 3/9, 4/11, 5/13, 6/14} {
\draw[line width=\edgewidth] (B\s)--(B\t); 
};

\foreach \x in {7,9,11,13} {
\filldraw [fill=white, line width=\edgewidth] (B\x) circle [radius=\noderad] ;}; 
\foreach \x in {8,10,12,14} {
\filldraw [fill=black, line width=\edgewidth] (B\x) circle [radius=\noderad] ; };
\foreach \n/\a/\r in {12/335/0.8, 23/270/0.8, 34/210/0.8, 45/150/0.8, 56/90/0.8, 16/30/0.78} {
\coordinate (V\n) at (\a:\r*\boundaryrad); }; 
\node[blue] at (V12) {\ydiagram{4,4}}; 
\node[blue] at (V23) {\ydiagram{3,3}}; 
\node[blue] at (V34) {\ydiagram{2,2}}; 
\node[blue] at (V45) {\ydiagram{1,1}}; 
\node[blue] at (V56) {$\varnothing$};
\node[blue] at (V16) {\ydiagram{4}};  
\foreach \n/\a/\r in {26/-15/0.35, 36/240/0.2, 46/130/0.4} {
\coordinate (V\n) at (\a:\r*\boundaryrad); }; 
\node[blue] at (V26) {\ydiagram{3}}; 
\node[blue] at (V36) {\ydiagram{2}}; 
\node[blue] at (V46) {\ydiagram{1}}; 
\draw[myarrow, red, line width=\edgewidth+1.5] (B1)--(B7)--(B10)--(B11)--(B12)--(B13)--(B5); 
\draw[myarrow, red, line width=\edgewidth+1.5] (B2)--(B8)--(B9)--(B3); 
\end{tikzpicture}};

\node at (12,0) {
\begin{tikzpicture}[myarrow/.style={-latex}]
\draw (0,0) circle(\boundaryrad) [gray, line width=\edgewidth];
\foreach \n in {1,...,6} { \draw (B\n+) node {\small $\n$}; }
\foreach \n/\a/\r in {7/60/0.65, 8/0/0.65, 9/310/0.65, 10/270/0.5, 11/240/0.7, 12/200/0.4, 13/170/0.65, 14/120/0.65} {
\coordinate (B\n) at (\a:\r*\boundaryrad); }; 
\foreach \s/\t in {7/8, 8/9, 9/10, 10/11, 11/12, 12/13, 13/14, 14/7, 7/10, 7/12, 
1/7, 2/8, 3/9, 4/11, 5/13, 6/14} {
\draw[line width=\edgewidth] (B\s)--(B\t); 
};

\foreach \x in {7,9,11,13} {
\filldraw [fill=white, line width=\edgewidth] (B\x) circle [radius=\noderad] ;}; 
\foreach \x in {8,10,12,14} {
\filldraw [fill=black, line width=\edgewidth] (B\x) circle [radius=\noderad] ; };
\foreach \n/\a/\r in {12/335/0.8, 23/270/0.8, 34/210/0.8, 45/150/0.8, 56/90/0.8, 16/30/0.78} {
\coordinate (V\n) at (\a:\r*\boundaryrad); }; 
\node[blue] at (V12) {\ydiagram{4,4}}; 
\node[blue] at (V23) {\ydiagram{3,3}}; 
\node[blue] at (V34) {\ydiagram{2,2}}; 
\node[blue] at (V45) {\ydiagram{1,1}}; 
\node[blue] at (V56) {$\varnothing$};
\node[blue] at (V16) {\ydiagram{4}};  
\foreach \n/\a/\r in {26/-15/0.35, 36/240/0.2, 46/130/0.4} {
\coordinate (V\n) at (\a:\r*\boundaryrad); }; 
\node[blue] at (V26) {\ydiagram{3}}; 
\node[blue] at (V36) {\ydiagram{2}}; 
\node[blue] at (V46) {\ydiagram{1}}; 
\draw[myarrow, red, line width=\edgewidth+1.5] (B1)--(B7)--(B12)--(B13)--(B5); 
\draw[myarrow, red, line width=\edgewidth+1.5] (B2)--(B8)--(B9)--(B3); 
\end{tikzpicture}};
\end{tikzpicture}}
\end{center}
\caption{The acyclic perfect orientation $\calO$ of $G_{2,6}^{\rec}$ with $I_\calO=\{1,2\}$ (left), 
the $J$-flows for $J=\{3,5\}$ (middle, right)}
\label{fig_Jflow}
\end{figure}
\end{example}

\begin{definition}
\label{def_order_Young}
For the Young diagrams $\lambda_1,\lambda_2\in\widetilde{\calP}_{k,n}$, 
we define a partial order $\lambda_1\prec\lambda_2$ if $\lambda_1$ is contained in $\lambda_2$ 
when we justify the top-left corners of $\lambda_1,\lambda_2$. 
In other words, $\lambda_1\prec\lambda_2$ if and only if $J_1>J_2$, 
where $J_1,J_2\in{ [n] \choose k }$ are the $k$-subsets corresponding to $\lambda_1,\lambda_2$ as in Definition~\ref{def_Young}, 
and $<$ is a component-wise ordering, i.e., $j_i\ge j_i'$ holds for each $i=1,\dots,k$ 
where $J_1=\{j_1,\dots,j_k\}$ (resp. $J_2=\{j_1',\dots,j_k'\}$) with $1 \leq j_1 < \cdots < j_k \leq n$ (resp. $1 \leq j_1' < \cdots < j_k' \leq n$). 
With this convention, the minimal element in $\widetilde{\calP}_{k,n}$ is $\varnothing$. 
\end{definition}

We consider the lexicographic order $<_{\rm lex}$ on monomials in $\{x_\lambda\}_{\lambda\in\widetilde{\calP}_G}$. 
Namely, if $x_{\lambda_1}<_{\rm lex}\cdots<_{\rm lex} x_{\lambda_{s+1}}$ where $\lambda_i\in\widetilde{\calP}_G$ and $s=k(n-k)$, 
then $x_{\lambda_1}^{a_1}\cdots x_{\lambda_{s+1}}^{a_{s+1}}<_{\rm lex} x_{\lambda_1}^{b_1}\cdots x_{\lambda_{s+1}}^{b_{s+1}}$ 
if and only if there exists $k$ $(1\le k\le s+1)$ such that $a_i=b_i$ for $i=1,\dots,k-1$ and $a_k<b_k$. 
This determines a total order on $\ZZ^{\widetilde{\calP}_G}$. 

Using these notions, we define a valuation on $\CC(\XX_{k,n}){\setminus}\{0\}$ which was developed in \cite[Section~8]{RW}. 
By \cite[Propositions~6.8 and 7.6]{RW}, there is a certain way to write any polynomial $f$ in the Pl\"{u}cker coordinates for $\Gr(k,n)$ 
uniquely as a Laurent polynomial in $\{x_\lambda\}_{\lambda\in\widetilde{\calP}_G}$. 
Using this Laurent polynomial expression, we first define the valuation 
\[
\val_G:A_{k,n}{\setminus}\{0\}\rightarrow \ZZ^{\widetilde{\calP}_G}
\]
on $A_{k,n}{\setminus}\{0\}$ as follows. 
We suppose that $0\neq f\in A_{k,n}$ is described as $f=\sum_{\lambda\in\widetilde{\calP}_G}a_\lambda x_\lambda^{m_\lambda}$
with $a_\lambda\in\CC, \, m_\lambda\in\ZZ^{\widetilde{\calP}_G}$. 
Then, we define $\val_G(f)\coloneqq\minlex\{m_\lambda \mid a_\lambda\neq 0\}$ where the minimality is determined by the above lexicographic order. 
For $f/g\in\CC(\XX_{k,n}){\setminus}\{0\}$ with $f,g\in A_{n,k}$, we define $\val_G(f/g)=\val_G(f)-\val_G(g)$, 
and this extends $\val_G$ to the valuation on the field of rational functions of $\XX_{k,n}$: 
\[
\val_G:\CC(\XX_{k,n}){\setminus}\{0\}\rightarrow \ZZ^{\widetilde{\calP}_G}. 
\]
For a Pl\"{u}cker coordinate $\overline{p}_J$ with $J\in{ [n] \choose k }$, the expression in $\{x_\lambda\}_{\lambda\in\widetilde{\calP}_G}$ is given by the flow polynomial $P^G_J$. 
Thus, the valuation is given by $\val_G(\overline{p}_J)=\val_G(P^G_J)$, 
which is the exponent vector of the minimal term in $P^G_J$ with respect to the above lexicographic order. 
We call the $J$-flow corresponding to the minimal term in $P^G_J$ the \emph{strongly minimal flow} from $\{1,\dots, k \}$ to $J$, 
and denote by $F_{\min}$, that is, it satisfies $\val_G(\overline{p}_J)=\wt(F_{\min})$. 
As we mentioned, $x_\varnothing$ does not appear in $P^G_J$ for any $J$. 
Thus, we omit the coordinate corresponding to $x_\varnothing$ from the valuation, but we use the same notation by abuse of notation, 
that is, we let 
\[
\val_G:\CC(\XX_{k,n}){\setminus}\{0\}\rightarrow \ZZ^{\calP_G}. 
\] 
We note that for the Pl\"{u}cker coordinate $\overline{p}_{[k]}=\overline{p}_{\{1,\dots, k\}}$, 
we have $\val_G(\overline{p}_{[k]})=(0,\dots,0)$ 
because $P^G_{[k]}=1$ when we fix the perfect orientation $\calO$ with $I_\calO=[k]$. 
Thus, we have $\val_G(\overline{p}_J/\overline{p}_{[k]})=\val_G(\overline{p}_J)$ in particular. 

\medskip

Using this valuation, we now define the Newton--Okounkov body $\Delta_G$ which was introduced in \cite{Oko1,Oko2} 
and developed in several papers e.g., \cite{KK1,KK2,LM}. 

\begin{definition}
\label{def_NObody}
We take the divisor $D=\{\overline{p}_{[k]}= 0\}$ in $\XX_{k,n}$. 
Then $L_{r}\coloneqq \rmH^0(\XX_{k,n},\calO(rD))$ is generated by monomials with the form $M/(\overline{p}_{[k]})^r$, 
where $M$ is a monomial of degree $r$ in Pl\"{u}cker coordinates. 
The \emph{Newton--Okounkov body} $\Delta_G(D)$ associated to $\val_G$ and $D$ is defined by 
\[
\Delta_G(D)\coloneqq \overline{\Conv\left(\bigcup_{r=1}^\infty \frac{1}{r}\val_G(L_r)\right)}. 
\]
\end{definition}

The Newton--Okounkov body is also defined for other divisors in $\XX_{k,n}$. 
Since we use only the divisor $D=\{\overline{p}_{[k]}= 0\}$ in this paper, we simply denote $\Delta_G(D)$ by $\Delta_G$. 
Besides, there is another polytope called the \emph{superpotential polytope} $\Gamma_G$ associated to $G$, 
which is defined by inequalities obtained by the tropicalization of the \emph{superpotential} $W$ (see \cite[Section~10]{RW}). 
Here, the superpotential $W$ is a certain regular function defined on the side of the dual Grassmannian $\check{\XX}_{k,n}$ 
which parametrizes $(n-k)$-planes in $(\CC^n)^*$, see \cite[Section~6]{MR}, \cite[Section~10]{RW} for more details. 
In this article, we do not observe much about the superpotential polytope, 
but we note that it was shown in \cite[Theorem~16.18]{RW} that it coincides with the Newton--Okounkov polytope, i.e., $\Delta_G=\Gamma_G$, 
and $\Delta_G$ is a rational polytope. Thus, we call $\Delta_G$ the \emph{Newton--Okounkov polytope} associated to $G$. 
We note that the polytope $\Delta_G$ is not necessarily a lattice polytope, see e.g., \cite[Section~9]{RW}, \cite[Theorem~3]{Bos}. 

Although it is hard to compute Newton--Okounkov bodies in general, in our situation we know the explicit description of lattice points of $\Delta_G$.

\begin{theorem}[{\cite[Corollaries~16.10 and 16.19, Theorem~16.12]{RW}}]
\label{thm_NO=C}
Let $G$ be a plabic graph of type $\pi_{k,n}$. 
Then $\Delta_G$ has ${n \choose k}$ lattice points, and they are precisely the valuations $\val_G(\overline{p}_J)$ of Pl\"{u}cker coordinates. 
In particular, if $\Delta_G$ is a lattice polytope, 
then we have $\Delta_G=\Conv_G$, where $\Conv_G\coloneqq \Conv\left(\val_G(L_1)\right)=\Conv\{\val_G(\overline{p}_J)\mid J\in{[n] \choose k}\}$. 
\end{theorem}

\begin{example}
\label{ex_NOpolytope_Gr(2,6)}
We again consider the rectangle plabic graph $G=G_{2,6}^{\rec}$ (see Figure~\ref{ex_plabic3}). 
In Example~\ref{ex_Jflow}, we had the flow polynomial 
\begin{equation}
\label{eq_flowpolynomial}
\ytableausetup{boxsize=3pt}
P^G_{\{3,5\}}=x_{\ydiagram{4,4}}^2 x_{\ydiagram{3,3}}x_{\ydiagram{2,2}}x_{\ydiagram{4}}x_{\ydiagram{3}}
+x_{\ydiagram{4,4}}^2 x_{\ydiagram{3,3}}x_{\ydiagram{2,2}}x_{\ydiagram{4}}x_{\ydiagram{3}}x_{\ydiagram{2}}.
\end{equation}
We see that the lexicographically minimal term is the first monomial in \eqref{eq_flowpolynomial}, 
and hence the middle of Figure~\ref{fig_Jflow} shows the strongly minimal flow from $\{1,2\}$ to $\{3,5\}$. 
Thus, we have 
\[
\val_G(\overline{p}_{\{3,5\}})=(0,0,1,1,0,1,1,2), 
\]
where the components are arranged with the same order as Table~\ref{table_rec_Gr(2,6)}. 
By a computation similar to the above, we have the valuations as in Table~\ref{table_rec_Gr(2,6)} below. 

\medskip

\begin{table}[H]
\begin{center}
\ytableausetup{boxsize=3.5pt}

\scalebox{0.9}{
{\renewcommand\arraystretch{1.1}\tabcolsep= 10pt
\begin{tabular}{c|cccccccc}
&\ydiagram{1}&\ydiagram{2}&\ydiagram{3}&\ydiagram{4}&
\raise0.5ex\hbox{\ydiagram{1,1}}&\raise0.5ex\hbox{\ydiagram{2,2}}&\raise0.5ex\hbox{\ydiagram{3,3}}&\raise0.5ex\hbox{\ydiagram{4,4}} \\ \hline
$\{1,2\}$&0&0&0&0&0&0&0&0 \\ \hline
$\{1,3\}$&0&0&0&0&0&0&0&1 \\ \hline
$\{1,4\}$&0&0&0&0&0&0&1&1 \\ \hline
$\{1,5\}$&0&0&0&0&0&1&1&1 \\ \hline
$\{1,6\}$&0&0&0&0&1&1&1&1 \\ \hline
$\{2,3\}$&0&0&0&1&0&0&0&1 \\ \hline
$\{2,4\}$&0&0&0&1&0&0&1&1 \\ \hline
$\{2,5\}$&0&0&0&1&0&1&1&1 \\ \hline
$\{2,6\}$&0&0&0&1&1&1&1&1 \\ \hline
$\{3,4\}$&0&0&1&1&0&0&1&2 \\ \hline
$\{3,5\}$&0&0&1&1&0&1&1&2 \\ \hline
$\{3,6\}$&0&0&1&1&1&1&1&2 \\ \hline
$\{4,5\}$&0&1&1&1&0&1&2&2 \\ \hline
$\{4,6\}$&0&1&1&1&1&1&2&2 \\ \hline
$\{5,6\}$&1&1&1&1&1&2&2&2 
\end{tabular}}}
\end{center}
\caption{The valuation $\val_G(\overline{p}_J)$ of the Pl\"{u}cker coordinate $\overline{p}_J$ for any $J\in{[6] \choose 2}$ (we write only $J$ instead of $\overline{p}_J$ in the leftmost column).}
\label{table_rec_Gr(2,6)}
\end{table}

\medskip

As we will see in Proposition~\ref{prop_unimodular_GT}, 
the Newton--Okounkov polytope associated to a rectangle plabic graph is unimodularly equivalent to a Gelfand--Tsetlin polytope, 
which is a lattice polytope. 
Thus, by Theorem~\ref{thm_NO=C}, the Newton--Okounkov polytope associated to $G_{2,6}^{\rec}$ is 
the convex hull of the row vectors given in the above table. 
\end{example}

The Newton--Okounkov polytope $\Delta_G$ gives a toric degeneration of $\XX_{k,n}$ (after rescaling the polytopes) as follows. 

\begin{theorem}[{\cite[Corollary~17.11]{RW}}]
\label{thm_toric_degeneration}
There is a flat degeneration of $\XX_{k,n}$ to the normal projective toric variety 
associated to the polytope $r_G\Delta_G$. 
Here, $r_G$ is the minimal integer such that the dilated polytope $r_G\Delta_G$ 
has the ``integer decomposition property" $($see e.g., \cite{CHHH} for more details$)$. 
\end{theorem}

\begin{remark}
\label{rem_GTtableau}
We remark that we can compute the valuation of each Pl\"{u}cker coordinate for $G_{k,n}^{\rm rec}$ by using the grid labeled by the rectangle Young diagrams as follows. 
We first consider the $k\times (n-k)$ grid and label the $i\times j$ entry with the rectangle Young diagram of shape $i\times j$, 
and fill the numbers $1,\dots,n$ in the right and the bottom sides of the grid (see the left of Figure~\ref{fig_trip_rectangle}). 
We give the orientation to the grid (except the right and the bottom sides) so that vertical lines are oriented from the top to the bottom and horizontal lines are oriented from the right to the left. 
Then we consider a lattice path $p$ from $1$ to $n$, which determines the $k$-subset $J$ by taking the south steps labeling of the lattice path as
we did in Definition~\ref{def_Young}. 
We also draw lattice paths on the left-hand side of $p$ so that they are separated by ``one step", see the right of Figure~\ref{fig_trip_rectangle}. 
Note that such lattice paths can be read off from a \emph{Gelfand--Tsetlin tableau} which is a certain array of integers 
in the $k\times (n-k)$ grid (see \cite[Section~14]{RW} for more details). 
By \cite[Lemma~14.2]{RW}, we see that these lattice paths correspond to the strongly minimal flow on $G_{k,n}^{\rm rec}$ from $[k]$ to $J\in{[n] \choose k}$, 
and hence we can compute the valuation of each Pl\"{u}cker coordinate for $G_{k,n}^{\rm rec}$ using this labeled grid. 
(It might be easier to observe this claim if we write the empty diagram $\varnothing$ at the upper left margin.) 
We note that we will also use this grid description in the proof of Theorem~\ref{thm_unimodular_FFLV2}. 

\begin{figure}[H]
\begin{center}
\scalebox{0.9}{
\begin{tikzpicture}
\newcommand{\edgewidth}{0.035cm} 
\ytableausetup{boxsize=4.5pt}

\usetikzlibrary{decorations.markings}
\tikzset{sarrow/.style={decoration={markings,
                               mark=at position .65 with {\arrow[scale=1.3, red]{latex}}},
                               postaction={decorate}}}

\foreach \n/\a/\b in {00/0/0,10/1/0,20/2/0,30/3/0,40/4/0,01/0/1,11/1/1,21/2/1,31/3/1,41/4/1,02/0/2,12/1/2,22/2/2,32/3/2,42/4/2} {
\coordinate (V\n) at (\a,\b); };

\node at (0,0) {
\begin{tikzpicture}
\foreach \s/\t in{00/02,10/12,20/22,30/32,40/42,00/40,01/41,02/42} {
\draw[line width=\edgewidth] (V\s)--(V\t); }
\foreach \s/\t in{42/32,32/22,22/12,12/02,41/31,31/21,21/11,11/01,32/31,31/30,22/21,21/20,12/11,11/10,02/01,01/00} {
\draw[sarrow, line width=\edgewidth] (V\s)--(V\t);
} 
\node at (4.3,2) {$1$}; \node at (4.3,1) {$2$}; \node at (3,-0.3) {$3$}; 
\node at (2,-0.3) {$4$}; \node at (1,-0.3) {$5$}; \node at (0,-0.3) {$6$}; 
\foreach \n/\a/\b in {44/3.5/0.5,33/2.5/0.5,22/1.5/0.5,11/0.5/0.5,4/3.5/1.5,3/2.5/1.5,2/1.5/1.5,1/0.5/1.5} {
\coordinate (Y\n) at (\a,\b); }; 
\node[blue] at (Y44) {\ydiagram{4,4}}; \node[blue] at (Y33) {\ydiagram{3,3}}; 
\node[blue] at (Y22) {\ydiagram{2,2}}; \node[blue] at (Y11) {\ydiagram{1,1}}; 
\node[blue] at (Y4) {\ydiagram{4}};  \node[blue] at (Y3) {\ydiagram{3}}; 
\node[blue] at (Y2) {\ydiagram{2}}; \node[blue] at (Y1) {\ydiagram{1}}; 
\node[blue] at (-0.3,2.3) {$\varnothing$};
\end{tikzpicture}}; 

\node at (7,0) {
\begin{tikzpicture}[myarrow/.style={-latex}]
\foreach \s/\t in{00/02,10/12,20/22,30/32,40/42,00/40,01/41,02/42} {
\draw[line width=\edgewidth] (V\s)--(V\t); }
\node at (4.3,2) {$1$}; \node at (4.3,1) {$2$}; \node at (3,-0.3) {$3$}; 
\node at (2,-0.3) {$4$}; \node at (1,-0.3) {$5$}; \node at (0,-0.3) {$6$}; 
\foreach \n/\a/\b in {44/3.5/0.5,33/2.5/0.5,22/1.5/0.5,11/0.5/0.5,4/3.5/1.5,3/2.5/1.5,2/1.5/1.5,1/0.5/1.5} {
\coordinate (Y\n) at (\a,\b); }; 
\node[blue] at (Y44) {\ydiagram{4,4}}; \node[blue] at (Y33) {\ydiagram{3,3}}; 
\node[blue] at (Y22) {\ydiagram{2,2}}; \node[blue] at (Y11) {\ydiagram{1,1}}; 
\node[blue] at (Y4) {\ydiagram{4}};  \node[blue] at (Y3) {\ydiagram{3}}; 
\node[blue] at (Y2) {\ydiagram{2}}; \node[blue] at (Y1) {\ydiagram{1}}; 
\node[blue] at (-0.3,2.3) {$\varnothing$};
\draw[myarrow, red, line width=\edgewidth+1.5] (V42)--(V22)--(V21)--(V11)--(V10); 
\draw[myarrow, red, line width=\edgewidth+1.5] (V41)--(V31)--(V30); 
\end{tikzpicture}}; 

\end{tikzpicture}}
\end{center}
\caption{The $2\times 4$ grid labeled by rectangle Young diagrams (left), and the lattice paths corresponding to the strongly minimal flow on $G_{2,6}^{\rm rec}$ from $\{1,2\}$ to $\{3,5\}$, which is also given in Example~\ref{ex_Jflow} and \ref{ex_NOpolytope_Gr(2,6)}.}
\label{fig_trip_rectangle}
\end{figure}
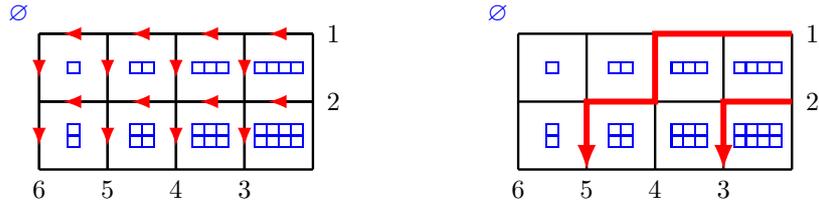
\end{remark}

\section{Preliminaries on combinatorial mutations of polytopes} 
\label{sec_prelim_mutation}

In this section, we review the notion of the \emph{combinatorial mutation of polytopes}, 
which arises from the new approach for classifying Fano manifolds that uses the mirror symmetry. 
In the context of mirror symmetry, a Fano manifold is expected to correspond to a certain Laurent polynomial (see \cite{CCGGK}). 
In general, there are several Laurent polynomials mirror to a given Fano manifold. 
To understand the relationship among such Laurent polynomials, the \emph{mutation of Laurent polynomials}, 
which is a birational transformation analogue to a cluster transformation, was introduced in \cite{GU}. 
On the other hand, a Laurent polynomial $f$ determines the Newton polytope $\Newt(f)$ as the convex hull of exponent vectors of monomials in $f$. 
When Laurent polynomials $f$ and $g$ are related by the mutation of Laurent polynomials, 
the mutation induces the operation between $\Newt(f)$ and $\Newt(g)$, 
which is exactly the combinatorial mutation of polytopes \cite{ACGK}. 

Following \cite{ACGK} and \cite{Hig}, we introduce the combinatorial mutations of polytopes. 
Let $P\subset N_\RR$ be a lattice polytope, and we assume that $P$ contains the origin ${\bf 0}$. 
We consider the family of polyhedra (not necessarily polytopes) which are of the following form: 
$$\calP \coloneqq \left\{ \bigcap_{v \in S} H_{v,\geq -1} \cap \bigcap_{v' \in T}H_{v',\geq 0} 
\subset N_\RR \mid S,T \subset M_\RR, \; |S|, |T| < \infty \right\},$$
where $H_{v,\geq h}$ denotes the affine half-space defined by $\{u \in N_\RR \mid \langle v,u \rangle \geq h\}$ for $v \in M_\RR$ and $h \in \RR$. 
A lattice polytope containing the origin of $N_\RR$ belongs to $\calP$ but the one not containing the origin does not belong to $\calP$ 
since one of the supporting hyperplanes of such polytope is of the form $H_{v, \geq k}$ for some $v \in M$ and $k >0$. 
Similarly, we define $\calQ$ by swapping the roles of $N_\RR$ and $M_\RR$. 
For a $P \in \calP$, we define the \emph{polar dual} $P^* \subset M_\RR$ of $P$ as 
\[
P^* \coloneqq \{v \in M_\RR\mid\langle v,u\rangle\ge-1 \text{ for all } u\in P\}\subset M_\RR. 
\]
Then we have $P^* \in \calQ$ and $(P^*)^*=P$ for any $P \in \calP$, see \cite[Theorem~9.1]{Sch}. 
Furthermore, for $Q \in \calQ$, we see that $Q$ is a polytope (i.e., $Q$ is bounded) if and only if $Q^* \in \calP$ is a polyhedron containing the origin in its interior (see \cite[Corollary 9.1a (iii)]{Sch}). 

We now define the two classes $\calP_N \subseteq \calP$ and $\calQ_M \subseteq \calQ$ as follows: 
\begin{itemize}
\item $\calQ_M$ is the set of full-dimensional rational polytopes in $M_\RR$ containing the origin. 
\item $\calP_N=\{Q^* \mid Q \in \calQ_M\}$. 
\end{itemize}

We remark that every polyhedron $P \in \calP_N$ contains the origin in its interior, because $P^*\in\calQ_M$. 
It is known (see \cite[Section 8.9]{Sch}) that any $P \in \calP_N$ has the unique decomposition $P=P^\prime+C$, 
where $P'$ is a rational polytope, $C$ is a rational polyhedral pointed cone and $+$ stands for the Minkowski sum. 
Here, a \emph{polyhedral cone} $C$ in $N_\RR$ (resp. $M_\RR$) is the cone $\Cone(V)$ generated by some finite set $V \subset N_\RR$ (resp. $V \subset M_\RR$) 
and $C$ is said to be \textit{rational} if $V \subset N$ (resp. $V \subset M$). 
A polyhedral cone $C$ is said to be \textit{pointed} if $C$ contains no $1$-dimensional linear subspace. 
Note that the pointedness follows from the full-dimensional condition for $\calQ_M$ (\cite[Corollary 9.1a (i)]{Sch}). 

\subsection{Combinatorial mutations in $N$} 

First, we prepare some notions used in the definition of the combinatorial mutations. 

We say that $w \in M$ is \emph{primitive} if $\Conv(\{{\bf 0},w\})$ contains no lattice point except for ${\bf 0}$ and $w$. 
Let $w\in M$ be a primitive lattice vector and consider the linear map $\langle w,-\rangle:N_\RR\rightarrow \RR$ determined by $w\in M$. 
We say that a lattice point $u\in N$ (resp., a subset $F \subset N_\RR$) is at \emph{height} $m$ with respect to $w$ if $\langle w,u\rangle=m$ (resp., $\langle w,u\rangle=m$ for any $u\in F$). 
Let $H_{w,h}=\{ u \in N_\RR \mid \langle w,u \rangle = h\}$ denote the hyperplane defined by $w$ with height $h$. 
We use the notation $w^\perp$ instead of $H_{w,0}$. 


We now define the \emph{combinatorial mutation} of $P \in \calP_N$. 
As mentioned above, $P$ can be decomposed as $P = P'+C$, where $P'$ is a rational polytope and $C$ is a rational polyhedral pointed cone. 
In what follows, we define the combinatorial mutations of a rational polytope $P'$ and a rational polyhedral pointed cone $C$. 
Then we define the combinatorial mutation of $P$ as the sum of those of $P'$ and $C$. 

\begin{definition}[{cf. \cite[Definition 5]{ACGK}, \cite[Section 2]{Hig}}]\label{def_mutation_polygonN}
Work with the same notation as above. 
Take $w \in M$ and a lattice polytope $F\subset N_\RR$ such that $F \subseteq w^\perp$. We call $F$ a \emph{factor}. 
\begin{itemize}
\item We assume that for each negative height $h<0$ there exists a possibly empty rational polytope $G_h \subset N_\RR$ satisfying 
\begin{equation}\label{condition_Gh_polytope}
\calV(P') \cap H_{w,h} \subseteq G_h+(-h)F\subseteq P' \cap H_{w,h}. 
\end{equation}
(We use the convention $A+\varnothing=\varnothing$ for any subset $A$.) 
Note that we may set $G_h=\emptyset$ if $\calV(P')  \cap H_{w,h}=\emptyset$. Since $\calV(P')$ is finite, we may consider only finitely many $h$'s. 
We define the \emph{combinatorial mutation} of $P'$ with respect to the vector $w$ and a factor $F$ as 
$$
\mut_w(P',F)\coloneqq\Conv\left(\bigcup_{h \leq -1}G_h \cup \bigcup_{h \geq 0}(P' \cap H_{w,h}+hF)\right), 
$$
which is actually a polytope since we may only consider finitely many $h$'s. 
\item Let $C$ be a rational polyhedral cone and let $V \subset N$ be a generator of $C$, i.e., $C=\Cone(V)$. 
Similarly to the above, we assume that for each negative height $h<0$, there exists a rational polyhedral cone $G_h\subset N_{\RR}$ satisfying 
\begin{equation}\label{condition_Gh_cone}
V  \cap H_{w,h} \subseteq G_h+(-h)F\subseteq C \cap H_{w,h}. 
\end{equation}
We define the \emph{combinatorial mutation} of $C$ with respect to the vector $w$ and a factor $F$ as 
$$
\mut_w(C,F)\coloneqq\Cone\left(\bigcup_{h \leq -1}G_h \cup \bigcup_{h \geq 0}(C \cap H_{w,h}+hF)\right), 
$$
which is actually a rational polyhedral cone. 
\item We say that $(P,w,F)=(P'+C,w,F)$ is \emph{well-defined} if $(P',w,F)$ satisfies \eqref{condition_Gh_polytope} and $(C,w,F)$ satisfies \eqref{condition_Gh_cone}. 
When $(P,w,F)$ is well-defined, we define the \emph{combinatorial mutation} of $P$ with respect to the vector $w$ and a factor $F$ as 
$$
\mut_w(P,F)\coloneqq\mut_w(P',F)+\mut_w(C,F). 
$$
\end{itemize}
\end{definition}
We remark that the mutation is independent of the choice of $\{G_h\}$ (see \cite[Proposition~1]{ACGK}), thus we omit $G_h$ from the notation. 
A translation of the factor $F$ does not affect the mutation, because 
for any $u\in N$ with $\langle w,u\rangle=0$ we have $\mut_w(P,F)\cong\mut_w(P,u+F)$, see \cite{ACGK} for more details. 
Also, the combinatorial mutation is an involution in the following sense; 
if $Q\coloneqq\mut_w(P,F)$, then we have $P=\mut_{(-w)}(Q,F)$ (see \cite[Lemma~2]{ACGK}). 

\subsection{Combinatorial mutations in $M$} 

We can consider the combinatorial mutation of a lattice polytope $P$ in terms of the polar dual $P^*$ of $P$. 

\begin{definition}[{cf. \cite[Section~3]{ACGK}}]
\label{def_mutation_polygonM}
Let the notation be the same as above. 
In particular, $w\in M$ is a primitive lattice vector, and $F$ is a lattice polytope such that $\langle w,u\rangle=0$ for any $u\in F$. 
We define the piecewise-linear map $\varphi_{w,F} : M_\RR\rightarrow M_\RR$ as 
\begin{equation}
\label{eq_piecewise-linear}
\varphi_{w,F}(v)\coloneqq v-v_{\rm min}w, \quad  
\text{where $v_{\min}\coloneqq \min\{\langle v, u\rangle\mid u\in F\}$.} 
\end{equation}
For a polytope $Q\subset M_\RR$, we say that $\varphi_{w,F}(Q)$ is the \emph{combinatorial mutation} of $Q$ 
with respect to $w$ and $F$ if $\varphi_{w,F}(Q)$ is convex. 
\end{definition}

We note that $\varphi_{-w,F} \circ \varphi_{w,F}$ is the identity, i.e., $\varphi_{w,F}^{-1}=\varphi_{-w,F}$, 
Also, the image $\varphi_{w,F}(Q)$ might not be convex in general since $\varphi_{w,F}$ is a piecewise-linear map. 

The two kinds of combinatorial mutations given in Definitions~\ref{def_mutation_polygonN} and \ref{def_mutation_polygonM} 
relate to each other in the following sense. 

\begin{proposition}[{e.g., \cite{ACGK}, \cite[Propositions~3.2 and 3.4]{Hig}, \cite[Proposition~6.6]{HN}}]
We fix a primitive vector $w \in M$ and a lattice polytope $F \subseteq w^\perp$. Let $Q \in \calQ_M$. 
Then $\mut_w(Q^*,F)$ is well-defined if and only if $\varphi_{w,F}(Q)$ is convex, in which case we have $\varphi_{w,F}(Q)=\mut_w(Q^*,F)^*$. 
\end{proposition}

We also remark that $\varphi_{w,F}$ preserves the Ehrhart series of $Q$ (cf. \cite[Proposition~4]{ACGK}). 
It is known that the leading coefficient of the Ehrhart series of $Q$ is the volume of $Q$, 
thus the volume of $Q$ coincides with that of $\varphi_{w,F}(Q)$. 

\section{Interpretation of tropicalized cluster mutations by combinatorial mutations} 
\label{sec_tropical_mutation}

We recall that we can obtain the Newton--Okounkov polytope $\Delta_G$ from a plabic graph $G$ of type $\pi_{k,n}$, 
see Definition~\ref{def_NObody}. 
Applying the operation (M1) in Definition~\ref{def_moves_plabic}, 
we have another plabic graph $G^\prime$ of type $\pi_{k,n}$ and the associated Newton--Okounkov polytope $\Delta_{G^\prime}$. 
Thus, it is natural to consider the relationship between $\Delta_G$ and $\Delta_{G^\prime}$. 
The desired answer was given in \cite[Section~11]{RW}, that is, they are related by the \emph{tropicalized cluster mutation} 
(see Definition~\ref{def_trop_mutation} below), 
which is obtained by tropicalizing the exchange relation used for defining a cluster algebra. 
In this section, we review this result from a viewpoint of combinatorial mutations. 
First, we introduce the tropicalized cluster mutation as follows, which is a special version of \cite[Definition~11.8]{RW}. 


\begin{definition}
\label{def_trop_mutation}
Let $G$ be a plabic graph of type $\pi_{k,n}$. 
In particular, the faces of $G$ are identified with $\widetilde{\calP}_G\eqqcolon \{\lambda_1, \dots, \lambda_s, \lambda_{s+1}=\varnothing\}$ 
where $s=k(n-k)$. 
We suppose that $\lambda_i\in\calP_G$ is a square face. 
Let $G^\prime$ be a plabic graph of type $\pi_{k,n}$ obtained by applying the square move (M1) to the face $\lambda_i$. 
We consider the piecewise-linear map $\Psi_{G,G^\prime}: \RR^{\calP_G}\rightarrow \RR^{\calP_{G^\prime}}$ defined by 
\begin{equation}
\label{eq_tropical_mutation1}
(v_{\lambda_1},\dots, v_{\lambda_{s}}) \mapsto 
(v_{\lambda_1},\dots, v_{\lambda_{i-1}}, v_{\lambda_i}^\prime, v_{\lambda_{i+1}},\dots, v_{\lambda_{s}}), 
\end{equation}
where 
\[
v_{\lambda_i}^\prime=-v_{\lambda_i}+\min\{v_{\lambda_a}+v_{\lambda_c}, v_{\lambda_b}+v_{\lambda_d}\}, 
\]
and $\lambda_a, \lambda_b, \lambda_c, \lambda_d$ are faces located cyclically around $\lambda_i$. 
If one of $\lambda_a, \lambda_b, \lambda_c, \lambda_d$ is the empty diagram, e.g. $\lambda_a=\varnothing$, then we define $v_{\lambda_a}=0$. 
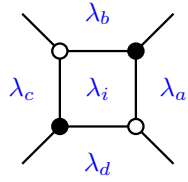
\begin{center}
\scalebox{1}{
\begin{tikzpicture}
\newcommand{\noderad}{0.1cm} 
\newcommand{\nodewidth}{0.03cm} 
\newcommand{\edgewidth}{0.03cm} 

\foreach \n/\a/\b in {A/0/0, B/1/0, C/1/1, D/0/1} {
\coordinate (\n) at (\a,\b); } 
\foreach \n/\a in {A/225, B/315, C/45, D/135} {
\path (\n) ++(\a:0.7) coordinate (\n+); }

\foreach \s/\t in {A/B, B/C, C/D, D/A, A/A+, B/B+, C/C+, D/D+}{
\draw[line width=\edgewidth] (\s)--(\t); }

\filldraw [fill=black, line width=\nodewidth] (A) circle [radius=\noderad] ;
\filldraw [fill=white, line width=\nodewidth] (B) circle [radius=\noderad] ; 
\filldraw [fill=black, line width=\nodewidth] (C) circle [radius=\noderad] ;
\filldraw [fill=white, line width=\nodewidth] (D) circle [radius=\noderad] ; 

\node[blue] (V) at (0.5,0.5) {\small $\lambda_i$} ;
\node[blue] (V1) at (1.5,0.5) {\small $\lambda_a$} ;\node[blue] (V2) at (0.5,1.5) {\small $\lambda_b$} ;
\node[blue] (V3) at (-0.5,0.5) {\small $\lambda_c$} ;\node[blue] (V4) at (0.5,-0.5) {\small $\lambda_d$} ;
\end{tikzpicture}
}
\end{center}
We call the map $\Psi_{G,G^\prime}$ the \emph{tropicalized cluster mutation} at $\lambda_i$. 
\end{definition}

By \cite[Theorem~13.1]{RW}, the map $\Psi_{G,G^\prime}$ satisfies 
\[\val_{G^\prime}(\overline{p}_J)=\Psi_{G,G^\prime}\big(\val_{G}(\overline{p}_J)\big)\]
for any Pl\"{u}cker coordinate $\overline{p}_J$ in $A_{k,n}$. 
Even though the polytope $\Conv_G$ is defined as the convex hull of these valuations (see Theorem~\ref{thm_NO=C}), 
the image $\Psi_{G,G^\prime}(\Conv_G)$ does not necessarily coincide with $\Conv_{G^\prime}$ (cf \cite[Remark~13.3]{RW}). 
Meanwhile, the tropicalized cluster mutation relates the Newton--Okounkov polytopes $\Delta_G$ and $\Delta_{G^\prime}$. 

\begin{theorem}[{\cite[Corollaries~11.16, 11.17, Theorem~16.18]{RW}}]
\label{thm_trop_map_NO}
Let the notation be the same as {\rm Definition~\ref{def_trop_mutation}}. 
The tropicalized cluster mutation $\Psi_{G,G^\prime}$ gives a bijection 
\[
\Psi_{G,G^\prime}: \Delta_G \rightarrow \Delta_{G^\prime}. 
\]
In particular, the numbers of lattice points of $\Delta_G$ and $\Delta_{G^\prime}$ are the same. 
\end{theorem}

We now reinterpret the map $\Psi_{G,G^\prime}$ in terms of combinatorial mutations. 
We note that a similar observation was already given in \cite{Cas} for the Newton polytope of the superpotential $W$, 
which is the polar dual of the Newton--Okounkov polytope. 
Thus, the following proposition is the dual of \cite[Proposition~3]{Cas}. 

\begin{proposition}
\label{prop_tropical=combmutation}
Let the notation be the same as {\rm Definition~\ref{def_trop_mutation}}. 
We define the lattice polytope $$F_{\lambda_i}\coloneqq\Conv\{\bfe_{\lambda_a}+\bfe_{\lambda_c}, \bfe_{\lambda_b}+\bfe_{\lambda_d}\},$$
where $\bfe_{\lambda_j}$ is the unit vector in $\RR^{\calP_G}$. Let $w\coloneqq -\bfe_{\lambda_i}$. 
Also, we define the linear map 
\[
\epsilon_{\lambda_i}: \RR^{\calP_G} \rightarrow \RR^{\calP_G} \quad
\Big(\hspace{3pt} (v_{\lambda_1},\dots, v_{\lambda_{s}}) \mapsto 
(v_{\lambda_1},\dots, v_{\lambda_{i-1}}, -v_{\lambda_i}, v_{\lambda_{i+1}},\dots, v_{\lambda_{s}}) \hspace{3pt} \Big), 
\]
which is clearly a unimodular transformation. 
Then, we have 
\begin{equation*}
\Psi_{G,G^\prime}=\varphi_{w, F_{\lambda_i}}\circ \epsilon_{\lambda_i}. 
\end{equation*}
Therefore, the Newton--Okounkov polytope $\Delta_{G^\prime}$ is the combinatorial mutation of $\epsilon_{\lambda_i}(\Delta_G)$ 
with respect to $w$ and $F_{\lambda_i}$. 
\end{proposition}

\begin{proof}
By definition, we see that $\langle w,u\rangle=0$ for any $u\in F_{\lambda_i}$. 
Thus, we can define the piecewise-linear map \eqref{eq_piecewise-linear} for $w=-\bfe_{\lambda_i}$ and $F_{\lambda_i}$. 
For $v=(v_{\lambda_1},\dots, v_{\lambda_{s}})\in\RR^{\calP_G}$, we have 
\[
v_{\min}\coloneqq \min\{\langle v, u\rangle\mid u\in F_{\lambda_i}\}=\min\{v_{\lambda_a}+v_{\lambda_c}, v_{\lambda_b}+v_{\lambda_d}\}.
\] 
Thus, we have 
\begin{equation}
\label{eq_proof_tropical}
\varphi_{w,F_{\lambda_i}}=v-v_{\min}w=v+\min\{v_{\lambda_a}+v_{\lambda_c}, v_{\lambda_b}+v_{\lambda_d}\}\bfe_{\lambda_i}. 
\end{equation}
Since the coordinate of \eqref{eq_proof_tropical} labeled by $\lambda_i$ is 
$v_{\lambda_i}+\min\{v_{\lambda_a}+v_{\lambda_c}, v_{\lambda_b}+v_{\lambda_d}\}$, 
we see that $\Psi_{G,G^\prime}=\varphi_{w, F_{\lambda_i}}\circ \epsilon_{\lambda_i}$ by definition.  
By Theorem~\ref{thm_trop_map_NO}, we know that 
$\Delta_{G^\prime}=\Psi_{G,G^\prime}(\Delta_G)=\varphi_{w, F_{\lambda_i}}\circ \epsilon_{\lambda_i}(\Delta_G)$ is a polytope, 
and hence $\Delta_{G^\prime}$ is the combinatorial mutation of $\epsilon_{\lambda_i}(\Delta_G)$. 
\end{proof}

\begin{remark}
\label{rem_tropicalized_mutation}
The tropicalized cluster mutation is defined for a general labeled seed mutation-equivalent to 
the one associated to a plabic graph of type $\pi_{k,n}$ as in \cite[Definition~11.8]{RW}, 
and Theorem~\ref{thm_trop_map_NO} holds for such a general labeled seed. 
Thus, we can extend Proposition~\ref{prop_tropical=combmutation} to such a general situation, 
see also Appendix~\ref{appendix_Gr(3,6)} for an example. 
\end{remark}

\begin{example}
\ytableausetup{boxsize=4.5pt}
We recall the plabic graphs given in Example~\ref{ex_plabic_mutation}. 
Their faces are labeled by Young diagrams as in the figure below. 
We denote the left plabic graph by $G$, and the right one by $G^\prime$. 
As we saw, applying the square move to the face of $G$ labeled by \ydiagram{2}, we have the plabic graph $G^\prime$, 
in which case the tropicalized cluster mutation $\Psi_{G,G^\prime}\colon\RR^8\rightarrow\RR^8$ is the piecewise-linear map given by 
\[
\ytableausetup{boxsize=3.5pt}
v_{\ydiagram{2}}\mapsto v_{\ydiagram{3,2}}\coloneqq -v_{\ydiagram{2}}+\min\{v_{\ydiagram{1}}+v_{\ydiagram{3,3}},v_{\ydiagram{2,2}}+v_{\ydiagram{3}}\}
\]
and other components are preserved. 
The map $\Psi_{G,G^\prime}$ satisfies $\val_{G^\prime}(\overline{p}_J)=\Psi_{G,G^\prime}\big(\val_{G}(\overline{p}_J)\big)$ 
for any Pl\"{u}cker coordinate $\overline{p}_J$ in $A_{2,6}$, and gives a bijection $\Psi_{G,G^\prime}: \Delta_G \rightarrow \Delta_{G^\prime}$. 

\ytableausetup{boxsize=4.5pt}
\begin{center}
\scalebox{0.8}{
\begin{tikzpicture}
\newcommand{\boundaryrad}{2cm} 
\newcommand{\noderad}{0.13cm} 
\newcommand{\nodewidth}{0.035cm} 
\newcommand{\edgewidth}{0.035cm} 
\newcommand{\boundarylabel}{8pt} 

\node at (0,0){
\begin{tikzpicture}
\draw (0,0) circle(\boundaryrad) [gray, line width=\edgewidth];
\foreach \n/\a in {1/60, 2/0, 3/300, 4/240, 5/180, 6/120} {
\coordinate (B\n) at (\a:\boundaryrad); 
\coordinate (B\n+) at (\a:\boundaryrad+\boundarylabel); };
\foreach \n in {1,...,6} { \draw (B\n+) node {\small $\n$}; }
\foreach \n/\a/\r in {7/60/0.65, 8/0/0.65, 9/310/0.65, 10/270/0.5, 11/240/0.7, 12/200/0.4, 13/170/0.65, 14/120/0.65} {
\coordinate (B\n) at (\a:\r*\boundaryrad); }; 
\foreach \s/\t in {7/8, 8/9, 9/10, 10/11, 11/12, 12/13, 13/14, 14/7, 7/10, 7/12, 
1/7, 2/8, 3/9, 4/11, 5/13, 6/14} {
\draw[line width=\edgewidth] (B\s)--(B\t); };
\foreach \x in {7,9,11,13} {
\filldraw [fill=white, line width=\edgewidth] (B\x) circle [radius=\noderad] ;}; 
\foreach \x in {8,10,12,14} {
\filldraw [fill=black, line width=\edgewidth] (B\x) circle [radius=\noderad] ; };
\foreach \n/\a/\r in {12/335/0.8, 23/270/0.8, 34/210/0.8, 45/150/0.8, 56/90/0.8, 16/30/0.78} {
\coordinate (V\n) at (\a:\r*\boundaryrad); }; 
\node[blue] at (V12) {\ydiagram{4,4}}; 
\node[blue] at (V23) {\ydiagram{3,3}}; 
\node[blue] at (V34) {\ydiagram{2,2}}; 
\node[blue] at (V45) {\ydiagram{1,1}}; 
\node[blue] at (V56) {$\varnothing$};
\node[blue] at (V16) {\ydiagram{4}};  
\foreach \n/\a/\r in {26/-15/0.35, 36/240/0.2, 46/130/0.4} {
\coordinate (V\n) at (\a:\r*\boundaryrad); }; 
\node[blue] at (V26) {\ydiagram{3}}; 
\node[blue] at (V36) {\ydiagram{2}}; 
\node[blue] at (V46) {\ydiagram{1}}; 
\end{tikzpicture}};

\draw[->, line width=0.03cm] (3,0)--(5.5,0) node[midway,xshift=0cm,yshift=0.35cm] {the square move} 
node[midway,xshift=0cm,yshift=-0.35cm] {at \ydiagram{2}}; 

\node at (8.5,0){
\begin{tikzpicture}
\draw (0,0) circle(\boundaryrad) [gray, line width=\edgewidth];
\foreach \n/\a in {1/60, 2/0, 3/300, 4/240, 5/180, 6/120} {
\coordinate (B\n) at (\a:\boundaryrad); 
\coordinate (B\n+) at (\a:\boundaryrad+\boundarylabel); };
\foreach \n in {1,...,6} { \draw (B\n+) node {\small $\n$}; }
\foreach \n/\a/\r in {7/60/0.65, 8/0/0.65, 9/310/0.65, 10/240/0.65, 11/170/0.65, 12/120/0.65, 13/0/0} {
\coordinate (B\n) at (\a:\r*\boundaryrad); }; 
\foreach \s/\t in {7/8, 8/9, 9/10, 10/11, 11/12, 12/7, 7/13, 9/13, 11/13, 
1/7, 2/8, 3/9, 4/10, 5/11, 6/12} {
\draw[line width=\edgewidth] (B\s)--(B\t); };
\foreach \x in {7,9,11} {
\filldraw [fill=white, line width=\edgewidth] (B\x) circle [radius=\noderad] ;}; 
\foreach \x in {8,10,12,13} {
\filldraw [fill=black, line width=\edgewidth] (B\x) circle [radius=\noderad] ; };
\foreach \n/\a/\r in {12/335/0.8, 23/270/0.8, 34/210/0.8, 45/150/0.8, 56/90/0.8, 16/30/0.78} {
\coordinate (V\n) at (\a:\r*\boundaryrad); }; 
\node[blue] at (V12) {\ydiagram{4,4}}; 
\node[blue] at (V23) {\ydiagram{3,3}}; 
\node[blue] at (V34) {\ydiagram{2,2}}; 
\node[blue] at (V45) {\ydiagram{1,1}}; 
\node[blue] at (V56) {$\varnothing$};
\node[blue] at (V16) {\ydiagram{4}};  
\foreach \n/\a/\r in {26/0/0.35, 24/240/0.3, 46/120/0.35} {
\coordinate (V\n) at (\a:\r*\boundaryrad); }; 
\node[blue] at (V26) {\ydiagram{3}}; 
\node[blue] at (V24) {\ydiagram{3,2}}; 
\node[blue] at (V46) {\ydiagram{1}}; 
\end{tikzpicture}};

\end{tikzpicture}
}
\end{center}

\end{example}

\begin{example}
\label{ex_not_convex}
\ytableausetup{boxsize=4.5pt}
By Proposition~\ref{prop_tropical=combmutation}, the composite $\varphi_{w, F_{\lambda_i}}\circ \epsilon_{\lambda_i}$ 
gives a bijection from $\Delta_G$ to $\Delta_{G^\prime}$. 
If we drop $\epsilon_{\lambda_i}$ from this map, then the image $\varphi_{w, F_{\lambda_i}}(\Delta_G)$ is not a polytope in general. 
For example, let us recall Example~\ref{ex_NOpolytope_Gr(2,6)}. Let $\lambda_i=\raise0.15ex\hbox{\ydiagram{1}}$. 
Then $\lambda_a=\varnothing$, $\lambda_b=\raise0.6ex\hbox{\ydiagram{1,1}}$, $\lambda_c=\raise0.6ex\hbox{\ydiagram{2,2}}$ and 
$\lambda_d=\raise0.15ex\hbox{\ydiagram{2}}$. 
Let $\varphi = \varphi_{w,F_{\lambda_i}}$. Now, we apply $\varphi$ to the vertices of $\Delta_G$. Then their images become as follows: 
\begin{center}
\ytableausetup{boxsize=3.5pt}
\scalebox{0.9}{
{\renewcommand\arraystretch{1.1}\tabcolsep= 10pt
\begin{tabular}{c|cccccccc}
&\ydiagram{1}&\ydiagram{2}&\ydiagram{3}&\ydiagram{4}&
\raise0.5ex\hbox{\ydiagram{1,1}}&\raise0.5ex\hbox{\ydiagram{2,2}}&\raise0.5ex\hbox{\ydiagram{3,3}}&\raise0.5ex\hbox{\ydiagram{4,4}} \\ \hline
$\{1,2\}$&0&0&0&0&0&0&0&0 \\ \hline
$\{1,3\}$&0&0&0&0&0&0&0&1 \\ \hline
$\{1,4\}$&0&0&0&0&0&0&1&1 \\ \hline
$\{1,5\}$&0&0&0&0&0&1&1&1 \\ \hline
$\{1,6\}$&1&0&0&0&1&1&1&1 \\ \hline
$\{2,3\}$&0&0&0&1&0&0&0&1 \\ \hline
$\{2,4\}$&0&0&0&1&0&0&1&1 \\ \hline
$\{2,5\}$&0&0&0&1&0&1&1&1 \\ \hline
$\{2,6\}$&1&0&0&1&1&1&1&1 \\ \hline
$\{3,4\}$&0&0&1&1&0&0&1&2 \\ \hline
$\{3,5\}$&0&0&1&1&0&1&1&2 \\ \hline
$\{3,6\}$&1&0&1&1&1&1&1&2 \\ \hline
$\{4,5\}$&1&1&1&1&0&1&2&2 \\ \hline
$\{4,6\}$&1&1&1&1&1&1&2&2 \\ \hline
$\{5,6\}$&3&1&1&1&1&2&2&2 
\end{tabular}}}
\end{center}
Let $\Delta_{G'}$ be the convex hull of those points. By using {\tt Magma}, we can compute the volume of $\Delta_{G'}$, which is $28/8!$. 
On the other hand, the volume of $\Delta_G$ is $14/8!$. Since $\varphi$ preserves the volume, these imply that $\Delta_{G'}$ is not the image of $\Delta_G$ by $\varphi$. 
This phenomenon happens since $\varphi(\Delta_G)$ is not convex. Indeed, consider the hyperplane $H=\{ v \in \RR^{\calP_G} \mid v_{\lambda_a}+v_{\lambda_c}=v_{\lambda_b}+v_{\lambda_d}\}$. 
Since $\varphi$ is just a linear map on $H$, we see that $\varphi(\Delta_G \cap H)$ is equal to the convex hull of the above points lying on $H$, 
which are those corresponding to \[\{1,2\}, \{1,3\}, \{1,4\}, \{1,6\}, \{2,3\}, \{2,4\}, \{2,6\}, \{3,4\}, \{3,6\}, \{4,5\},\{5,6\}.\] 
On the other hand, the vertices $\val_G(\overline{p}_{\{3,5\}})$ and $\val_G(\overline{p}_{\{4,6\}})$ lie on the different side with respect to $H$. 
Let us consider the midpoint $v'$ of $\varphi(\val_G(\overline{p}_{\{3,5\}}))$ and $\varphi(\val_G(\overline{p}_{\{4,6\}}))$. Then 
\[v'=(1/2,1/2,1,1,1/2,1,3/2,2)\] 
and $v'$ lies on $H$. If $\varphi(\Delta_G)$ is convex, then $v'$ is contained in $\varphi(\Delta_G \cap H)$. 
Due to the last coordinate of $v'$, we see that $v'$ should be the convex combination of the points 
\[\varphi(\val_G(\overline{p}_{\{3,4\}})), \varphi(\val_G(\overline{p}_{\{3,6\}})), \varphi(\val_G(\overline{p}_{\{4,5\}}))\text{ and }\varphi(\val_G(\overline{p}_{\{5,6\}})).\] 
Suppose that \[v'=c_1 \varphi(\val_G(\overline{p}_{\{3,4\}}))+ c_2 \varphi(\val_G(\overline{p}_{\{3,6\}})) + c_3 \varphi(\val_G(\overline{p}_{\{4,5\}})) + c_4 \varphi(\val_G(\overline{p}_{\{5,6\}})),\]
where $c_1,\dots,c_4 \geq 0$ and $c_1+\cdots+c_4=1$. 
Then we have $c_2+c_3+3c_4=1/2$ and $c_2+c_3+2c_4=1$. These equations conclude $c_4=-1/2$, but this contradicts $c_4\ge0$. 
\end{example}


\section{Combinatorial mutations associated to posets of Young diagrams} 
\label{sec_mutation_youngposet}

We first recall the order polytopes and chain polytopes which were originally introduced in \cite{S86}. 
Let $\Pi$ be a poset equipped with a partial order $\prec$. 
For $a,b \in \Pi$, we say that $a$ \emph{covers} $b$ if $b \prec a$ and there is no $a'$ such that $a' \neq a, b$ and $b \prec a' \prec a$. 
We say that $\alpha \subseteq \Pi$ is a \emph{poset filter} if $\alpha$ satisfies the condition: 
$$a \in \alpha \text{ and } a \prec b \;\Longrightarrow b \in \alpha. $$
We regard an empty set as a poset filter. 
Also, we say that $A \subseteq \Pi$ is an {\em antichain} if $a \not\prec b$ and $b \not\prec a$ hold for any $a,b \in A$ with $a \neq b$. 
We regard an empty set as an antichain. 

\begin{definition}[cf. \cite{S86}]
With the settings as above, we define two convex polytopes as follows: 
\begin{align*}
\calO(\Pi)&=\{ (x_a)_{a \in \Pi} \in \RR^\Pi \mid x_a \leq x_b \text{ if }a \preceq b \text{ in }\Pi, \; 0 \leq x_a \leq 1 \text{\hspace{5pt}for all } a \in \Pi\}, \\
\calC(\Pi)&=\{ (x_a)_{a \in \Pi} \in \RR^\Pi \mid 
x_{a_{i_1}}+\cdots+x_{a_{i_\ell}} \leq 1 \text{ if }a_{i_1} \prec \cdots \prec a_{i_\ell} \text{ in }\Pi, \; x_a \geq 0 \text{\hspace{5pt}for all } a \in \Pi\}. 
\end{align*}
We call $\calO(\Pi)$ the \emph{order polytope} of $\Pi$, and call $\calC(\Pi)$ the \emph{chain polytope} of $\Pi$. 
\end{definition}

We here collect some important properties of these polytopes which were shown in \cite{S86}. 
First, both $\calO(\Pi)$ and $\calC(\Pi)$ are $(0,1)$-polytopes, that is, polytopes whose vertices are $(0,1)$-vectors. 
There is a one-to-one correspondence between the vertices of $\calO(\Pi)$ (resp., $\calC(\Pi)$) and the poset filters (resp., the antichains) of $\Pi$. 
More precisely, a $(0,1)$-vector $(x_a)_{a \in \Pi}$ is a vertex of $\calO(\Pi)$ (resp., $\calC(\Pi)$) 
if and only if $\{a \mid x_a =1\}$ is a poset filter (resp., an antichain) of $\Pi$. 
In particular, both $\calO(\Pi)$ and $\calC(\Pi)$ contain the origin in their boundaries. 
Note that the poset filters one-to-one correspond to the antichains, 
thus the number of vertices of $\calO(\Pi)$ coincides with that of $\calC(\Pi)$. 
Since these are $(0,1)$-polytopes, we conclude that $\# (\calO(\Pi) \cap \ZZ^\Pi)=\# (\calC(\Pi) \cap \ZZ^\Pi)$. 

We remark that $\calO(\Pi)$ and $\calC(\Pi)$ are not necessarily unimodularly equivalent in general. 
The necessary and sufficient condition for $\calO(\Pi)$ and $\calC(\Pi)$ to be unimodularly equivalent is given in \cite[Theorem 2.1]{HL}. 
Moreover, we have $\# (m\calO(\Pi) \cap \ZZ^\Pi)=\# (m\calC(\Pi) \cap \ZZ^\Pi)$ for any positive integer $m$ \cite{S86}, 
which is proved by using the \textit{transfer map} $\phi:\RR^\Pi \rightarrow \RR^\Pi$ defined as follows: 
\[
\phi((x_a)_{a \in \Pi})=(\phi(x_a))_{a \in \Pi}, \text{ where }\phi(x_a):=\min\{x_a-x_{b} \mid \text{$a$ covers $b$} \}.
\]
This map $\phi$ gives a bijection between $\calO(\Pi)$ and $\calC(\Pi)$, precisely $\phi\big(\calO(\Pi)\big)=\calC(\Pi)$ (see \cite[Theorem 3.2]{S86}). 

\medskip

In what follows, we consider the order polytope and the chain polytope for a special class of posets, denoted by $\Pi_{k,n}$, 
whose Hasse diagram takes the form as in Figure~\ref{hasse_poset_kn}. 
In this case, the polytope $\calO(\Pi_{k,n})$ is called the  \emph{Gelfand--Tsetlin $($GT$)$ polytope}, 
and $\calC(\Pi_{k,n})$ is called the \emph{Feigin--Fourier--Littelmann--Vinberg $($FFLV$)$ polytope} after the work \cite{GT,FFL,Vin}, 
see Subsection~\ref{subsec_intro_summary} for more details. 

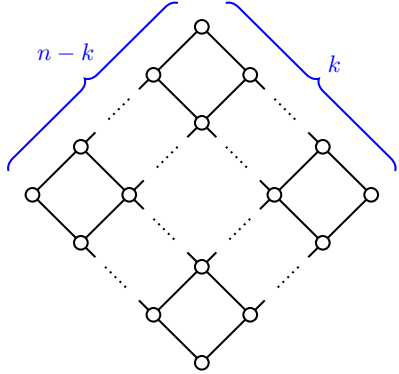
\begin{figure}[H]
\begin{center}
\scalebox{0.9}{
\begin{tikzpicture}
\newcommand{\noderad}{0.1cm} 
\newcommand{\edgewidth}{0.03cm} 

\foreach \n/\a/\b in {11/0/0, 12/-45/1, 13/-45/2.5, 14/-45/3.5} {
\coordinate (V\n) at (\a:\b); }; 
\foreach \n/\m/\a/\b in {21/11/225/1, 22/12/225/1, 23/13/225/1, 24/14/225/1, 
31/11/225/2.5, 32/12/225/2.5, 33/13/225/2.5, 34/14/225/2.5, 
41/11/225/3.5, 42/12/225/3.5, 43/13/225/3.5, 44/14/225/3.5} {
\path (V\m) ++(\a:\b) coordinate (V\n); }; 

\foreach \s/\t in {11/12, 13/14, 21/22, 23/24, 31/32, 33/34, 41/42, 43/44, 
11/21, 31/41, 12/22, 32/42,13/23, 33/43,14/24, 34/44} {
\draw[line width=\edgewidth] (V\s)--(V\t); };
\foreach \n/\a in {21/225, 22/225, 23/225, 24/225, 31/45, 32/45, 33/45, 34/45, 
12/-45, 22/-45, 32/-45, 42/-45, 13/135, 23/135, 33/135, 43/135} {
\draw[line width=\edgewidth] (V\n)-- ++(\a:0.3); };
\foreach \n/\a in {21/225, 22/225, 23/225, 24/225, 12/-45, 22/-45, 32/-45, 42/-45} {
\path (V\n) ++(\a:0.5) coordinate (V\n+); 
\draw[line width=\edgewidth, dotted] (V\n+)-- ++(\a:0.5); };

\foreach \n in {11,12,13,14, 21, 22, 23, 24, 31, 32, 33, 34, 41, 42, 43, 44} {
\filldraw [fill=white, line width=\edgewidth] (V\n) circle [radius=\noderad] ; };

\draw [line width=0.03cm, decorate, blue, decoration={brace, amplitude=5pt}] (0.35,0.35) -- ++(-45:3.5) node[midway,xshift=0.35cm,yshift=0.35cm,blue] {\small $k$}; 
\draw [line width=0.03cm, decorate, blue, decoration={brace, mirror, amplitude=5pt}] (-0.35,0.35) -- ++(225:3.5) node[midway,xshift=-0.4cm,yshift=0.5cm,blue] {\small $n-k$}; 
\end{tikzpicture}}
\end{center}
\caption{The Hasse diagram of the poset $\Pi_{k,n}$}
\label{hasse_poset_kn}
\end{figure}

We here recall a partial order for Young diagrams in $\widetilde{\calP}_{k,n}$ defined in Definition~\ref{def_order_Young}. 
For the rectangle plabic graph $G^{\rec}_{k,n}$, we see that 
the poset structure of Young diagrams in $\calP_{G^{\rec}_{k,n}}=\widetilde{\calP}_{G^{\rec}_{k,n}}{\setminus}\{\varnothing\}$ coincides with $\Pi_{k,n}$. 
Thus, in the following, we label each component of an element in $\Pi_{k,n}$ by a Young diagram in $\calP_{G^{\rec}_{k,n}}$ along this partial order. 

\begin{example}
\label{ex_poset_Gr(2,6)}
We consider the rectangle plabic graph $G^{\rec}_{2,6}$ given in Figure~\ref{ex_plabic3}. 
In this case, we have 
\[
\ytableausetup{boxsize=4.5pt}
\widetilde{\calP}_{G^{\rec}_{2,6}}=\left\{\varnothing \hspace{2pt},\hspace{3pt} \ydiagram{1} \hspace{2pt},\hspace{3pt} \ydiagram{2} \hspace{2pt},\hspace{3pt} 
\ydiagram{3} \hspace{2pt},\hspace{3pt} \ydiagram{4} \hspace{2pt},\hspace{3pt} 
\raise0.5ex\hbox{\ydiagram{1,1}} \hspace{2pt},\hspace{3pt} \raise0.5ex\hbox{\ydiagram{2,2}} \hspace{2pt},\hspace{3pt} 
\raise0.5ex\hbox{\ydiagram{3,3}} \hspace{2pt},\hspace{3pt} \raise0.5ex\hbox{\ydiagram{4,4}}\,\right\}.
\]
The Hasse diagram of the poset $\calP_{G^{\rec}_{2,6}}$, which coincides with the one of $\Pi_{2,6}$, is the following.

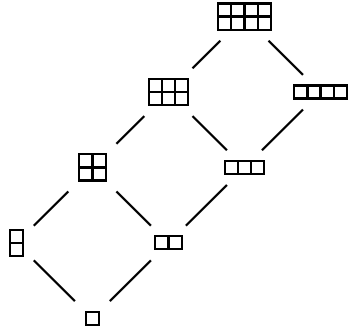
\begin{center}
\scalebox{1}{
\begin{tikzpicture}
\newcommand{\edgewidth}{0.03cm} 
\ytableausetup{boxsize=4.5pt}

\node (V1) at (0,0) {\ydiagram{1}};  \node (V2) at (-1,1) {\ydiagram{1,1}};  
\node (V3) at (1,1) {\ydiagram{2}};  \node (V4) at (0,2) {\ydiagram{2,2}};  
\node (V5) at (2,2) {\ydiagram{3}};  \node (V6) at (1,3) {\ydiagram{3,3}}; 
\node (V7) at (3,3) {\ydiagram{4}};  \node (V8) at (2,4) {\ydiagram{4,4}};  

\foreach \s/\t in {1/2, 1/3, 2/4, 3/4, 3/5, 4/6, 5/6, 5/7, 6/8, 7/8} {
\draw[line width=\edgewidth] (V\s)--(V\t); };

\end{tikzpicture}}
\end{center}
\end{example}

Then we note the relationship between the GT polytope $\calO(\Pi_{k,n})$ and 
the Newton--Okounkov polytope associated to the rectangle plabic graph $G^{\rm rec}_{k,n}$. 

\begin{proposition}[{\cite[Lemma~16.2]{RW}}]
\label{prop_unimodular_GT}
Let $G=G^{\rm rec}_{k,n}$. 
We consider the linear map 
$f\colon\RR^{\calP_G}\rightarrow\RR^{\calP_G}$ given by 
\[
v_{i\times j}\mapsto x_{i\times j}\coloneqq v_{i\times j}-v_{(i-1)\times (j-1)},
\]
where $v_{i\times j}$ is the component of $v\in\RR^{\calP_G}$ corresponding to 
the rectangle Young diagram of shape $i\times j$, and $v_{i\times j}\mapsto v_{i\times j}$ if $i=1$ or $j=1$. 
Then $f$ is unimodular, and we have $f(\Delta_G)=\calO(\Pi_{k,n})$, that is, 
the Newton--Okounkov polytope $\Delta_G$ is unimodularly equivalent to the GT polytope $\calO(\Pi_{k,n})$. 
In particular, $\Delta_G$ is a lattice polytope. 
\end{proposition}

We note that by the map $f$ in Proposition~\ref{prop_unimodular_GT}, $v_{i\times j}$ is sent to the component of an element in $\calO(\Pi_{k,n})$ 
labeled by $i\times j$. Thus we may write $\calO(\Pi_{k,n})$ as 
\[
f(\Delta_{G^{\rm rec}_{k,n}})=
\calO(\Pi_{k,n})=\{ (x_\lambda) \in \RR^{\calP_{G^{\rm rec}_{k,n}}} 
\mid x_{\lambda} \leq x_{\lambda^\prime} \text{ if }\lambda \preceq \lambda^\prime \text{ in }\calP_{G^{\rm rec}_{k,n}}, 
\; 0 \leq x_\lambda \leq 1 \text{\hspace{5pt}for all } \lambda \in \calP_{G^{\rm rec}_{k,n}}\}.
\]

\begin{example}
\ytableausetup{boxsize=3.5pt}
We know that the Newton--Okounkov polytope $\Delta_{G^{\rm rec}_{2,6}}$ is the convex hull of the row vectors, 
which are the valuations of the Pl\"{u}cker coordinates, in the table of Example~\ref{ex_NOpolytope_Gr(2,6)}. 
Let $G=G^{\rm rec}_{2,6}$. 
We apply the transformation $f$ in Proposition~\ref{prop_unimodular_GT}, which is described as 
\[
v_{\ydiagram{4,4}}  \mapsto v_{\ydiagram{4,4}} -v_{\ydiagram{3}}, \hspace{15pt}
v_{\ydiagram{3,3}}  \mapsto v_{\ydiagram{3,3}} -v_{\ydiagram{2}}, \hspace{15pt}
v_{\ydiagram{2,2}}  \mapsto v_{\ydiagram{2,2}} -v_{\ydiagram{1}}, 
\]
\[
v_\lambda  \mapsto v_\lambda 
\quad \text{for } \lambda\in\calP_G{\setminus}\{\raise0.6ex\hbox{\ydiagram{4,4}}, \raise0.6ex\hbox{\ydiagram{3,3}}, \raise0.6ex\hbox{\ydiagram{2,2}}\}, 
\]
to the valuations. 
Then we have the GT polytope $\calO(\Pi_{2,6})$ as the convex hull of $f(\val_G(\overline{p}_J))$ for all $J\in{[6] \choose 2}$ 
by Proposition~\ref{prop_unimodular_GT}. 
\end{example}

Now, for a plabic graph $G$ of type $\pi_{k,n}$, we introduce the new plabic graph $G^\vee$, which will be called the \emph{dual} of $G$. 
We denote by $G^*$ the plabic graph obtained by replacing black vertices with white ones and vice versa in $G$. 
Then we define the dual $G^\vee$ of $G$ by changing the numbering of the boundary vertices of $G^*$ as $r\mapsto r+k$ (modulo $n$) 
for any $r=1,\dots,n$. We see that $G^\vee$ is a plabic graph of type $\pi_{k,n}$. 

\begin{observation}
\label{obs_dualplabic}
We now consider the rectangle plabic graph $G^{\rm rec}_{n-k,n}$, and its dual $(G^{\rm rec}_{n-k,n})^\vee$. 
We note that $(G^{\rm rec}_{n-k,n})^\vee$ is of type $\pi_{k,n}$, but it is no longer rectangle. 
Let $\lambda_{i\times j}$ be the rectangle Young diagram of shape $i\times j$, where $1\le i\le k, 1\le j\le n-k$. 
Assigning the boundary vertex $G^{\rm rec}_{n-k,n}$ numbered by $i$ to that of $(G^{\rm rec}_{n-k,n})^\vee$ numbered by $i+k$, 
we have a natural bijection between faces of $G^{\rm rec}_{n-k,n}$ and those of $(G^{\rm rec}_{n-k,n})^\vee$. 
With this bijection, we see that the face of $G^{\rm rec}_{n-k,n}$ labeled by $\lambda_{j\times i}$ 
corresponds to the face of $(G^{\rm rec}_{n-k,n})^\vee$ labeled by $\lambda_{\max}{\setminus}\lambda_{i\times j}$. 
Here, $\lambda_{\max}$ is the rectangle Young diagrams of shapes $k\times (n-k)$ and 
$\lambda_{\max}{\setminus}\lambda_{i\times j}$ is the Young diagram obtained by arranging $\lambda_{i\times j}$ so that 
its bottom-right corner coincides with the bottom-right corner of $\lambda_{\max}$ and cutting out $\lambda_{i\times j}$ from $\lambda_{\max}$. 
\end{observation}

It was shown in \cite[Theorem~3]{FF} that the Newton--Okounkov polytope associated to $(G^{\rm rec}_{n-k,n})^\vee$ is unimodularly equivalent to the FFLV polytope $\calC(\Pi_{k,n})$. 
We will study further details of this equivalence. 
Namely, we give the precise description of a unimodular transformation relating $\Delta_{(G^{\rm rec}_{n-k,n})^\vee}$ and $\calC(\Pi_{k,n})$ 
in terms of the Young diagrams labeling $(G^{\rm rec}_{n-k,n})^\vee$, 
which is an analogue of Proposition~\ref{prop_unimodular_GT} for $(G^{\rm rec}_{n-k,n})^\vee$. 

To do so, we prepare some notations. 
Let $G=G^{\rm rec}_{n-k,n}$ and $(v_\lambda)_{\lambda\in\widetilde{\calP}_{G^\vee}}\in\RR^{\widetilde{\calP}_{G^\vee}}$. 
Let $w_{i \times j}$ denote the component of $(v_\lambda)_{\lambda\in\widetilde{\calP}_{G^\vee}}$ corresponding to $\lambda_{\max}{\setminus}\lambda_{i\times j}$. 
We use the convention $w_{i \times j}=w_\varnothing$ if $i \leq 0$ or $j \leq 0$. 
For example, for $k=2$ and $n=6$, we have 
\ytableausetup{boxsize=3.5pt}
\begin{center}
\begin{tabular}{lllll}
$w_{1 \times 1}=v_{\ydiagram{4,3}}$,&
$w_{1 \times 2}=v_{\ydiagram{4,2}}$,&
$w_{1 \times 3}=v_{\ydiagram{4,1}}$,&
$w_{1 \times 4}=v_{\ydiagram{4}}$, \\
$w_{2 \times 1}=v_{\ydiagram{3,3}}$,&
$w_{2 \times 2}=v_{\ydiagram{2,2}}$,&
$w_{2 \times 3}=v_{\ydiagram{1,1}}$,& 
$w_{2 \times 4}=v_\varnothing$, &
$w_\varnothing=v_{\ydiagram{4,4}}$. 
\end{tabular}
\end{center}

\medskip

\begin{theorem}
\label{thm_unimodular_FFLV2}
Let $G=G^{\rm rec}_{n-k,n}$ and $\calP^\prime_{G^\vee}\coloneqq \widetilde{\calP}_{G^\vee}{\setminus}\{\lambda_{\max}\}$. 
We consider the linear map $g\colon\RR^{\calP_{G^\vee}}\rightarrow\RR^{\calP^\prime_{G^\vee}}$ given by 
\[
w_{i\times j} \mapsto x_{i\times j}\coloneqq w_{(i-1) \times j}+w_{i \times (j-1)} -w_{i \times j} - w_{(i-1) \times (j-1)}. 
\]
Note that $w_\varnothing=w_{0 \times 0} \mapsto x_{0 \times 0}=w_{(-1) \times 0}+w_{0 \times (-1)} - w_{0 \times 0}- w_{(-1) \times (-1)} = 0$. 
Then $g$ is unimodular, and we have $g(\Delta_{G^\vee})=\calC(\Pi_{k,n})$. 
In particular, $\Delta_{G^\vee}$ is a lattice polytope. 
\end{theorem}

\begin{proof}
We divide our proof into three steps. 

\noindent
{\bf (Step 1)}: Let $J=\{j_1,\ldots,j_k\} \in \binom{[n]}{k}$ with $1 \leq j_1 < \cdots < j_k \leq n$. 
Assume that $j_{k-1} \leq k < j_k$. Let $i$ be the unique integer with $i \in [k] \setminus \{j_1,\ldots,j_{k-1}\}$ and let $j=j_k-k$. 
First, we prove that \begin{align*}
g(\val_{G^\vee}(\overline{p}_J))=(x_{a \times b})_{\small\substack{1 \leq a \leq k\hphantom{-k}\\ 1 \leq b \leq n-k}}, \text{ where }x_{a \times b}=\begin{cases}
1 &\text{ for } (a, b) = (k+1-i,j), \\
0 &\text{ otherwise}. 
\end{cases}
\end{align*}

Since we are considering the acyclic perfect orientation $\calO$ with $I_\calO=[k]$, 
we see that $I_\calO - (I_\calO \cap J)=\{i\}$ and $J - (I_\calO \cap J)=\{j_k\}$ by the assumption on $J$. 
Thus, we may consider the paths from $i$ to $j_k$ for computing a $J$-flow. 
We recall that the strongly minimal flows (and hence the valuation) for $G$ can be read off from the labeled $(n-k)\times k$ grid discussed in Remark~\ref{rem_GTtableau}. 
Thus, by Observation~\ref{obs_dualplabic}, we consider the $(n-k)\times k$ grid whose $i\times j$ entry is labeled by $\lambda_{\max}{\setminus}\lambda_{i\times j}$, and change the numbering in the right and the bottom side so that $i\mapsto i+k$ for observing strongly minimal flows on $G^\vee$. 
We also give the orientation to the grid (except the right and the bottom sides) so that vertical lines are oriented from the bottom to the top 
and horizontal lines are oriented from the left to the right,  
because the changing the colors of internal vertices reverses the acyclic orientation. 
(We would write $\lambda_{\max}{\setminus}\varnothing=\lambda_{\max}$ at the upper left margin for convenience.) 
Then we see that the strongly minimal flow on $G^\vee$ from $[k]$ to the above $J$ corresponds to the lattice path given by a vertical path 
from $i$ followed by a horizontal path ending in $j_k$ as in Figure~\ref{fig_grid_dual} (see the paragraph just above of \cite[Proposition 3]{FF}). 
Those imply that $\val_{G^\vee}(\overline{p}_J)=(w_{a \times b})_{a,b}$, where 
\begin{align*}
w_{a \times b}=\begin{cases}
1 &\text{ if }a \leq k-i \text{ or } b \leq j-1, \\
0 &\text{ otherwise}. 
\end{cases}
\end{align*}
Hence, by applying $g$, we obtain that $g(\val_{G^\vee}(\overline{p}_J))=(g(w_{a \times b}))_{a,b}$, where 
\begin{align*}
g(w_{a \times b})=w_{(a-1) \times b}+w_{a \times (b-1)} -w_{a \times b} - w_{(a-1) \times (b-1)}=
\begin{cases}
1 &\text{ if }(a,b)=(k+1-i,j), \\
0 &\text{ otherwise}. 
\end{cases}
\end{align*}

\noindent
{\bf (Step 2)}: 
Regarding other $J' \in \binom{[n]}{k}$, \cite[Proposition 3]{FF} claims that $\val_{G^\vee}(\overline{p}_{J'})$ can be described 
as a sum of $\val_{G^\vee}(\overline{p}_J)$'s for $J \in \binom{[n]}{k}$ considered in (Step 1). 
More precisely, for $J'=\{j_1',\ldots,j_k'\}$ with $1 \leq j_1' < \cdots < j_m' \leq k < j_{m+1}' < \cdots < j_k' \leq n$, \cite[Proposition 3]{FF} says that 
$\val_{G^\vee}(\overline{p}_{J'})=\sum_{\ell=1}^{k-m} \val_{G^\vee}(\overline{p}_{J_\ell})$ holds, 
where $J_\ell=([k] \setminus \{j_\ell''\}) \cup \{j_{k+1-\ell}'\}$ for $\ell=1,\dots,k-m$ 
and $\{j_1'',\dots,j_{k-m}''\} = [k] \setminus \{j_1',\dots,j_m'\}$ with $j_1'' < \cdots < j_{k-m}''$. 

\noindent
{\bf (Step 3)}: 
Consider the poset $\{q_{i \times j} \mid 1 \leq i \leq k, 1 \leq j \leq n-k\}$ equipped with the partial orders $q_{i \times j} \prec q_{i \times (j-1)}$ and $q_{i \times j} \prec q_{(i-1) \times j}$. 
This poset is nothing but $\Pi_{k,n}$. 
We show that the valuations $\val_{G^\vee}(\overline{p}_{J'})$ for $J' \in \binom{[n]}{k}$ one-to-one correspond to the antichains in $\Pi_{k,n}$. 

For any $J' \in \binom{[n]}{k}$, work with the same notation as in (Step 2), thus $\val_{G^\vee}(\overline{p}_{J'})=\sum_{\ell=1}^{k-m} \val_{G^\vee}(\overline{p}_{J_\ell})$ holds. 
By assigning each $J_\ell$ to $q_{j_\ell''} \times (j_{k+1-\ell}'-k)$, since $q_{i \times j}$ and $q_{i' \times j'}$ are incomparable in $\Pi_{k,n}$ if and only if $i \geq i'$ and $j < j'$ or $i<i'$ and $j \geq j'$, 
we conclude that $\val_{G^\vee}(\overline{p}_{J'})$ one-to-one corresponds to the antichain $\{x_{j_1'' \times (j_k'-k)},\dots,x_{j_{k-m}'' \times (j_{m+1}'-k)}\}$. 
\end{proof}

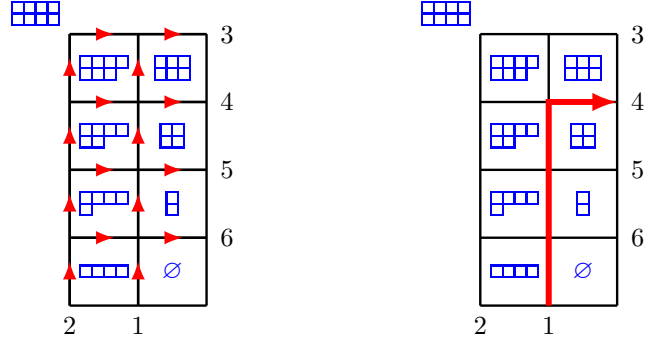
\begin{figure}[H]
\begin{center}
\scalebox{0.9}{
\begin{tikzpicture}
\newcommand{\edgewidth}{0.035cm} 
\ytableausetup{boxsize=4.5pt}

\usetikzlibrary{decorations.markings}
\tikzset{sarrow/.style={decoration={markings,
                               mark=at position .65 with {\arrow[scale=1.3, red]{latex}}},
                               postaction={decorate}}}

\foreach \n/\a/\b in {00/0/0,10/1/0,20/2/0,01/0/1,11/1/1,21/2/1,02/0/2,12/1/2,22/2/2,03/0/3,13/1/3,23/2/3,04/0/4,14/1/4,24/2/4} {
\coordinate (V\n) at (\a,\b); };

\node at (0,0) {
\begin{tikzpicture}
\foreach \s/\t in{00/04,10/14,20/24,00/20,01/21,02/22,03/23,04/24} {
\draw[line width=\edgewidth] (V\s)--(V\t); }
\foreach \s/\t in{00/01,01/02,02/03,03/04,10/11,11/12,12/13,13/14,01/11,11/21,02/12,12/22,03/13,13/23,04/14,14/24} {
\draw[sarrow, line width=\edgewidth] (V\s)--(V\t);} 
\node at (1,-0.3) {$1$}; \node at (0,-0.3) {$2$}; \node at (2.3,4) {$3$}; 
\node at (2.3,3) {$4$}; \node at (2.3,2) {$5$}; \node at (2.3,1) {$6$}; 
\foreach \n/\a/\b in {4/0.5/0.5,41/0.5/1.5,42/0.5/2.5,43/0.5/3.5,11/1.5/1.5,22/1.5/2.5,33/1.5/3.5,44/-0.5/4.3} {
\coordinate (Y\n) at (\a,\b); }; 
\node[blue] at (Y4) {\ydiagram{4}}; \node[blue] at (Y41) {\ydiagram{4,1}}; 
\node[blue] at (Y42) {\ydiagram{4,2}}; \node[blue] at (Y43) {\ydiagram{4,3}}; 
\node[blue] at (Y11) {\ydiagram{1,1}};  \node[blue] at (Y22) {\ydiagram{2,2}}; 
\node[blue] at (Y33) {\ydiagram{3,3}}; \node[blue] at (Y44) {\ydiagram{4,4}}; 
\node[blue] at (1.5,0.5) {$\varnothing$};
\end{tikzpicture}}; 

\node at (6,0) {
\begin{tikzpicture}[myarrow/.style={-latex}]
\foreach \s/\t in{00/04,10/14,20/24,00/20,01/21,02/22,03/23,04/24} {
\draw[line width=\edgewidth] (V\s)--(V\t); }
\node at (1,-0.3) {$1$}; \node at (0,-0.3) {$2$}; \node at (2.3,4) {$3$}; 
\node at (2.3,3) {$4$}; \node at (2.3,2) {$5$}; \node at (2.3,1) {$6$}; 
\foreach \n/\a/\b in {4/0.5/0.5,41/0.5/1.5,42/0.5/2.5,43/0.5/3.5,11/1.5/1.5,22/1.5/2.5,33/1.5/3.5,44/-0.5/4.3} {
\coordinate (Y\n) at (\a,\b); }; 
\node[blue] at (Y4) {\ydiagram{4}}; \node[blue] at (Y41) {\ydiagram{4,1}}; 
\node[blue] at (Y42) {\ydiagram{4,2}}; \node[blue] at (Y43) {\ydiagram{4,3}}; 
\node[blue] at (Y11) {\ydiagram{1,1}};  \node[blue] at (Y22) {\ydiagram{2,2}}; 
\node[blue] at (Y33) {\ydiagram{3,3}}; \node[blue] at (Y44) {\ydiagram{4,4}}; 
\node[blue] at (1.5,0.5) {$\varnothing$};
\draw[myarrow, red, line width=\edgewidth+1.5] (V10)--(V13)--(V23); 
\end{tikzpicture}}; 

\end{tikzpicture}}
\end{center}
\caption{The $4\times 2$ labeled grid (left). 
The lattice path given by a vertical path from $i=1$ followed by a horizontal path ending in $j_k=4$ (right), 
which corresponds to the strongly minimal flow on $(G_{4,6}^{\rm rec})^\vee$ from $\{1,2\}$ to $\{2,4\}$.}
\label{fig_grid_dual}
\end{figure}

\begin{example}
We consider the rectangle plabic graph $G^{\rm rec}_{4,6}$ (the left of Figure~\ref{ex_plabic_dual}) and 
its dual $(G^{\rm rec}_{4,6})^\vee$ (the middle of Figure~\ref{ex_plabic_dual}). 
In particular, $(G^{\rm rec}_{4,6})^\vee$ is of type $\pi_{2,6}$. 
We consider the acyclic perfect orientation of $(G^{\rm rec}_{4,6})^\vee$ as in the right of Figure~\ref{ex_plabic_dual}. 
We note that this corresponds to the grid description as in the left of Figure~\ref{fig_grid_dual}. 

\begin{figure}[H]
\begin{center}
\scalebox{0.8}{
\begin{tikzpicture}[myarrow/.style={black, -latex}]
\newcommand{\boundaryrad}{2cm} 
\newcommand{\noderad}{0.13cm} 
\newcommand{\nodewidth}{0.035cm} 
\newcommand{\edgewidth}{0.035cm} 
\newcommand{\boundarylabel}{8pt} 
\newcommand{\arrowwidth}{0.05cm} 

\ytableausetup{boxsize=4.5pt}

\foreach \n/\a in {1/60, 2/0, 3/300, 4/240, 5/180, 6/120} {
\coordinate (B\n) at (\a:\boundaryrad); 
\coordinate (B\n+) at (\a:\boundaryrad+\boundarylabel); };

\node at (0,0) {
\begin{tikzpicture}
\draw (0,0) circle(\boundaryrad) [gray, line width=\edgewidth];
\foreach \n in {1,...,6} { \draw (B\n+) node {\small $\n$}; }
\foreach \n/\a/\r in {7/60/0.65, 8/0/0.68, 9/330/0.55, 10/300/0.65, 11/260/0.55, 12/230/0.7, 13/180/0.68, 14/120/0.65} {
\coordinate (B\n) at (\a:\r*\boundaryrad); }; 
\foreach \s/\t in {7/8, 8/9, 9/10, 10/11, 11/12, 12/13, 13/14, 14/7, 9/14, 11/14, 
1/7, 2/8, 3/10, 4/12, 5/13, 6/14} {
\draw[line width=\edgewidth] (B\s)--(B\t); 
};
\foreach \x in {7,9,11,13} {
\filldraw [fill=white, line width=\edgewidth] (B\x) circle [radius=\noderad] ;}; 
\foreach \x in {8,10,12,14} {
\filldraw [fill=black, line width=\edgewidth] (B\x) circle [radius=\noderad] ; };
\foreach \n/\a/\r in {1234/210/0.8, 2345/150/0.8, 3456/90/0.75, 1456/30/0.8, 1256/330/0.8, 1236/270/0.8} {
\coordinate (V\n) at (\a:\r*\boundaryrad); }; 
\node[blue] at (V1234) {\ydiagram{2,2,2,2}}; 
\node[blue] at (V2345) {\ydiagram{1,1,1,1}}; 
\node[blue] at (V3456) {$\varnothing$};
\node[blue] at (V1456) {\ydiagram{2}}; 
\node[blue] at (V1256) {\ydiagram{2,2}}; 
\node[blue] at (V1236) {\ydiagram{2,2,2}}; 
\foreach \n/\a/\r in {2456/50/0.35, 2356/280/0.15, 2346/190/0.4} {
\coordinate (V\n) at (\a:\r*\boundaryrad); }; 
\node[blue] at (V2456) {\ydiagram{1}}; 
\node[blue] at (V2356) {\ydiagram{1,1}}; 
\node[blue] at (V2346) {\ydiagram{1,1,1}}; 
\end{tikzpicture}};

\node at (6,0) {
\begin{tikzpicture}
\draw (0,0) circle(\boundaryrad) [gray, line width=\edgewidth];
\draw (B1+) node {\small $3$}; \draw (B2+) node {\small $4$}; 
\draw (B3+) node {\small $5$}; \draw (B4+) node {\small $6$}; 
\draw (B5+) node {\small $1$}; \draw (B6+) node {\small $2$}; 
\foreach \n/\a/\r in {7/60/0.65, 8/0/0.68, 9/330/0.55, 10/300/0.65, 11/260/0.55, 12/230/0.7, 13/180/0.68, 14/120/0.65} {
\coordinate (B\n) at (\a:\r*\boundaryrad); }; 
\foreach \s/\t in {7/8, 8/9, 9/10, 10/11, 11/12, 12/13, 13/14, 14/7, 9/14, 11/14, 
1/7, 2/8, 3/10, 4/12, 5/13, 6/14} {
\draw[line width=\edgewidth] (B\s)--(B\t); 
};
\foreach \x in {7,9,11,13} {
\filldraw [fill=black, line width=\edgewidth] (B\x) circle [radius=\noderad] ;}; 
\foreach \x in {8,10,12,14} {
\filldraw [fill=white, line width=\edgewidth] (B\x) circle [radius=\noderad] ; };
\foreach \n/\a/\r in {34/210/0.8, 45/150/0.8, 56/90/0.75, 16/30/0.8, 12/330/0.8, 23/270/0.8} {
\coordinate (V\n) at (\a:\r*\boundaryrad); }; 
\node[blue] at (V12) {\ydiagram{2,2}}; 
\node[blue] at (V23) {\ydiagram{1,1}}; 
\node[blue] at (V34) {$\varnothing$};
\node[blue] at (V45) {\ydiagram{4}}; 
\node[blue] at (V56) {\ydiagram{4,4}}; 
\node[blue] at (V16) {\ydiagram{3,3}}; 
\foreach \n/\a/\r in {13/50/0.36, 14/290/0.2, 15/190/0.42} {
\coordinate (V\n) at (\a:\r*\boundaryrad); }; 
\node[blue] at (V13) {\ydiagram{4,3}}; 
\node[blue] at (V14) {\ydiagram{4,2}}; 
\node[blue] at (V15) {\ydiagram{4,1}}; 
\end{tikzpicture}};

\node at (12,0) {
\begin{tikzpicture}
\usetikzlibrary{decorations.markings}
\tikzset{sarrow/.style={decoration={markings, 
                               mark=at position .7 with {\arrow[scale=1.3, red]{latex}}},
                               postaction={decorate}}}                               
\tikzset{ssarrow/.style={decoration={markings, 
                               mark=at position .8 with {\arrow[scale=1.3, red]{latex}}},
                               postaction={decorate}}}
\draw (0,0) circle(\boundaryrad) [gray, line width=\edgewidth];
\draw (B1+) node {\small $3$}; \draw (B2+) node {\small $4$}; 
\draw (B3+) node {\small $5$}; \draw (B4+) node {\small $6$}; 
\draw (B5+) node {\small $1$}; \draw (B6+) node {\small $2$}; 
\foreach \n/\a/\r in {7/60/0.65, 8/0/0.68, 9/330/0.55, 10/300/0.65, 11/260/0.55, 12/230/0.7, 13/180/0.68, 14/120/0.65} {
\coordinate (B\n) at (\a:\r*\boundaryrad); }; 
\foreach \s/\t in {8/7, 9/8, 10/9, 11/10, 12/11, 13/12, 14/13, 14/7, 14/9, 14/11, 5/13, 6/14} {
\draw[sarrow, line width=\edgewidth] (B\s)--(B\t); 
};
\foreach \s/\t in {7/1, 8/2, 10/3, 12/4} {
\draw[ssarrow, line width=\edgewidth] (B\s)--(B\t); 
};
\foreach \x in {7,9,11,13} {
\filldraw [fill=black, line width=\edgewidth] (B\x) circle [radius=\noderad] ;}; 
\foreach \x in {8,10,12,14} {
\filldraw [fill=white, line width=\edgewidth] (B\x) circle [radius=\noderad] ; };
\foreach \n/\a/\r in {34/210/0.8, 45/150/0.8, 56/90/0.75, 16/30/0.8, 12/330/0.8, 23/270/0.8} {
\coordinate (V\n) at (\a:\r*\boundaryrad); }; 
\node[blue] at (V12) {\ydiagram{2,2}}; 
\node[blue] at (V23) {\ydiagram{1,1}}; 
\node[blue] at (V34) {$\varnothing$};
\node[blue] at (V45) {\ydiagram{4}}; 
\node[blue] at (V56) {\ydiagram{4,4}}; 
\node[blue] at (V16) {\ydiagram{3,3}}; 
\foreach \n/\a/\r in {13/50/0.36, 14/290/0.2, 15/190/0.42} {
\coordinate (V\n) at (\a:\r*\boundaryrad); }; 
\node[blue] at (V13) {\ydiagram{4,3}}; 
\node[blue] at (V14) {\ydiagram{4,2}}; 
\node[blue] at (V15) {\ydiagram{4,1}}; 
\end{tikzpicture}};

\end{tikzpicture}}
\end{center}
\caption{The rectangle plabic graph $G^{\rm rec}_{4,6}$ (left), the dual $(G^{\rm rec}_{4,6})^\vee$ (middle), and 
the acyclic perfect orientation of $(G^{\rm rec}_{4,6})^\vee$ (right)}
\label{ex_plabic_dual}
\end{figure}
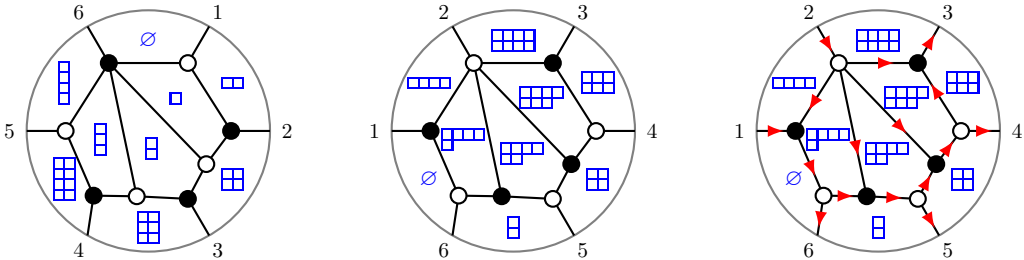

Let $G^\vee=(G^{\rm rec}_{4,6})^\vee$. 
Using this acyclic perfect orientation, we have the valuation $\val_{G^\vee}(\overline{p}_J)$ of the Pl\"{u}cker coordinate $\overline{p}_J$ 
for any $J\in{[6] \choose 2}$ as in the following table. 
By Theorem~\ref{thm_unimodular_FFLV2}, the Newton--Okounkov polytope $\Delta_{G^\vee}$ is a lattice polytope 
unimodularly equivalent to $\calC(\Pi_{2,6})$ via $g$. 
Thus, $\Delta_{G^\vee}$ is the convex hull of the valuations $\val_{G^\vee}(\overline{p}_J)$ for all $J\in{[6] \choose 2}$ given as in the following table.   

\medskip

\begin{center}
\ytableausetup{boxsize=3.5pt}
\scalebox{0.9}{
{\renewcommand\arraystretch{1.1}\tabcolsep= 10pt
\begin{tabular}{c|cccccccc}
&\raise0.5ex\hbox{\ydiagram{4,3}}&\raise0.5ex\hbox{\ydiagram{4,2}}&\raise0.5ex\hbox{\ydiagram{4,1}}&\ydiagram{4}&
\raise0.5ex\hbox{\ydiagram{1,1}}&\raise0.5ex\hbox{\ydiagram{2,2}}&\raise0.5ex\hbox{\ydiagram{3,3}}&\raise0.5ex\hbox{\ydiagram{4,4}} \\ \hline
$\{1,2\}$&0&0&0&0&0&0&0&0 \\ \hline
$\{1,3\}$&0&0&0&0&0&0&0&1 \\ \hline
$\{1,4\}$&1&0&0&0&0&0&1&1 \\ \hline
$\{1,5\}$&1&1&0&0&0&1&1&1 \\ \hline
$\{1,6\}$&1&1&1&0&1&1&1&1 \\ \hline
$\{2,3\}$&1&1&1&1&0&0&0&1 \\ \hline
$\{2,4\}$&1&1&1&1&0&0&1&1 \\ \hline
$\{2,5\}$&1&1&1&1&0&1&1&1 \\ \hline
$\{2,6\}$&1&1&1&1&1&1&1&1 \\ \hline
$\{3,4\}$&1&1&1&1&0&0&1&2 \\ \hline
$\{3,5\}$&1&1&1&1&0&1&1&2 \\ \hline
$\{3,6\}$&1&1&1&1&1&1&1&2 \\ \hline
$\{4,5\}$&2&1&1&1&0&1&2&2 \\ \hline
$\{4,6\}$&2&1&1&1&1&1&2&2 \\ \hline
$\{5,6\}$&2&2&1&1&1&2&2&2 
\end{tabular}}}
\end{center}

\ytableausetup{boxsize=3.5pt}
We apply the transformation $g$ in Theorem~\ref{thm_unimodular_FFLV2}, which is described as 
\begin{align*}
&w_{\varnothing}=v_{\ydiagram{4,4}}\mapsto 0, \hspace{15pt}
w_{1 \times 1}=v_{\ydiagram{4,3}}\mapsto v_{\ydiagram{4,4}}-v_{\ydiagram{4,3}}, \hspace{15pt}
w_{1 \times 2}=v_{\ydiagram{4,2}}\mapsto v_{\ydiagram{4,3}}-v_{\ydiagram{4,2}}, \\
&w_{1 \times 3}=v_{\ydiagram{4,1}}\mapsto v_{\ydiagram{4,2}}-v_{\ydiagram{4,1}}, \hspace{15pt}
w_{1 \times 4}=v_{\ydiagram{4}}\mapsto v_{\ydiagram{4,1}}-v_{\ydiagram{4}}, \\
&w_{2 \times 1}=v_{\ydiagram{3,3}}\mapsto v_{\ydiagram{4,3}}-v_{\ydiagram{3,3}}, \hspace{15pt}
w_{2 \times 2}=v_{\ydiagram{2,2}}\mapsto v_{\ydiagram{4,2}}+v_{\ydiagram{3,3}}-v_{\ydiagram{2,2}}-v_{\ydiagram{4,3}}, \\
&w_{2 \times 3}=v_{\ydiagram{1,1}}\mapsto v_{\ydiagram{4,1}}+v_{\ydiagram{2,2}}-v_{\ydiagram{1,1}}-v_{\ydiagram{4,2}},  \hspace{15pt}
w_{2 \times 4}=v_{\varnothing}\mapsto v_{\ydiagram{4}}+v_{\ydiagram{1,1}}-v_{\varnothing}-v_{\ydiagram{4,1}}, 
\end{align*}
to the valuations, we have the following table. 
Here, each column vector labeled by a Young diagram $\lambda$ stands for the image of $v_\lambda$ under $g$, 
and omit the one corresponding to \raise0.5ex\hbox{\ydiagram{4,4}} because $v_{\ydiagram{4,4}}\mapsto 0$. 
	
\begin{center}
\ytableausetup{boxsize=3.5pt}
\scalebox{0.9}{
{\renewcommand\arraystretch{1.1}\tabcolsep= 10pt
\begin{tabular}{c|cccccccc}
&\raise0.5ex\hbox{\ydiagram{4,3}}&\raise0.5ex\hbox{\ydiagram{4,2}}&\raise0.5ex\hbox{\ydiagram{4,1}}&\ydiagram{4}&
\raise0.5ex\hbox{\ydiagram{1,1}}&\raise0.5ex\hbox{\ydiagram{2,2}}&\raise0.5ex\hbox{\ydiagram{3,3}}&$\varnothing$ \\ \hline
$\{1,2\}$&0&0&0&0&0&0&0&0 \\ \hline
$\{1,3\}$&1&0&0&0&0&0&0&0 \\ \hline
$\{1,4\}$&0&1&0&0&0&0&0&0 \\ \hline
$\{1,5\}$&0&0&1&0&0&0&0&0 \\ \hline
$\{1,6\}$&0&0&0&1&0&0&0&0 \\ \hline
$\{2,3\}$&0&0&0&0&1&0&0&0 \\ \hline
$\{2,4\}$&0&0&0&0&0&1&0&0 \\ \hline
$\{2,5\}$&0&0&0&0&0&0&1&0 \\ \hline
$\{2,6\}$&0&0&0&0&0&0&0&1 \\ \hline
$\{3,4\}$&1&0&0&0&0&1&0&0 \\ \hline
$\{3,5\}$&1&0&0&0&0&0&1&0 \\ \hline
$\{3,6\}$&1&0&0&0&0&0&0&1 \\ \hline
$\{4,5\}$&0&1&0&0&0&0&1&0 \\ \hline
$\{4,6\}$&0&1&0&0&0&0&0&1 \\ \hline
$\{5,6\}$&0&0&1&0&0&0&0&1 
\end{tabular}}}
\end{center}

We observe that these vectors correspond to the antichains of the poset $\{q_{i \times j} \mid 1 \leq i \leq 2, 1 \leq j \leq 4\}$ 
equipped with the partial orders $q_{i \times j} \prec q_{i \times (j-1)}$ and $q_{i \times j} \prec q_{(i-1) \times j}$ as showed in (Step 3) of the proof of Theorem~\ref{thm_unimodular_FFLV2}. 
\end{example}

As we mentioned, the polytopes $\calO(\Pi_{k,n})$ and $\calC(\Pi_{k,n})$ are connected by the transfer map $\phi$. 
It was shown in \cite[Theorem~4.1]{Hig} that $\phi$ can be described as the composition of the combinatorial mutations defined along the poset $\Pi_{k,n}$. 

\begin{theorem}[{cf. \cite[Theorem~4.1 and its proof]{Hig}}]
\label{thm_transfer_mutation}
Let $G=G^{\rec}_{k,n}$. 
We consider the set of Young diagrams $\calP_G$, which is isomorphic to $\Pi_{k,n}$ as a poset. 
For $\lambda_i\in\calP_G$, let $w_i\coloneqq -\bfe_{\lambda_i}$. 
We define the lattice polytope 
\[H_{\lambda_i}\coloneqq\Conv\{-\bfe_{\lambda_j}\mid \text{$\lambda_i$ {\rm covers} $\lambda_j$}\},\]
which satisfies $H_{\lambda_i}\subset w_i^\perp$, 
and the piecewise-linear map 
\[\varphi_{\lambda_i}\coloneqq\varphi_{w_i,H_{\lambda_i}}\colon \RR^{\calP_G}\rightarrow\RR^{\calP_G}\]
as in \eqref{eq_piecewise-linear}. 
Let $\lambda_1,\dots, \lambda_{s-1}$ be all the non-minimal elements in $\calP_G$ arranged such as 
$\lambda_i\prec\lambda_j$ only if $i>j$. 
Then the transfer map $\phi$ can be described as the composition of the piecewise-linear maps: 
\[
\phi=\varphi_{\lambda_{s-1}}\circ\cdots\circ\varphi_{\lambda_2}\circ\varphi_{\lambda_1}. 
\]
Moreover, for any $i=1, \dots, s-1$, we see that $\varphi_{\lambda_i}\circ\cdots\circ\varphi_{\lambda_1}\big(\calO(\Pi_{k,n})\big)$ is 
the combinatorial mutation of $\varphi_{\lambda_{i-1}}\circ\cdots\circ\varphi_{\lambda_1}\big(\calO(\Pi_{k,n})\big)$ with respect to $w_i$ and $H_{\lambda_i}$. 
\end{theorem}

\begin{example}
\ytableausetup{boxsize=3.5pt}
We consider the GT polytope: 
\[
\calO(\Pi_{2,6})=\{ (x_\lambda) \in \RR^{\calP_{G^{\rm rec}_{2,6}}} 
\mid x_{\lambda} \leq x_{\lambda^\prime} \text{ if }\lambda \preceq \lambda^\prime \text{ in }\calP_{G^{\rm rec}_{2,6}}, 
\; 0 \leq x_\lambda \leq 1 \text{\hspace{5pt}for all } \lambda \in \calP_{G^{\rm rec}_{2,6}}\}.
\]
We will apply the combinatorial mutations $\varphi_{\lambda_i}$ given in Theorem~\ref{thm_transfer_mutation} 
along the poset structure on $\calP_{G^{\rm rec}_{2,6}}$ (see Example~\ref{ex_poset_Gr(2,6)}). 
Since the largest element in $\calP_{G^{\rm rec}_{2,6}}$ is \raise0.5ex\hbox{\ydiagram{4,4}}, 
we first consider $\varphi_{\ydiagram{4,4}}(\calO(\Pi_{2,6}))$. 
Next, \raise0.5ex\hbox{\ydiagram{3,3}} and \raise0.1ex\hbox{\ydiagram{4}} are covered by \raise0.6ex\hbox{\ydiagram{4,4}}. 
Thus, we will apply $\varphi_{\ydiagram{3,3}}$ and $\varphi_{\ydiagram{4}}$. 
In this situation, we do not care of the order of applications of these maps, 
because $\varphi_{\ydiagram{3,3}}\circ\varphi_{\ydiagram{4}}=\varphi_{\ydiagram{4}}\circ\varphi_{\ydiagram{3,3}}$. 
We note that since \raise0.1ex\hbox{\ydiagram{4}} covers only \raise0.1ex\hbox{\ydiagram{3}}, the map $\varphi_{\ydiagram{4}}$ is 
clearly a unimodular transformation. 
Repeating these arguments, we apply the sequence of combinatorial mutations: 
\begin{equation}
\label{eq_GT_mutation_FFLV}
\varphi_{\ydiagram{1}}\circ\varphi_{\ydiagram{2}}\circ\varphi_{\ydiagram{1,1}}\circ\varphi_{\ydiagram{3}}\circ\varphi_{\ydiagram{2,2}}\circ\varphi_{\ydiagram{4}}\circ\varphi_{\ydiagram{3,3}}\circ\varphi_{\ydiagram{4,4}}
\end{equation}
to $\calO(\Pi_{2,6})$. 
Here, the maps $\varphi_{\ydiagram{2}}, \varphi_{\ydiagram{1,1}}, \varphi_{\ydiagram{3}}, \varphi_{\ydiagram{4}}$ are 
unimodular transformations, not piecewise-linear transformations, since each mentioned Young diagram covers only one diagram. 
Also, the map $\varphi_{\ydiagram{1}}$ is the identity map. 
By Theorem~\ref{thm_transfer_mutation}, the map \eqref{eq_GT_mutation_FFLV} coincides with the transfer map $\phi$. 
Thus, applying the combinatorial mutations \eqref{eq_GT_mutation_FFLV} to $\calO(\Pi_{2,6})$, we have the FFLV polytope $\calC(\Pi_{2,6})$. 
\end{example}

As we saw in Proposition~\ref{prop_move_equiv}, all plabic graphs of type $\pi_{k,n}$ are move-equivalent. 
Thus, there is a sequence $\bfs$ of square moves connecting $G^{\rm rec}_{k,n}$ and $(G^{\rm rec}_{k,n})^\vee$. 
Since a square move of a plabic graph induces a tropicalized cluster mutation of the associated Newton--Okounkov polytope (see Theorem~\ref{thm_trop_map_NO}), 
we have the sequence of tropicalized cluster mutations induced by $\bfs$ that transforms $\Delta_{G^{\rm rec}_{k,n}}$ into $\Delta_{(G^{\rm rec}_{n-k,n})^\vee}$.
Furthermore, it can be described by combining the unimodular transformations and the combinatorial mutations as we showed in Proposition~\ref{prop_tropical=combmutation}. 
By Proposition~\ref{prop_unimodular_GT} and Theorem~\ref{thm_unimodular_FFLV2}, 
$\Delta_{G^{\rm rec}_{k,n}}$ and $\Delta_{(G^{\rm rec}_{n-k,n})^\vee}$ are respectively unimodularly equivalent to $\calO(\Pi_{k,n})$ and $\calC(\Pi_{k,n})$ 
via $f$ and $g$. 
Thus, combining Proposition~\ref{prop_unimodular_GT}, Theorems~\ref{thm_unimodular_FFLV2} and \ref{thm_transfer_mutation}, 
we have another sequence of unimodular transformations and combinatorial mutations 
that connects $\Delta_{G^{\rm rec}_{k,n}}$ and $\Delta_{(G^{\rm rec}_{n-k,n})^\vee}$: 
\begin{equation}
\label{seq_transfer_unimodular}
\Delta_{G^{\rm rec}_{k,n}}\xrightarrow{f} \calO(\Pi_{k,n})
\xrightarrow{\phi\,=\,\varphi_{\lambda_{s-1}}\,\circ\,\cdots\,\circ\,\varphi_{\lambda_2}\,\circ\,\varphi_{\lambda_1}}
\calC(\Pi_{k,n})\xrightarrow{g^{-1}}\Delta_{(G^{\rm rec}_{n-k,n})^\vee}.
\end{equation}

\begin{remark}
As noted in \cite[Remark~1]{FF}, 
by changing the numbering of the boundary vertices of $G^{\rm rec}_{k,n}$ as $r\mapsto n+1-r$ (modulo $n$) for $r=1,\dots,n$, 
we have another plabic graph $G^\prime$ whose associated Newton--Okounkov polytope is unimodularly equivalent to $\calC(\Pi_{k,n})$. 
This plabic graph $G^\prime$ can be obtained from $G^{\rm rec}_{k,n}$ by applying a special sequence of mutations 
called a \emph{maximal green sequence} (cf. \cite{Kel,DK}). 
Note that maximal green sequences for a rectangle plabic graph is well studied in \cite[Section~11]{MaSc}. 
Thus the sequence of tropicalized cluster mutations induced from this maximal green sequence connects $\Delta_{G^{\rm rec}_{k,n}}$ and 
$\Delta_{(G^{\rm rec}_{n-k,n})^\vee}$, and hence so does $\calO(\Pi_{k,n})$ and $\calC(\Pi_{k,n})$ via $f,g^{-1}$. 
However, it is different from the sequence \eqref{seq_transfer_unimodular} in general. 
\end{remark}

\section{Remarks on tropicalization with respect to the maximum convention} 
\label{sec_max_tropical}

The tropicalized cluster mutation $\Psi_{G,G^\prime}$ discussed in Section~\ref{sec_tropical_mutation} 
corresponds to the tropicalization of the exchange relation of cluster variables. 
This tropicalization uses the ``minimum convention", whereas we can also consider the ``maximum convention". 
As noted in \cite[Remark~5.11]{EH}, the tropicalization with respect to the maximum convention also appears in the context of 
the wall-crossing of the Newton--Okounkov polytopes associated to maximal prime cone of the tropical Grassmannian. 
In the first part of this section, we observe the relationship between the wall crossing map and the tropicalization with respect to the maximum convention. 
In the later half, we show two kinds of tropicalizations are related via the piecewise-linear map used in Proposition~\ref{prop_tropical=combmutation}. 

\medskip

First, we give a brief review of the tropical Grassmannian and its wall-crossing phenomenon following \cite{EH,SS}. 
As we noted in Section~\ref{sec_intro}, tropical geometry gives an effective method to construct toric degenerations of $\Gr(k, n)$. 
The \emph{tropical Grassmannian} $\Trop(\Gr(k, n))$ is the tropicalization of the Pl\"{u}cker ideal $I_{k,n}$ for $\XX_{k,n}=\Gr(k, n)$ (cf. \cite{SS}). 
Namely, $\Trop(\Gr(k, n))$ is the set of weight vectors $\bfw\in\RR^{n \choose k}$ such that 
the initial ideal ${\rm in}_\bfw(I_{k,n})$ does not contain a monomial. 
The tropical Grassmannian $\Trop(\Gr(k, n))$ is a subfan of the Gr\"{o}bner fan of $I_{k,n}$, 
and any maximal cone of  has the dimension $s+1=k(n-k)+1$ \cite[Corollary~3.1]{SS}. 
Also, weight vectors contained in the same cone $C$ of $\Trop(\Gr(k, n))$ give the same initial ideal, hence we may write this as ${\rm in}_C(I_{k,n})$. 
If the initial ideal ${\rm in}_C(I_{k,n})$ associated to a maximal cone $C$ is prime and binomial, 
then we have a toric degeneration of $\XX_{k,n}$ via Gr\"{o}bner theory (e.g., \cite[Theorem~15.17]{Eis}). 
Also, by using the method in \cite{KM}, we can construct the Newton--Okounkov polytopes associated to a maximal prime cone $C$, 
which gives the same toric degeneration as the one obtained from $C$ via Gr\"{o}bner theory. 

We turn our attention to the case of $k=2$ (see \cite[Section~5]{EH} for more details). 
It was shown in \cite[Corollary~4.4]{SS} that any initial ideal ${\rm in}_C(I_{2,n})$ associated to a maximal cone $C$ in $\Trop(\Gr(2, n))$ 
is prime and binomial. 
In this case, for any pair of adjacent maximal prime cones $C_1,C_2$ in $\Trop(\Gr(2, n))$, 
we can consider two kinds of the (\emph{geometric}) \emph{wall-crossing maps} (called the \emph{flip} and \emph{shift maps}) introduced in \cite{EH}  
for relating the Newton--Okounkov polytopes associated to $C_1$ and $C_2$. 
It is known that the maximal prime cones of $\Trop(\Gr(2, n))$ are parametrized by \emph{labeled trivalent trees} with $n$ leaves (see \cite[Section~4]{SS}), 
and labeled trivalent trees corresponding to the adjacent maximal prime cones are related by the \emph{mutation of a labeled trivalent tree} (see \cite[Section~6]{BFFHL}). 
Thus, the wall-crossing maps can be understood in terms of labeled trivalent trees and their mutations. 
Also, a labeled trivalent tree can be obtained from a triangulation of the $n$-gon (see \cite[Section~4]{BFFHL}). 
For such a triangulation, we can associate a plabic graph of type $\pi_{2,n}$ (see \cite{KW}). 
In this situation, the mutation of a labeled trivalent tree, the ``flip" of a diagonal in a triangulation, 
and the square move of a plabic graph, are equivalent to one another. 
Thus, we can observe the wall-crossing maps in terms of a plabic graph. 

\medskip

We now introduce two piecewise-linear maps $\sfF, \sfS$. 
Let $G$ be a plabic graph of type $\pi_{k,n}$. 
We label the coordinates of the ambient space $\RR^{s+1}$ by the Young diagrams in $\widetilde{\calP}_G$ where $s=k(n-k)$. 
Then we define 
\[
\sfF\colon\RR^{s+1}\rightarrow\RR^{s+1}, \quad 
(v_{\lambda_1},\dots,v_{\lambda_{s+1}})\mapsto  (v_{\lambda_1},\dots,v_{\lambda_{i-1}},v^\prime_{\lambda_i},v_{\lambda_{i+1}},\dots,v_{\lambda_{s+1}}), 
\]
where 
\begin{equation}
\label{eq_F}
v^\prime_{\lambda_i}=-v_{\lambda_i}+\min\{v_{\lambda_a}+v_{\lambda_d}, v_{\lambda_b}+v_{\lambda_c}\}
+\max\{|v_{\lambda_a}-v_{\lambda_b}|, |v_{\lambda_c}-v_{\lambda_d}|\},
\end{equation}

\[
\sfS\colon\RR^{s+1}\rightarrow\RR^{s+1}, \quad 
(v_{\lambda_1},\dots,v_{\lambda_{s+1}})\mapsto  (v_{\lambda_1},\dots,v_{\lambda_{i-1}},v^{\prime\prime}_{\lambda_i},v_{\lambda_{i+1}},\dots,v_{\lambda_{s+1}}), 
\]
where 
\begin{equation}
\label{eq_S}
v^{\prime\prime}_{\lambda_i}=v_{\lambda_i}+\max\{|v_{\lambda_a}-v_{\lambda_d}|, |v_{\lambda_b}-v_{\lambda_c}|\}
-\max\{|v_{\lambda_a}-v_{\lambda_b}|, |v_{\lambda_c}-v_{\lambda_d}|\}, 
\end{equation}
and $\lambda_a, \lambda_b, \lambda_c, \lambda_d$ are the faces of $G$ located cyclically around 
a square face $\lambda_i$ as in Lemma~\ref{lem_simplify_FS} below. 

For the case $k=2$, these maps $\sfF, \sfS$ coincide with the ones given in \cite[Subsection~5.3]{EH}, 
and $\sfS$ coincides with the piecewise-linear map given in \cite[Proposition~3.5]{NU}. 
These maps relate the Newton--Okounkov polytopes associated to adjacent maximal prime cones in $\Trop(\Gr(2, n))$ \cite[Section~5]{EH}, 
and this wall-crossing phenomenon can also be understood from the viewpoint of the combinatorial mutations \cite[Appendix]{EH}. 
As observed in \cite[Remark~5.11]{EH} and \cite[Corollary~4.64]{BMN}, this wall-crossing is related to the tropicalization of the cluster mutation. 

In what follows, we apply these maps $\sfF, \sfS$ to the valuation (and the Newton--Okounkov polytope) arising from a plabic graph, 
in which case we have the tropicalization of the cluster mutation with respect to the maximum convention as noted in \cite[Remark~5.11]{EH}. 
We first observe that $\sfF$ and $\sfS$ can be simplified as in Proposition~\ref{prop_simplify_FS} below when they are applied to the valuation. 
The key is the following lemma. 

\begin{lemma}
\label{lem_simplify_FS}
Let $G$ be a plabic graph of type $\pi_{k,n}$. 
Let $\nu^J\coloneqq\val_{G}(\overline{p}_J)\in\ZZ^{\widetilde{\calP}_G}$ for any Pl\"{u}cker coordinate $\overline{p}_J$ in $A_{k,n}$. 
We denote by $\nu_\lambda^J$ the component of $\nu^J$ corresponding to a face $\lambda\in\widetilde{\calP}_G$. 
We suppose that $\lambda_i\in\widetilde{\calP}_G$ is a square face, 
and $\lambda_a, \lambda_b, \lambda_c, \lambda_d\in\widetilde{\calP}_G$ are faces locating cyclically around $\lambda_i$ 
as in the figure below. 
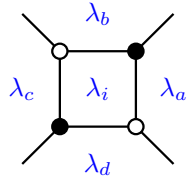
\begin{center}
\scalebox{1}{
\begin{tikzpicture}
\newcommand{\noderad}{0.1cm} 
\newcommand{\nodewidth}{0.03cm} 
\newcommand{\edgewidth}{0.03cm} 

\foreach \n/\a/\b in {A/0/0, B/1/0, C/1/1, D/0/1} {
\coordinate (\n) at (\a,\b); } 
\foreach \n/\a in {A/225, B/315, C/45, D/135} {
\path (\n) ++(\a:0.7) coordinate (\n+); }

\foreach \s/\t in {A/B, B/C, C/D, D/A, A/A+, B/B+, C/C+, D/D+}{
\draw[line width=\edgewidth] (\s)--(\t); }

\filldraw [fill=black, line width=\nodewidth] (A) circle [radius=\noderad] ;
\filldraw [fill=white, line width=\nodewidth] (B) circle [radius=\noderad] ; 
\filldraw [fill=black, line width=\nodewidth] (C) circle [radius=\noderad] ;
\filldraw [fill=white, line width=\nodewidth] (D) circle [radius=\noderad] ; 

\node[blue] (V) at (0.5,0.5) {\small $\lambda_i$} ;
\node[blue] (V1) at (1.5,0.5) {\small $\lambda_a$} ;\node[blue] (V2) at (0.5,1.5) {\small $\lambda_b$} ;
\node[blue] (V3) at (-0.5,0.5) {\small $\lambda_c$} ;\node[blue] (V4) at (0.5,-0.5) {\small $\lambda_d$} ;
\end{tikzpicture}
}
\end{center}
Then (after rotating the labels $a,b,c,d$ if necessary) we have 
\begin{equation}
\label{eq_compare_valuation}
\nu_{\lambda_a}^J\ge\nu_{\lambda_b}^J\ge\nu_{\lambda_c}^J \quad\text{and}\quad \nu_{\lambda_a}^J\ge\nu_{\lambda_d}^J\ge\nu_{\lambda_c}^J
\end{equation}
for any $J\in{[n] \choose k}$. 
\end{lemma}

\begin{proof}
To observe the valuations defined from $G$, we consider a unique acyclic perfect orientation 
$\calO$ such that $I_\calO=[k]=\{1,\dots, k\}$, discussed in Subsection~\ref{subsec_prelim_NO}. 
Up to rotations, this orientation $\calO$ takes the following form around any square face in $G$. 

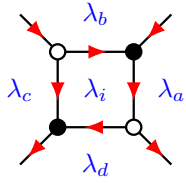
\begin{center}
\scalebox{1}{
\begin{tikzpicture}
\newcommand{\noderad}{0.1cm} 
\newcommand{\nodewidth}{0.03cm} 
\newcommand{\edgewidth}{0.03cm} 

\usetikzlibrary{decorations.markings}
\tikzset{sarrow/.style={decoration={markings,
                               mark=at position .65 with {\arrow[scale=1.3, red]{latex}}},
                               postaction={decorate}}}
                  
 \tikzset{ssarrow/.style={decoration={markings,
                               mark=at position .8 with {\arrow[scale=1.3, red]{latex}}},
                               postaction={decorate}}}

\foreach \n/\a/\b in {A/0/0, B/1/0, C/1/1, D/0/1} {
\coordinate (\n) at (\a,\b); } 
\foreach \n/\a in {A/225, B/315, C/45, D/135} {
\path (\n) ++(\a:0.7) coordinate (\n+); }

\draw[sarrow, line width=\edgewidth] (B)--(A); \draw[sarrow, line width=\edgewidth] (C)--(B); 
\draw[sarrow, line width=\edgewidth] (D)--(C); \draw[sarrow, line width=\edgewidth] (D)--(A); 

\draw[ssarrow, line width=\edgewidth] (A)--(A+); \draw[ssarrow, line width=\edgewidth] (B)--(B+); 
\draw[sarrow, line width=\edgewidth] (C+)--(C); \draw[sarrow, line width=\edgewidth] (D+)--(D); 

\filldraw [fill=black, line width=\nodewidth] (A) circle [radius=\noderad] ;
\filldraw [fill=white, line width=\nodewidth] (B) circle [radius=\noderad] ; 
\filldraw [fill=black, line width=\nodewidth] (C) circle [radius=\noderad] ;
\filldraw [fill=white, line width=\nodewidth] (D) circle [radius=\noderad] ; 

\node[blue] (V) at (0.5,0.5) {\small $\lambda_i$} ;
\node[blue] (V1) at (1.5,0.5) {\small $\lambda_a$} ;\node[blue] (V2) at (0.5,1.5) {\small $\lambda_b$} ;
\node[blue] (V3) at (-0.5,0.5) {\small $\lambda_c$} ;\node[blue] (V4) at (0.5,-0.5) {\small $\lambda_d$} ;
\end{tikzpicture}}
\end{center}

We observe how to behave the strongly minimal flow $F_{\min}$ from $\{1,\dots,k\}$ to $J$ which satisfies $\nu^J=\wt(F_{\min})$ around the face $\lambda_i$. 
When we restrict the flow $F_{\min}$ to the neighborhood of $\lambda_i$, it takes one of the following forms (cf. \cite[The proof of Theorem~13.1]{RW}): 

\begin{center}
\scalebox{1}{
\begin{tikzpicture}[myarrow/.style={black, -latex}]
\newcommand{\noderad}{0.1cm} 
\newcommand{\nodewidth}{0.03cm} 
\newcommand{\edgewidth}{0.03cm} 
\newcommand{\flowwidth}{0.08cm} 
\newcommand{\graphscale}{0.9} 

\foreach \n/\a/\b in {A/0/0, B/1/0, C/1/1, D/0/1} {
\coordinate (\n) at (\a,\b); } 
\foreach \n/\a in {A/225, B/315, C/45, D/135} {
\path (\n) ++(\a:0.7) coordinate (\n+); }

\node at (0,0){
\scalebox{\graphscale}{
\begin{tikzpicture}
\foreach \s/\t in {A/B,B/C,C/D,D/A,A/A+,B/B+,C/C+,D/D+} {
\draw[line width=\edgewidth] (\s)--(\t); }

\filldraw [fill=black, line width=\nodewidth] (A) circle [radius=\noderad] ;
\filldraw [fill=white, line width=\nodewidth] (B) circle [radius=\noderad] ; 
\filldraw [fill=black, line width=\nodewidth] (C) circle [radius=\noderad] ;
\filldraw [fill=white, line width=\nodewidth] (D) circle [radius=\noderad] ; 

\node[blue] (V) at (0.5,0.5) {\small $\lambda_i$} ;
\node[blue] (V1) at (1.5,0.5) {\small $\lambda_a$} ;\node[blue] (V2) at (0.5,1.5) {\small $\lambda_b$} ;
\node[blue] (V3) at (-0.5,0.5) {\small $\lambda_c$} ;\node[blue] (V4) at (0.5,-0.5) {\small $\lambda_d$} ;
\end{tikzpicture}
}}; 

\node at (4,0){
\scalebox{\graphscale}{
\begin{tikzpicture}
\foreach \s/\t in {A/B,B/C,C/D,D/A,A/A+,B/B+,C/C+,D/D+} {
\draw[line width=\edgewidth] (\s)--(\t); }

\filldraw [fill=black, line width=\nodewidth] (A) circle [radius=\noderad] ;
\filldraw [fill=white, line width=\nodewidth] (B) circle [radius=\noderad] ; 
\filldraw [fill=black, line width=\nodewidth] (C) circle [radius=\noderad] ;
\filldraw [fill=white, line width=\nodewidth] (D) circle [radius=\noderad] ; 

\node[blue] (V) at (0.5,0.5) {\small $\lambda_i$} ;
\node[blue] (V1) at (1.5,0.5) {\small $\lambda_a$} ;\node[blue] (V2) at (0.5,1.5) {\small $\lambda_b$} ;
\node[blue] (V3) at (-0.5,0.5) {\small $\lambda_c$} ;\node[blue] (V4) at (0.5,-0.5) {\small $\lambda_d$} ;

\draw [myarrow, line width=\flowwidth,red] (C+)--(C)--(B)--(B+); 
\end{tikzpicture}
}}; 

\node at (8,0){
\scalebox{\graphscale}{
\begin{tikzpicture}
\foreach \s/\t in {A/B,B/C,C/D,D/A,A/A+,B/B+,C/C+,D/D+} {
\draw[line width=\edgewidth] (\s)--(\t); }

\filldraw [fill=black, line width=\nodewidth] (A) circle [radius=\noderad] ;
\filldraw [fill=white, line width=\nodewidth] (B) circle [radius=\noderad] ; 
\filldraw [fill=black, line width=\nodewidth] (C) circle [radius=\noderad] ;
\filldraw [fill=white, line width=\nodewidth] (D) circle [radius=\noderad] ; 

\node[blue] (V) at (0.5,0.5) {\small $\lambda_i$} ;
\node[blue] (V1) at (1.5,0.5) {\small $\lambda_a$} ;\node[blue] (V2) at (0.5,1.5) {\small $\lambda_b$} ;
\node[blue] (V3) at (-0.5,0.5) {\small $\lambda_c$} ;\node[blue] (V4) at (0.5,-0.5) {\small $\lambda_d$} ;

\draw [myarrow, line width=\flowwidth,red] (D+)--(D)--(A)--(A+); 
\draw [myarrow, line width=\flowwidth,red] (C+)--(C)--(B)--(B+); 
\end{tikzpicture}
}}; 

\node at (0,-2.8){
\scalebox{\graphscale}{
\begin{tikzpicture}
\foreach \s/\t in {A/B,B/C,C/D,D/A,A/A+,B/B+,C/C+,D/D+} {
\draw[line width=\edgewidth] (\s)--(\t); }

\filldraw [fill=black, line width=\nodewidth] (A) circle [radius=\noderad] ;
\filldraw [fill=white, line width=\nodewidth] (B) circle [radius=\noderad] ; 
\filldraw [fill=black, line width=\nodewidth] (C) circle [radius=\noderad] ;
\filldraw [fill=white, line width=\nodewidth] (D) circle [radius=\noderad] ; 

\node[blue] (V) at (0.5,0.5) {\small $\lambda_i$} ;
\node[blue] (V1) at (1.5,0.5) {\small $\lambda_a$} ;\node[blue] (V2) at (0.5,1.5) {\small $\lambda_b$} ;
\node[blue] (V3) at (-0.5,0.5) {\small $\lambda_c$} ;\node[blue] (V4) at (0.5,-0.5) {\small $\lambda_d$} ;

\draw [myarrow, line width=\flowwidth,red] (C+)--(C)--(B)--(A)--(A+); 
\end{tikzpicture}
}}; 

\node at (4,-2.8){
\scalebox{\graphscale}{
\begin{tikzpicture}
\foreach \s/\t in {A/B,B/C,C/D,D/A,A/A+,B/B+,C/C+,D/D+} {
\draw[line width=\edgewidth] (\s)--(\t); }

\filldraw [fill=black, line width=\nodewidth] (A) circle [radius=\noderad] ;
\filldraw [fill=white, line width=\nodewidth] (B) circle [radius=\noderad] ; 
\filldraw [fill=black, line width=\nodewidth] (C) circle [radius=\noderad] ;
\filldraw [fill=white, line width=\nodewidth] (D) circle [radius=\noderad] ; 

\node[blue] (V) at (0.5,0.5) {\small $\lambda_i$} ;
\node[blue] (V1) at (1.5,0.5) {\small $\lambda_a$} ;\node[blue] (V2) at (0.5,1.5) {\small $\lambda_b$} ;
\node[blue] (V3) at (-0.5,0.5) {\small $\lambda_c$} ;\node[blue] (V4) at (0.5,-0.5) {\small $\lambda_d$} ;

\draw [myarrow, line width=\flowwidth,red] (D+)--(D)--(C)--(B)--(A)--(A+); 
\end{tikzpicture}
}}; 

\node at (8,-2.8){
\scalebox{\graphscale}{
\begin{tikzpicture}
\foreach \s/\t in {A/B,B/C,C/D,D/A,A/A+,B/B+,C/C+,D/D+} {
\draw[line width=\edgewidth] (\s)--(\t); }

\filldraw [fill=black, line width=\nodewidth] (A) circle [radius=\noderad] ;
\filldraw [fill=white, line width=\nodewidth] (B) circle [radius=\noderad] ; 
\filldraw [fill=black, line width=\nodewidth] (C) circle [radius=\noderad] ;
\filldraw [fill=white, line width=\nodewidth] (D) circle [radius=\noderad] ; 

\node[blue] (V) at (0.5,0.5) {\small $\lambda_i$} ;
\node[blue] (V1) at (1.5,0.5) {\small $\lambda_a$} ;\node[blue] (V2) at (0.5,1.5) {\small $\lambda_b$} ;
\node[blue] (V3) at (-0.5,0.5) {\small $\lambda_c$} ;\node[blue] (V4) at (0.5,-0.5) {\small $\lambda_d$} ;

\draw [myarrow, line width=\flowwidth,red] (D+)--(D)--(C)--(B)--(B+); 
\end{tikzpicture}
}}; 
\end{tikzpicture}}
\end{center}

Here, the upper left picture means that $F_{\min}$ does not pass through the neighborhood of $\lambda_i$. 
From these pictures, we see that if $\lambda_b$ or $\lambda_d$ is located at the left of a certain path $p$ in $F_{\min}$, 
then $\lambda_a$ is also located at the left of $p$. 
Thus we have $\nu_{\lambda_a}^J\ge\nu_{\lambda_b}^J$ and $\nu_{\lambda_a}^J\ge\nu_{\lambda_d}^J$. 
Also, when $\lambda_c$ is located at the left of a certain path $p$ in $F_{\min}$, 
$\lambda_b$ and $\lambda_d$ are also located at the left of $p$. 
Thus we have $\nu_{\lambda_b}^J\ge\nu_{\lambda_c}^J$ and $\nu_{\lambda_d}^J\ge\nu_{\lambda_c}^J$. 
\end{proof}

\begin{proposition}
\label{prop_simplify_FS}
Work with the same notation as {\rm Lemma~\ref{lem_simplify_FS}}. 
In particular, we label the faces of $G$ so that \eqref{eq_compare_valuation} is satisfied. 
Let $\sfF, \sfS$ be the maps defined as \eqref{eq_F} and \eqref{eq_S} respectively. 
Then we see that 
\begin{align*}
\sfF(\nu^J)_i &=-\nu_{\lambda_i}^J+\max\{\nu_{\lambda_a}^J+\nu_{\lambda_c}^J,\nu_{\lambda_b}^J+\nu_{\lambda_d}^J\}, \\
\sfS(\nu^J)_i &=\nu_{\lambda_i}^J+\nu_{\lambda_b}^J-\nu_{\lambda_d}^J.
\end{align*}
\end{proposition}

\begin{proof}
Using Lemma~\ref{lem_simplify_FS} and the fact that any component of our valuation is non-negative, we see that 
\begin{align*}
&-\nu^J_{\lambda_i}+\min\{\nu^J_{\lambda_a}+\nu^J_{\lambda_d}, \nu^J_{\lambda_b}+\nu^J_{\lambda_c}\}
+\max\{|\nu^J_{\lambda_a}-\nu^J_{\lambda_b}|, |\nu^J_{\lambda_c}-\nu^J_{\lambda_d}|\} \\
=&-\nu^J_{\lambda_i}+(\nu^J_{\lambda_b}+\nu^J_{\lambda_c})
+\max\{\nu^J_{\lambda_a}-\nu^J_{\lambda_b}, \nu^J_{\lambda_d}-\nu^J_{\lambda_c}\} \\
=&-\nu^J_{\lambda_i}+\max\{\nu^J_{\lambda_a}+\nu^J_{\lambda_c}, \nu^J_{\lambda_b}+\nu^J_{\lambda_d}\}, 
\end{align*}
and 
\begin{align*}
&\nu^J_{\lambda_i}+\max\{|\nu^J_{\lambda_a}-\nu^J_{\lambda_d}|, |\nu^J_{\lambda_b}-\nu^J_{\lambda_c}|\} 
-\max\{|\nu^J_{\lambda_a}-\nu^J_{\lambda_b}|, |\nu^J_{\lambda_c}-\nu^J_{\lambda_d}|\} \\
=\,&\nu^J_{\lambda_i}+\max\{\nu^J_{\lambda_a}-\nu^J_{\lambda_d}, \nu^J_{\lambda_b}-\nu^J_{\lambda_c}\} 
-\max\{\nu^J_{\lambda_a}-\nu^J_{\lambda_b}, \nu^J_{\lambda_d}-\nu^J_{\lambda_c}\} \\
=\,&\nu^J_{\lambda_i}+\max\{\nu^J_{\lambda_a}-\nu^J_{\lambda_b}-\nu^J_{\lambda_d}, -\nu^J_{\lambda_c}\}+\nu^J_{\lambda_b}
-\max\{\nu^J_{\lambda_a}-\nu^J_{\lambda_b}-\nu^J_{\lambda_d}, -\nu^J_{\lambda_c}\}-\nu^J_{\lambda_d} \\
=\,&\nu_{\lambda_i}^J+\nu_{\lambda_b}^J-\nu_{\lambda_d}^J
\end{align*}
for any $J\in{[n] \choose k}$. Thus we have the assertions. 
\end{proof}

\begin{remark}
By the description of $\sfS$ as in Proposition~\ref{prop_simplify_FS}, we see that 
if the Newton--Okounkov polytope $\Delta_G$ is integral, then $\Delta_G$ is the convex hull of the valuations of Pl\"{u}cker coordinates 
by Theorem~\ref{thm_NO=C}, in which case we see that $\Delta_G$ and $\sfS(\Delta_G)$ are unimodularly equivalent. 
On the other hand, $\sfF(\Delta_G)$ is not convex in general. 
For example, we consider the Newton--Okounkov polytope $\Delta_G$ for $G=G^{\rm rec}_{2,6}$, 
which was given in Example~\ref{ex_NOpolytope_Gr(2,6)}. 
In this case, we see that $\sfF(\Delta_G)$ is not convex by the same argument as Example~\ref{ex_not_convex}. 
\end{remark}

The image of $\sfF$ in Proposition~\ref{prop_simplify_FS} is the tropicalization of the exchange relation with respect to the maximum convention. 
Thus, we then compare this map with the tropicalized cluster mutation $\Psi_{G,G^\prime}$ that uses the minimum convention. 
To do so, we introduce a new piecewise-linear map $\Psi^{\max}_{G,G^\prime}$ that coincides with $\sfF$ on the level of valuations, 
i.e., $\Psi^{\max}_{G,G^\prime}(\nu^J)=\sfF(\nu^J)$ for any Pl\"{u}cker coordinate $\overline{p}_J$. 
For consistency, let $\lambda_{s+1}=\varnothing$ and we omit the coordinate labeled by $\lambda_{s+1}$ 
since $v_{\lambda_{s+1}}=0$ in our interested situation. 

\begin{definition}
\label{def_maxpropical}
Let the notation be the same as Definition~\ref{def_trop_mutation}.
We define the piecewise-linear map $\Psi^{\max}_{G,G^\prime}: \RR^{\calP_G}\rightarrow \RR^{\calP_{G^\prime}}$ as
$(v_{\lambda_1},\dots, v_{\lambda_s}) \mapsto 
(v_{\lambda_1},\dots, v_{\lambda_{i-1}}, v_{\lambda_i}^\prime, v_{\lambda_{i+1}},\dots, v_{\lambda_s})$, 
where 
\[
v_{\lambda_i}^\prime=-v_{\lambda_i}+\max\{v_{\lambda_a}+v_{\lambda_c}, v_{\lambda_b}+v_{\lambda_d}\}. 
\]
\end{definition}

In contrast to this, we denote $\Psi^{\min}_{G,G^\prime}=\Psi_{G,G^\prime}$. 
As we showed in Proposition~\ref{prop_tropical=combmutation}, $\Psi^{\min}_{G,G^\prime}$ can be 
described as $\Psi^{\min}_{G,G^\prime}=\varphi_{w,F_{\lambda_i}}\circ\epsilon_{\lambda_i}$. 
We have a similar description for $\Psi_{G,G^\prime}^{\max}$. 

\begin{proposition}
Let the notation be the same as {\rm Proposition~\ref{prop_tropical=combmutation}}. 
Then we have that $\Psi_{G,G^\prime}^{\max}=\varphi_{-w,F_{\lambda_i}}\circ\epsilon_{\lambda_i}$. 
In particular, we have 
\[
\varphi_{w,F_{\lambda_i}}\circ \Psi_{G,G^\prime}^{\max}=\varphi_{-w,F_{\lambda_i}}\circ \Psi_{G,G^\prime}^{\min}=\epsilon_i 
\quad \text{and} \quad
\Psi_{G,G^\prime}^{\min}=(\varphi_{w,F_{\lambda_i}})^2\circ \Psi_{G,G^\prime}^{\max}. 
\]
\end{proposition}

\begin{proof}
First, $\Psi_{G,G^\prime}^{\max}=\varphi_{-w,F_{\lambda_i}}\circ\epsilon_{\lambda_i}$ follows from an argument similar to the proof of Proposition~\ref{prop_tropical=combmutation}. 
The remaining assertions follow from the fact that $\varphi_{w,F_{\lambda_i}}^{-1}=\varphi_{-w,F_{\lambda_i}}$ and 
$\Psi^{\min}_{G,G^\prime}=\varphi_{w,F_{\lambda_i}}\circ\epsilon_{\lambda_i}$. 
\end{proof}


\appendix
\section{Computations for $\Gr(3,6)$}
\label{appendix_Gr(3,6)}

In this section, we apply several statements discussed in this article to the case of $\Gr(3,6)$. 
A notable point is that we encounter Newton--Okounkov polytopes that are not integral. 

Following \cite[Section~7]{BFFHL}, \cite[Section~9]{RW} and \cite[Section~5]{SS}, we recall some facts concerning toric degenerations of $\Gr(3,6)$. 
First, it is known that there are $34$ plabic graphs of type $\pi_{3,6}$ up to moving operations (M2) and (M3), see Figure~\ref{ex_graph_Gr36_1} and \ref{ex_graph_Gr36_2} below.
From these plabic graphs, we have $6$ isomorphism classes of Newton--Okounkov polytopes, 
which give toric degenerations as in Theorem~\ref{thm_toric_degeneration}. 
On the other hand, there are $7$ isomorphism classes of maximal prime cones in the tropical Grassmannian $\Trop(\Gr(3, 6))$ 
that give rise to toric degenerations of $\Gr(3,6)$. 
In the language of \cite{SS}, these classes are respectively labeled by FFGG, EEEE, EEFF1, EEFF2, EFFG, EEEG, and EEFG. 
In these classes, only the last $5$ isomorphism classes correspond to the ones arising from plabic graphs. 
Thus, there is an isomorphism class of Newton--Okounkov polytopes arising from plabic graphs that does not correspond to 
a maximal prime cone of $\Trop(\Gr(3, 6))$, which should be labeled by GG (cf. \cite[after Lemma~5.1]{SS}) due to its feature. 
There are exactly two plabic graphs whose associated Newton--Okounkov polytopes are GG. 
In summary, we have the following table: 

\medskip

\scalebox{0.8}{
{\renewcommand\arraystretch{1.1}\tabcolsep= 7pt
\begin{tabular}{r|cccccccc}
Labels of isomorphism classes &FFGG&EEEE&EEFF1&EEFF2&EFFG&EEEG&EEFG&GG \\ \hline\hline
max. prime cones in $\Trop(\Gr(3, 6))$&\checkmark&\checkmark&\checkmark&\checkmark&\checkmark&\checkmark&\checkmark \\ \hline
plabic graphs of type $\pi_{3,6}$&&&\checkmark&\checkmark&\checkmark&\checkmark&\checkmark&\checkmark
\end{tabular}}}

\medskip

In what follows, we focus on the Newton--Okounkov polytopes associated to plabic graphs of type $\pi_{3,6}$. 
As we noted above, $32$ plabic graphs give the integral Newton--Okounkov polytopes which are respectively labeled one of 
$\{\text{EEFF1, EEFF2, EFFG, EEEG, EEFG}\}$, and the remaining two plabic graphs give the non-integral Newton--Okounkov polytopes labeled by GG. 
The precise labeling can be revealed by comparing the $3$-subsets labeling the internal faces of a plabic graph and the first column of \cite[Table~1]{BFFHL}. 
For example, the rectangle plabic graph $G=G^{\rm rec}_{3,6}$ takes the form as in the left of Figure~\ref{ex_plabic36}. 
The right of Figure~\ref{ex_plabic36} is the acyclic perfect orientation whose source set is $\{1,2,3\}$. 

\begin{figure}[H]
\begin{center}
\scalebox{0.8}{
\begin{tikzpicture}[myarrow/.style={black, -latex}]
\newcommand{\boundaryrad}{2cm} 
\newcommand{\noderad}{0.13cm} 
\newcommand{\nodewidth}{0.035cm} 
\newcommand{\edgewidth}{0.035cm} 
\newcommand{\boundarylabel}{8pt} 
\newcommand{\arrowwidth}{0.05cm} 
\ytableausetup{boxsize=4.5pt}
\usetikzlibrary{decorations.markings}
\tikzset{sarrow/.style={decoration={markings, 
                               mark=at position .7 with {\arrow[scale=1.3, red]{latex}}},
                               postaction={decorate}}}                               
\tikzset{ssarrow/.style={decoration={markings, 
                               mark=at position .8 with {\arrow[scale=1.3, red]{latex}}},
                               postaction={decorate}}}

\foreach \n/\a in {1/60, 2/0, 3/300, 4/240, 5/180, 6/120} {
\coordinate (B\n) at (\a:\boundaryrad); 
\coordinate (B\n+) at (\a:\boundaryrad+\boundarylabel); };

\node at (0,0) {
%
%
\begin{tikzpicture}
\draw (0,0) circle(\boundaryrad) [gray, line width=\edgewidth];
\foreach \n in {1,...,6} { \draw (B\n+) node {\small $\n$}; }
\foreach \n/\a/\r in {7/60/0.65, 8/0/0.7, 9/330/0.58, 10/300/0.65, 11/240/0.65, 12/210/0.58, 13/180/0.7, 14/120/0.65, 15/0/0.15, 16/180/0.15} {
\coordinate (B\n) at (\a:\r*\boundaryrad); }; 
\foreach \s/\t in {1/7, 2/8, 3/10, 4/11, 5/13, 6/14, 7/8, 8/9, 9/10, 10/11, 11/12, 12/13, 13/14, 14/7, 7/15, 9/15, 12/16, 14/16, 15/16} {
\draw[line width=\edgewidth] (B\s)--(B\t); 
};
\foreach \x in {8,10,12,14,15} {
\filldraw [fill=black, line width=\edgewidth] (B\x) circle [radius=\noderad-0.01cm] ;}; 
\foreach \x in {7,9,11,13,16} {
\filldraw [fill=white, line width=\edgewidth] (B\x) circle [radius=\noderad] ; };

\foreach \n/\a/\r in {123/270/0.8, 234/210/0.82, 345/150/0.8, 456/90/0.8, 156/30/0.8, 126/330/0.82} {
\coordinate (V\n) at (\a:\r*\boundaryrad); }; 
\node[blue] at (V123) {\ydiagram{3,3,3}}; 
\node[blue] at (V234) {\ydiagram{2,2,2}}; 
\node[blue] at (V345) {\ydiagram{1,1,1}};
\node[blue] at (V456) {$\varnothing$}; 
\node[blue] at (V156) {\ydiagram{3}};
\node[blue] at (V126) {\ydiagram{3,3}}; 
\foreach \n/\a/\r in {356/90/0.3, 236/270/0.3, 256/10/0.42, 346/170/0.42} {
\coordinate (V\n) at (\a:\r*\boundaryrad); }; 
\node[blue] at (V356) {\ydiagram{1}}; 
\node[blue] at (V236) {\ydiagram{2,2}}; 
\node[blue] at (V256) {\ydiagram{2}}; 
\node[blue] at (V346) {\ydiagram{1,1}}; 
\end{tikzpicture}};

\node at (7,0) {
%
%
\begin{tikzpicture}                               
\draw (0,0) circle(\boundaryrad) [gray, line width=\edgewidth];
\draw (0,0) circle(\boundaryrad) [gray, line width=\edgewidth, fill=white];
\foreach \n in {1,...,6} { \draw (B\n+) node {\small $\n$}; }
\foreach \n/\a/\r in {7/60/0.65, 8/0/0.7, 9/330/0.58, 10/300/0.65, 11/240/0.65, 12/210/0.58, 13/180/0.7, 14/120/0.65, 15/0/0.15, 16/180/0.15} {
\coordinate (B\n) at (\a:\r*\boundaryrad); }; 
\foreach \s/\t in {1/7, 2/8, 3/10, 7/8, 7/14, 7/15, 8/9, 9/10, 10/11, 11/12, 12/13, 13/14, 9/15, 15/16, 16/12, 16/14} {
\draw[sarrow, line width=\edgewidth] (B\s)--(B\t);}; 
\foreach \s/\t in {11/4, 13/5, 14/6} {
\draw[ssarrow, line width=\edgewidth] (B\s)--(B\t);}; 
\foreach \x in {8,10,12,14,15} {
\filldraw [fill=black, line width=\edgewidth] (B\x) circle [radius=\noderad-0.01cm] ;}; 
\foreach \x in {7,9,11,13,16} {
\filldraw [fill=white, line width=\edgewidth] (B\x) circle [radius=\noderad] ; };

\foreach \n/\a/\r in {123/270/0.8, 234/210/0.82, 345/150/0.8, 456/90/0.8, 156/30/0.8, 126/330/0.82} {
\coordinate (V\n) at (\a:\r*\boundaryrad); }; 
\node[blue] at (V123) {\ydiagram{3,3,3}}; 
\node[blue] at (V234) {\ydiagram{2,2,2}}; 
\node[blue] at (V345) {\ydiagram{1,1,1}};
\node[blue] at (V456) {$\varnothing$}; 
\node[blue] at (V156) {\ydiagram{3}};
\node[blue] at (V126) {\ydiagram{3,3}}; 
\foreach \n/\a/\r in {356/90/0.3, 236/270/0.3, 256/10/0.42, 346/170/0.42} {
\coordinate (V\n) at (\a:\r*\boundaryrad); }; 
\node[blue] at (V356) {\ydiagram{1}}; 
\node[blue] at (V236) {\ydiagram{2,2}}; 
\node[blue] at (V256) {\ydiagram{2}}; 
\node[blue] at (V346) {\ydiagram{1,1}}; 
\end{tikzpicture}};
\end{tikzpicture}}
\end{center}
\caption{The rectangle plabic graph $G=G^{\rm rec}_{3,6}$ (left) and the acyclic perfect orientation (right)}
\label{ex_plabic36}
\end{figure}
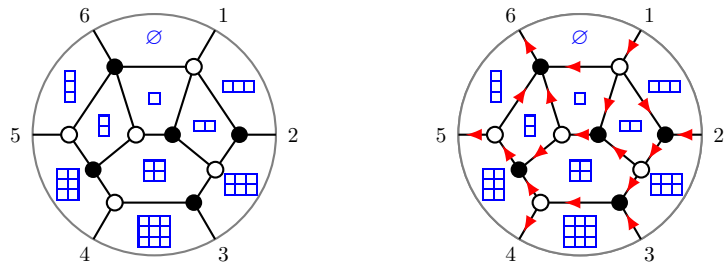

\medskip

Using this perfect orientation, we can compute the flow polynomial $P^G_J$ for any $3$-subset $J\in{[6] \choose 3}$, 
and it determines the valuations as in Table~\ref{table_rec_Gr(3,6)}.

\ytableausetup{boxsize=3.5pt}
\begin{table}[H]
\begin{center}
\scalebox{0.9}{
{\renewcommand\arraystretch{1.1}\tabcolsep= 10pt
\begin{tabular}{c|ccccccccc}
&\raise0.5ex\hbox{\ydiagram{1}}&\raise0.5ex\hbox{\ydiagram{2}}&\raise1ex\hbox{\ydiagram{1,1}}
&\raise1ex\hbox{\ydiagram{2,2}}&\raise0.5ex\hbox{\ydiagram{3}}&\raise1.5ex\hbox{\ydiagram{1,1,1}}&\raise1ex\hbox{\ydiagram{3,3}}&
\raise1.5ex\hbox{\ydiagram{2,2,2}}&\raise1.5ex\hbox{\ydiagram{3,3,3}} \rule[0mm]{0mm}{5mm} \\ \hline
$\{1,2,3\}$&0&0&0&0&0&0&0&0&0 \\ \hline
$\{1,2,4\}$&0&0&0&0&0&0&0&0&1 \\ \hline
$\{1,2,5\}$&0&0&0&0&0&0&0&1&1 \\ \hline
$\{1,2,6\}$&0&0&0&0&0&1&0&1&1 \\ \hline
$\{1,3,4\}$&0&0&0&0&0&0&1&0&1 \\ \hline
$\{1,3,5\}$&0&0&0&0&0&0&1&1&1 \\ \hline
$\{1,3,6\}$&0&0&0&0&0&1&1&1&1 \\ \hline
$\{1,4,5\}$&0&0&0&1&0&0&1&1&2 \\ \hline
$\{1,4,6\}$&0&0&0&1&0&1&1&1&2 \\ \hline
$\{1,5,6\}$&0&0&1&1&0&1&1&2&2 \\ \hline
$\{2,3,4\}$&0&0&0&0&1&0&1&0&1 \\ \hline
$\{2,3,5\}$&0&0&0&0&1&0&1&1&1 \\ \hline
$\{2,3,6\}$&0&0&0&0&1&1&1&1&1 \\ \hline
$\{2,4,5\}$&0&0&0&1&1&0&1&1&2 \\ \hline
$\{2,4,6\}$&0&0&0&1&1&1&1&1&2 \\ \hline
$\{2,5,6\}$&0&0&1&1&1&1&1&2&2 \\ \hline
$\{3,4,5\}$&0&1&0&1&1&0&2&1&2 \\ \hline
$\{3,4,6\}$&0&1&0&1&1&1&2&1&2 \\ \hline
$\{3,5,6\}$&0&1&1&1&1&1&2&2&2 \\ \hline
$\{4,5,6\}$&1&1&1&2&1&1&2&2&3 
\end{tabular}}}
\end{center}
\caption{The valuation $\val_G(\overline{p}_J)$ for any Pl\"{u}cker coordinate $\overline{p}_J$ of $\Gr(3,6)$ (we write only $J$ instead of $\overline{p}_J$ in the columns).}
\label{table_rec_Gr(3,6)}
\end{table}

By Proposition~\ref{prop_unimodular_GT}, $\Delta_G$ is unimodular equivalent to the Gelfand--Tsetlin polytope $\calO(\Pi_{3,6})$, which is a lattice polytope. 
Thus $\Delta_G$ is the convex hull of the row vectors given in this table by Theorem~\ref{thm_NO=C}, 
and $\Delta_G$ corresponds to the isomorphism class EEFF1. 
We remark that by Proposition~\ref{prop_unimodular_GT}, the unimodular transformation $f$ satisfying $f(\Delta_G)=\calO(\Pi_{3,6})$ is given as follows: 
\ytableausetup{boxsize=3.5pt}
\[
\nu^J_{\ydiagram{3,3,3}}  \mapsto \nu^J_{\ydiagram{3,3,3}} -\nu^J_{\ydiagram{2,2}}, \hspace{15pt}
\nu^J_{\ydiagram{2,2,2}}  \mapsto \nu^J_{\ydiagram{2,2,2}} -\nu^J_{\ydiagram{1,1}}, \hspace{15pt}
\nu^J_{\ydiagram{3,3}}  \mapsto \nu^J_{\ydiagram{3,3}} -\nu^J_{\ydiagram{2}}, \hspace{15pt}
\nu^J_{\ydiagram{2,2}}  \mapsto \nu^J_{\ydiagram{2,2}} -\nu^J_{\ydiagram{1}},
\]
\[
\nu^J_\lambda  \mapsto \nu^J_\lambda 
\quad \text{for } \lambda\in\calP_G{\setminus}\big\{\raise1ex\hbox{\ydiagram{3,3,3}}, \raise1ex\hbox{\ydiagram{2,2,2}}, \raise0.6ex\hbox{\ydiagram{3,3}}, \raise0.6ex\hbox{\ydiagram{2,2}}\big\}, 
\]
for any $J\in{[6] \choose 3}$, 
where $\nu_\lambda^J$ stands for the component of $\val_G(\overline{p}_J)$ corresponding to the Young diagram $\lambda$. 

\medskip

Next, as we noted in Proposition~\ref{prop_move_equiv}, all plabic graphs of type $\pi_{3,6}$ are move-equivalent. 
Thus the associated Newton--Okounkov polytopes are connected by the tropicalized cluster mutations (see Theorem~\ref{thm_trop_map_NO}), 
and hence by the composition of unimodular transformations and combinatorial mutations (see Proposition~\ref{prop_tropical=combmutation}). 
In Figures \ref{ex_graph_Gr36_1} and \ref{ex_graph_Gr36_2} below, we give the \emph{mutation graph} which shows 
how to connect the plabic graphs of type $\pi_{3,6}$ and the associated Newton--Okounkov polytopes. 
Precisely, in the figures, a vertex corresponds to a plabic graph $G$, 
and plabic graphs $G, G^\prime$ connected by an edge are transformed each other by a square move, 
in which case the associated Newton--Okounkov polytope $\Delta_G$ and $\Delta_{G^\prime}$ are 
related by the piecewise-linear map $\Psi_{G,G^\prime}$ or $\varphi_{w, F_{\lambda_i}}\circ\epsilon_{\lambda_i}$ as in Proposition~\ref{prop_tropical=combmutation}. 
We also add the labels of isomorphism classes of the Newton--Okounkov polytopes in these figures. 

\ytableausetup{boxsize=3.5pt}
As these figures show, if we apply the square move to an internal face of $G^{\rec}_{3,6}$ labeled by \ydiagram{1}, \raise0.5ex\hbox{\ydiagram{1,1}} 
or \ydiagram{2} , then we again have a plabic graph of type $\pi_{3,6}$. 
Whereas, the internal face of $G^{\rec}_{3,6}$ labeled by \raise0.5ex\hbox{\ydiagram{2,2}} is not a square, thus we can not apply the square move. 
However, we can consider the quiver mutation of $Q(G^{\rec}_{3,6})$ at the vertex corresponding to \raise0.5ex\hbox{\ydiagram{2,2}}, 
and the tropicalized cluster mutation at \raise0.5ex\hbox{\ydiagram{2,2}} as in \cite[Definition~11.8]{RW}. 
As we mentioned in Remark~\ref{rem_tropicalized_mutation}, Theorem~\ref{thm_trop_map_NO} also holds even if we consider this situation. 
More precisely, we modify $v_{\lambda_i}^\prime$ in \eqref{eq_tropical_mutation1} to 
\[
v_{\lambda_i}^\prime=-v_{\lambda_i}+\min\bigg\{\sum_{j\rightarrow i} v_{\lambda_j}, \,\, \sum_{i\rightarrow j} v_{\lambda_j}\bigg\}. 
\]
Thus, in our situation we consider 
\[
v_{\ydiagram{2,2}}^\prime=-v_{\ydiagram{2,2}}+\min\Big\{v_{\ydiagram{2}}+v_{\ydiagram{1,1}}+v_{\ydiagram{3,3,3}}, \,\, 
v_{\ydiagram{1}}+v_{\ydiagram{2,2,2}}+v_{\ydiagram{3,3}}\Big\}.
\]
Under this piecewise-linear map, the column vector corresponding to \raise0.5ex\hbox{\ydiagram{2,2}} in Table~\ref{table_rec_Gr(3,6)} 
will be send to the transpose of 
\[
(0,0,1,1,1,1,1,1,1,2,1,1,1,1,1,2,2,2,3,3). 
\]
The Newton--Okounkov polytope defined by this new valuation has properties similar to the one arising from a plabic graph of type $\pi_{3,6}$. 
In particular, it gives a toric degeneration of $\XX_{3,6}$, see \cite[Corollary~17.11]{RW}. 

%
%

\begin{figure}[H]
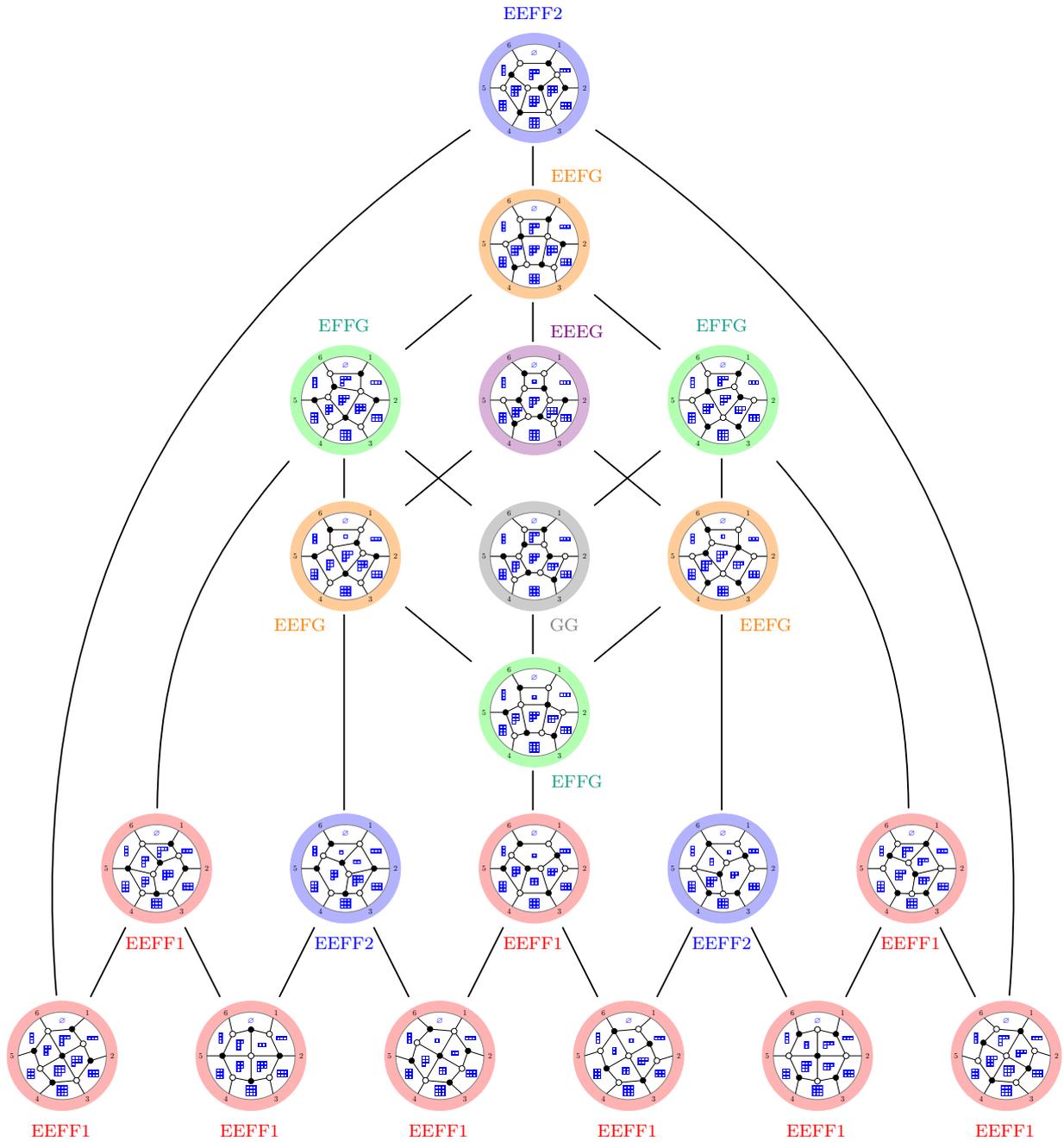

\begin{center}
\scalebox{0.96}{
}};

\foreach \s/\t in {5/13,13/3,3/15,15/7,7/4,4/8,8/17,17/10,10/11,11/12,15/29,4/22,17/31,
22/29,22/1,22/31,29/21,29/28,1/21,1/25,31/28,31/25,21/30,28/30,25/30,30/16} {
\draw[shorten >=-0.08cm, shorten <=-0.08cm, auto, line width=\edgewidth-0.01cm] (P\s)--(P\t); 
};

\foreach \s/\t in {5/16,16/12} {
\draw[shorten >=-0.05cm, shorten <=-0.05cm, line width=\edgewidth-0.01cm] (P\s)  [bend left] to (P\t); 
};

\foreach \s/\t in {13/21,25/11} {
\draw[shorten >=-0.05cm, shorten <=-0.05cm, line width=\edgewidth-0.01cm] (P\s)  [bend left=20] to (P\t); 
};
\end{tikzpicture}}
\end{center}
\caption{A part of the mutation graph for plabic graphs of type $\pi_{3,6}$ (continued). 
The bottom is identified with the top of Figure~\ref{ex_graph_Gr36_2}}
\label{ex_graph_Gr36_1}
\end{figure}

%
%

\begin{figure}[H]
\begin{center}
\scalebox{0.96}{
}};

\foreach \s/\t in {5/6,6/3,3/20,20/7,7/9,9/8,8/18,18/10,10/14,14/12,20/32,9/23,23/32,23/2,23/33,
18/33,32/26,32/27,2/26,2/24,33/27,33/24,26/34,27/34,24/34,34/19} {
\draw[shorten >=-0.08cm, shorten <=-0.08cm, auto, line width=\edgewidth-0.01cm] (P\s)--(P\t); 
};
\foreach \s/\t in {19/5,12/19} {
\draw[shorten >=-0.05cm, shorten <=-0.05cm, line width=\edgewidth-0.01cm] (P\s)  [bend left] to (P\t); 
};
\foreach \s/\t in {26/6,14/24} {
\draw[shorten >=-0.05cm, shorten <=-0.05cm, line width=\edgewidth-0.01cm] (P\s)  [bend left=20] to (P\t); 
};

\end{tikzpicture}}
\end{center}
\caption{A part of the mutation graph for plabic graphs of type $\pi_{3,6}$. The top is identified with the bottom of Figure~\ref{ex_graph_Gr36_1}}
\label{ex_graph_Gr36_2}
\end{figure}


Meanwhile, the top plabic graph in Figure~\ref{ex_graph_Gr36_1} is the dual $(G^{\rm rec}_{3,6})^\vee$ of $G^{\rm rec}_{3,6}$. 
The Newton--Okounkov polytopes $\Delta_{G^\vee}$ associated to $G^\vee=(G^{\rm rec}_{3,6})^\vee$ is unimodularly equivalent
to the FFLV polytope $\calC(\Pi_{3,6})$, which is a lattice polytope, by Theorem~\ref{thm_unimodular_FFLV2}. 
Thus, by Theorem~\ref{thm_NO=C}, $\Delta_{G^\vee}$ is the convex hull of the valuations $\val_{G^\vee}(\overline{p}_J)$ for all $J\in{[6] \choose 3}$. 
Here, Table~\ref{table_dualrec_Gr(3,6)} shows the valuations $\val_{G^\vee}(\overline{p}_J)$, 
which can be obtained by considering the acyclic orientation of $G^\vee$ 
whose source set is $\{1,2,3\}$, or by applying the tropicalized cluster mutations to Table~\ref{table_rec_Gr(3,6)} 
along the edges connecting $G^{\rm rec}_{3,6}$ and $(G^{\rm rec}_{3,6})^\vee$ in Figure~\ref{ex_graph_Gr36_1} and \ref{ex_graph_Gr36_2}. 
We also see that $\Delta_{G^\vee}$ corresponds to the isomorphism class EEFF2. 

\ytableausetup{boxsize=3.5pt}
\begin{table}[H]
\begin{center}
\scalebox{0.9}{
{\renewcommand\arraystretch{1.1}\tabcolsep= 10pt
\begin{tabular}{c|ccccccccc}
&\raise1.5ex\hbox{\ydiagram{3,3,2}}&\raise1.5ex\hbox{\ydiagram{3,3,1}}&\raise1.5ex\hbox{\ydiagram{3,2,2}}
&\raise1.5ex\hbox{\ydiagram{3,1,1}}&\raise0.5ex\hbox{\ydiagram{3}}&\raise1.5ex\hbox{\ydiagram{1,1,1}}&\raise1ex\hbox{\ydiagram{3,3}}&
\raise1.5ex\hbox{\ydiagram{2,2,2}}&\raise1.5ex\hbox{\ydiagram{3,3,3}} \rule[0mm]{0mm}{5mm} \\ \hline
$\{1,2,3\}$&0&0&0&0&0&0&0&0&0 \\ \hline
$\{1,2,4\}$&0&0&0&0&0&0&0&0&1 \\ \hline
$\{1,2,5\}$&1&0&1&0&0&0&0&1&1 \\ \hline
$\{1,2,6\}$&1&1&1&1&0&1&0&1&1 \\ \hline
$\{1,3,4\}$&1&1&0&0&0&0&1&0&1 \\ \hline
$\{1,3,5\}$&1&1&1&0&0&0&1&1&1 \\ \hline
$\{1,3,6\}$&1&1&1&1&0&1&1&1&1 \\ \hline
$\{1,4,5\}$&1&1&1&0&0&0&1&1&2 \\ \hline
$\{1,4,6\}$&1&1&1&1&0&1&1&1&2 \\ \hline
$\{1,5,6\}$&2&1&2&1&0&1&1&2&2 \\ \hline
$\{2,3,4\}$&1&1&1&1&1&0&1&0&1 \\ \hline
$\{2,3,5\}$&1&1&1&1&1&0&1&1&1 \\ \hline
$\{2,3,6\}$&1&1&1&1&1&1&1&1&1 \\ \hline
$\{2,4,5\}$&1&1&1&1&1&0&1&1&2 \\ \hline
$\{2,4,6\}$&1&1&1&1&1&1&1&1&2 \\ \hline
$\{2,5,6\}$&2&1&2&1&1&1&1&2&2 \\ \hline
$\{3,4,5\}$&2&2&1&1&1&0&2&1&2 \\ \hline
$\{3,4,6\}$&2&2&1&1&1&1&2&1&2 \\ \hline
$\{3,5,6\}$&2&2&2&1&1&1&2&2&2 \\ \hline
$\{4,5,6\}$&2&2&2&1&1&1&2&2&3 
\end{tabular}}}
\end{center}
\caption{The valuation $\val_{G^\vee}(\overline{p}_J)$ for any Pl\"{u}cker coordinate $\overline{p}_J$ of $\Gr(3,6)$.}
\label{table_dualrec_Gr(3,6)}
\end{table}

\ytableausetup{boxsize=3.5pt}
Applying the unimodular transformation $g$ in Theorem~\ref{thm_unimodular_FFLV2}, which is described as 
\begin{align*}
&w_{\varnothing}=v_{\ydiagram{3,3,3}}\mapsto 0, \hspace{15pt}
w_{1 \times 1}=v_{\ydiagram{3,3,2}}\mapsto v_{\ydiagram{3,3,3}}-v_{\ydiagram{3,3,2}}, \hspace{15pt}
w_{1 \times 2}=v_{\ydiagram{3,3,1}}\mapsto v_{\ydiagram{3,3,2}}-v_{\ydiagram{3,3,1}}, \\
&w_{1 \times 3}=v_{\ydiagram{3,3}}\mapsto v_{\ydiagram{3,3,1}}-v_{\ydiagram{3,3}}, \hspace{15pt}
w_{2 \times 1}=v_{\ydiagram{3,2,2}}\mapsto v_{\ydiagram{3,3,2}}-v_{\ydiagram{3,2,2}}, \\
&w_{2 \times 2}=v_{\ydiagram{3,1,1}}\mapsto v_{\ydiagram{3,3,1}}+v_{\ydiagram{3,2,2}}-v_{\ydiagram{3,1,1}}-v_{\ydiagram{3,3,2}},  \hspace{15pt}
w_{2 \times 3}=v_{\ydiagram{3}}\mapsto v_{\ydiagram{3,3}}+v_{\ydiagram{3,1,1}}-v_{\ydiagram{3}}-v_{\ydiagram{3,3,1}}, \\
&w_{3 \times 1}=v_{\ydiagram{2,2,2}}\mapsto v_{\ydiagram{3,2,2}}-v_{\ydiagram{2,2,2}},  \hspace{15pt}
w_{3 \times 2}=v_{\ydiagram{1,1,1}}\mapsto v_{\ydiagram{3,1,1}}+v_{\ydiagram{2,2,2}}-v_{\ydiagram{1,1,1}}-v_{\ydiagram{3,2,2}}, \\
&w_{3 \times 3}=v_{\varnothing}\mapsto v_{\ydiagram{3}}+v_{\ydiagram{1,1,1}}-v_{\varnothing}-v_{\ydiagram{3,1,1}}, 
\end{align*}
to these vectors, we have $g(\Delta_{G^\vee})=\calC(\Pi_{3,6})$. 

\subsection*{Acknowledgement} 
The authors would like to thank the organizers of the online workshop ``The McKay correspondence, Mutation and related topics" 
for giving us opportunities to talk concerning the combinatorial mutations, and to write this article for the proceeding. 
The authors would also like to thank Lara Bossinger and Xin Fang for providing useful comments on an earlier draft of the manuscript, 
and thank the anonymous referee for valuable comments and suggestions. 

The first author is supported by JSPS Grant-in-Aid for Scientific Research (C) 20K03513. 
The second author is supported by JSPS Grant-in-Aid for Young Scientists (B) 17K14159, and JSPS Grant-in-Aid for Early-Career Scientists 20K14279. 


\end{document}